\documentclass[12pt,draft]{amsart}
\usepackage{amsmath,amssymb,amsthm,bm}
\usepackage{graphicx,enumerate}
\usepackage{layout}

\theoremstyle{plain}
\newtheorem{theorem}{Theorem}[section]
\newtheorem{lemma}[theorem]{Lemma}

\newtheorem{corollary}[theorem]{Corollary}

\theoremstyle{definition}

\newtheorem{example}[theorem]{Example}
\newtheorem{remark}[theorem]{Remark}
\newtheorem*{acknowledgments}{Acknowledgments}

\theoremstyle{remark}

\newcommand{\si}{\sigma}\newcommand{\ta}{\tau}\newcommand{\la}{\lambda}
\newcommand{\al}{\alpha}\newcommand{\be}{\beta}
\newcommand{\om}{\omega}
\newcommand{\tu}{\widetilde{u}}
\newcommand{\tv}{\widetilde{v}}
\newcommand{\ts}{\widetilde{s_1}}
\newcommand{\bZ}{\mathbb{Z}}\newcommand{\bQ}{\mathbb{Q}}
\newcommand{\bC}{\mathbb{C}}
\newcommand{\ep}{\varepsilon} 
\newcommand{\ca}{\sigma_{3{\rm A}}}
\newcommand{\cb}{\sigma_{3{\rm B}}}
\newcommand{\caa}{\sigma_{4{\rm A}}}
\newcommand{\cbb}{\sigma_{4{\rm B}}}

\title[three-dimensional monomial group actions]
{Rationality problem of three-dimensional monomial group actions}

\author[A. Hoshi]{Akinari Hoshi}
\address{Department of Mathematics, Rikkyo University, 
3--34--1 Nishi-Ikebukuro Toshima-ku, Tokyo, 171--8501, Japan}

\author[H. Kitayama]{Hidetaka Kitayama}
\address{Department of Mathematics, Graduate school of Science, Osaka University, 
Machikaneyama 1-1, Toyonaka, Osaka, 560--0043, Japan}

\author[A. Yamasaki]{Aiichi Yamasaki}
\address{Department of Mathematics, Graduate school of Science, Kyoto University, 
Kyoto, 606--8502, Japan}


\begin{document}
\begin{abstract}
Let $K$ be a field of characteristic not two 
and $K(x,y,z)$ the rational function field over $K$ with three variables $x,y,z$. 
Let $G$ be a finite group acting on $K(x,y,z)$ by monomial $K$-automorphisms. 
We consider the rationality problem of the fixed field $K(x,y,z)^G$ 
under the action of $G$, namely whether $K(x,y,z)^G$ is rational 
(that is, purely transcendental) over $K$ or not. 
We may assume that $G$ is a subgroup of $\mathrm{GL}(3,\bZ)$ and 
the problem is determined up to conjugacy in $\mathrm{GL}(3,\bZ)$. 
There are $73$ conjugacy classes of $G$ in $\mathrm{GL}(3,\bZ)$. 
By results of Endo-Miyata, Voskresenski\u\i, Lenstra, Saltman, Hajja, Kang and Yamasaki, 
$8$ conjugacy classes of $2$-groups in $\mathrm{GL}(3,\bZ)$ have negative answers to 
the problem under certain monomial actions over some base field $K$, and the necessary 
and sufficient condition for the rationality of $K(x,y,z)^G$ over $K$ is given. 
In this paper, we show that the fixed field $K(x,y,z)^G$ under monomial action of $G$ 
is rational over $K$ except for possibly negative $8$ cases of $2$-groups and 
unknown one case of the alternating group of degree four. 
Moreover we give explicit transcendental bases of the fixed fields over $K$. 
For the unknown case, we obtain an affirmative solution to the problem under some 
conditions. 
In particular, we show that if $K$ is quadratically closed field then $K(x,y,z)^G$ 
is rational over $K$. 
We also give an application of the result to $4$-dimensional linear Noether's problem. 
\end{abstract}

\maketitle

\section{Introduction}\label{seintro}

Let $K$ be a field of  char $K\neq 2$ and $K(x_1,\ldots,x_n)$ the rational function field 
over $K$ with $n$ variables $x_1,\ldots,x_n$. 
A $K$-automorphism $\sigma$ of $K(x_1,\ldots,x_n)$ is called {\it monomial} if 
\begin{align}
\si(x_j)=c_j(\si)\prod_{i=1}^n x_i^{a_{i,j}},\quad 1\leq j\leq n\label{mact}
\end{align}
where $[a_{i,j}]_{1\leq i,j\leq n}\in {\rm GL}(n,\bZ)$ and 
$c_j(\si)\in K^\times:=K\setminus\{0\}$. 
If $c_j(\si)=1$ for any $1\leq j\leq n$ then $\sigma$ is called {\it purely monomial}. 
A group action on $K(x_1,\ldots,x_n)$ by monomial $K$-automorphisms is called monomial. 

Let $G$ be a finite group acting on $K(x_1,\ldots,x_n)$ by monomial $K$-automorphisms. 
The aim of this paper is to investigate the rationality problem of $K(x_1,x_2,x_3)$ 
under $3$-dimensional monomial actions. 
Namely we consider whether the fixed field $K(x_1,x_2,x_3)^G$ under monomial action 
of $G$ is rational (that is, purely transcendental) over $K$ or not. 
We write $K(x,y,z):=K(x_1,x_2,x_3)$ for simplicity. 
For $n=2,3$, the following theorems are known: 
\begin{theorem}[Hajja \cite{Haj83}, \cite{Haj87}]\label{thHaj}
Let $K$ be any field and $G$ a finite group acting on $K(x,y)$ by monomial $K$-automorphisms. 
Then the fixed field $K(x,y)^G$ under monomial action of $G$ is rational over $K$. 
\end{theorem}
\begin{theorem}[Hajja-Kang \cite{HK92}, \cite{HK94}, Hoshi-Rikuna \cite{HR08}]\label{thHKHR}
Let $K$ be any field and $G$ a finite group acting on $K(x,y,z)$ by purely 
monomial $K$-automorphisms. 
Then the fixed field $K(x,y,z)^G$ under the purely monomial action of $G$ is rational over $K$. 
\end{theorem}

However, for $3$-dimensional monomial $\bQ$-automorphism 
$\sigma : x\mapsto y\mapsto z\mapsto -1/(xyz)$, 
the fixed field $\bQ(x,y,z)^{\langle\sigma\rangle}$ is not rational over $\bQ$ 
(see \cite[page 249]{Haj83}). 
This phenomenon comes from the negative answer to Noether's problem 
for the cyclic group of order $8$ (cf. Endo-Miyata \cite{EM73}, 
Voskresenski\u\i\ \cite{Vos73}, Lenstra \cite{Len74}, Saltman \cite{Sal82}). 

Let $G$ be a finite group acting on $K(x,y,z)$ by monomial $K$-automorphisms. 
There is a group homomorphism 
$p : G\rightarrow {\rm GL}(3,\bZ)$, $\sigma\mapsto [a_{i,j}]_{1\leq i,j\leq 3}$ 
where $[a_{i,j}]_{1\leq i,j\leq 3}$ is given as in (\ref{mact}). 
Then there is a normal subgroup $H$ of $G$ such that (i) $K(x,y,z)^H=K(t_1,t_2,t_3)$, (ii) 
the group $G/H$ acts on $K(t_1,t_2,t_3)$ by monomial $K$-automorphisms 
and (iii) the associated group homomorphism $p : G/H\rightarrow \mathrm{GL}(3,\bZ)$ is injective. 
Hence we may assume that $p : G\rightarrow {\rm GL}(3,\bZ)$ is injective\footnote{The referee 
informed the authors of the importance of this assumption. 
More precisely, see \cite[page 46]{Haj87} and \cite[Lemma 2.8]{KP10}.}. 
Therefore we treat the case where $G$ is a subgroup of ${\rm GL}(3,\bZ)$. 

Let $G$ be a finite subgroup of ${\rm GL}(3,\bZ)$. 
The rationality problem is determined up to conjugacy in $\mathrm{GL}(3,\bZ)$ 
since a conjugate of $G$ 
corresponds to some base change of $K(x,y,z)$. 
There are 73 conjugacy classes $[G]$ of finite subgroups $G$ in $\mathrm{GL}(3,\bZ)$ 
which are classified into $7$ crystal systems as in Table $1$ of \cite{BBNWZ78}. 
We denote by $[G_{i,j,k}]$, $1\leq i\leq 7$, the finite group of the $k$-th 
$\bZ$-class of the $j$-th $\bQ$-class of the $i$-th crystal system of dimension $3$, 
and we take a representative $G_{i,j,k}\subset\mathrm{GL}(3,\bZ)$ of each $\bZ$-class 
as in Table $1$ of \cite{BBNWZ78} (cf. also Section \ref{senot} in the present paper). 

Let $I_n\in\mathrm{GL}(n,\bZ)$ be the $n\times n$ identity matrix. 
In the case of $G_{1,2,1}=\langle-I_3\rangle$, 
Saltman \cite{Sal00} showed that $K(x,y,z)^{G_{1,2,1}}$ is not rational over $K$ 
for some field $K$ and choice of $c_j(\sigma), (\sigma\in G_{1,2,1})$ where $c_j(\sigma)$ 
is given as in (\ref{mact}): 
\begin{theorem}[Saltman {\cite[Theorem 0.1]{Sal00}}]\label{thSalt}
Let $K$ be a field of {\rm char} $K\neq 2$ and $G_{1,2,1}=\langle-I_3\rangle$ act on 
$K(x,y,z)$ by 
\begin{align*}
-I_3 : x\ \mapsto\ \frac{a_1}{x},\ y\ \mapsto\ \frac{a_2}{y},\ z\ \mapsto\ \frac{a_3}{z}, 
\quad a_i\in K^\times,\ 1\leq i\leq 3. 
\end{align*}
If $[K(\sqrt{a_1},\sqrt{a_2},\sqrt{a_3}):K]=8$ then the fixed field 
$K(x,y,z)^{G_{1,2,1}}$ is not retract rational over $K$, 
and hence not rational over $K$. 
\end{theorem}

The following generalization of Saltman's result was given by Kang \cite{Kan05}. 

\begin{theorem}[Kang {\cite[Theorem 4.4]{Kan05}}]\label{thKan}
Let $K$ be a field of {\rm char} $K\neq 2$ and $n\geq 3$. If $-I_n\in \mathrm{GL}(n,\bZ)$ 
acts on $K(x_1,\ldots,x_n)$ by \,$-I_n : x_i\mapsto a_i/x_i$, $a_i\in K^\times$, then 
the fixed field $K(x_1,\ldots,x_n)^{\langle -I_n\rangle}$ is not rational over $K$ 
if and only if  $[K(\sqrt{a_1},\ldots ,\sqrt{a_n}):K]\geq 8$. 
If $K(x_1,\ldots,x_n)^{\langle -I_n\rangle}$ is not rational over $K$ then 
it is not retract rational over $K$. 
\end{theorem}

Kang \cite{Kan04} also established the necessary and sufficient condition 
for the rationality of $K(x,y,z)^{G_{4,2,2}}$ over any field $K$: 
\begin{theorem}[Kang {\cite[Theorem 1.8]{Kan04}}]\label{thKan04b}
Let $K$ be any field and $G_{4,2,2}=\langle -\cbb\rangle$ act on $K(x,y,z)$ by 
\begin{align*}
-\cbb : x\ \mapsto\ y\ \mapsto\ z\ \mapsto\ \frac{c}{xyz}\ \mapsto\ x,\quad c\in K^\times. 
\end{align*}
Then $K(x,y,z)^{G_{4,2,2}}$ is rational over $K$ if and only if  at least one of the following 
conditions is satisfied, {\rm (i)} {\rm char} $K=2;$ {\rm (ii)} $c\in K^2;$ 
{\rm (iii)} $-4c\in K^4;$ {\rm (iv)} $-1\in K^2$. 
If $K(x,y,z)^{\langle -\cbb\rangle}$ is not rational over $K$ then 
it is not retract rational over $K$. 
\end{theorem}

We put 
\begin{align}
\mathcal{N}:=\bigl\{[G_{1,2,1}],[G_{2,3,1}],[G_{3,1,1}],
[G_{3,3,1}],[G_{4,2,1}],[G_{4,2,2}],[G_{4,3,1}],[G_{4,4,1}]\bigr\}.\label{listN}
\end{align}

Yamasaki \cite{Yam} showed that if $G\in\mathcal{N}$ then the fixed field $K(x,y,z)^G$ 
under monomial action is not rational over $K$ for some field $K$ and choice of coefficients 
$c_j(\sigma)$, $(\sigma\in G)$ where $c_j(\sigma)$ is given as in (\ref{mact}). 
For $G\in\mathcal{N}$, moreover, the necessary and sufficient condition 
for the rationality of $K(x,y,z)^G$ over a field $K$ of char $K\neq 2$ 
was given in \cite{Yam} by using ideas of Lenstra's paper \cite{Len74} and 
Saltman's paper \cite{Sal00} (cf. Theorem \ref{thSalt}).

In the present paper, we treat the cases $G\not\in\mathcal{N}$ 
and the main result is the following theorem:
\begin{theorem}\label{thmain}
Let $K$ be a field of {\rm char} $K\neq 2$ and $G$ a finite subgroup of $\mathrm{GL}(3,\bZ)$ 
which is not in $\mathcal{N}$. 
Then the fixed field $K(x,y,z)^G$ under monomial action of $G$ is rational 
over $K$ except for the case $G\in [G_{7,1,1}]$. 
\end{theorem}
Moreover a method of this paper enables us to construct transcendental bases of the 
fixed fields $K(x,y,z)^G$ over $K$ explicitly, and also gives another proof to Theorem 
\ref{thHKHR} which is the case of purely monomial actions. 

In the exceptional case $G_{7,1,1}=\langle\ta_1,\la_1,\cb\rangle\cong\mathcal{A}_4$, 
the problem may be reduced to the following cases: 
\begin{align}
\ta_1 &: x\ \mapsto\ \frac{a}{x},\ y\ \mapsto\ \frac{\ep_1 a}{y},\ z\ \mapsto\ \ep_1 z,&
\la_1 &: x\ \mapsto\ \frac{\ep_1 a}{x},\ y\ \mapsto\ \ep_1 y,\ z\ \mapsto\ \frac{a}{z},
\label{introact71}\\
\cb &: x\ \mapsto\  y,\ y\ \mapsto\  z,\ z\ \mapsto\ x\nonumber
\end{align}
where $a\in K^\times$, $\ep_1=\pm 1$. 

For $G_{7,1,1}$, we prove the following theorem in Section \ref{se71}: 
\begin{theorem}\label{thc}
Let $K$ be a field of {\rm char} $K\neq 2$.\\
{\rm (i)} If $\ep_1=1$ then $K(x,y,z)^{G_{7,1,1}}$ is rational over $K$;\\
{\rm (ii)} If $\ep_1=-1$ and $[K(\sqrt{a},\sqrt{-1}):K]\leq 2$ then $K(x,y,z)^{G_{7,1,1}}$ 
is rational over $K$. 
\end{theorem}
However we do not know whether the fixed field $K(x,y,z)^{G_{7,1,1}}$ is 
rational over $K$ or not in the case where $\ep_1=-1$ and $[K(\sqrt{a},\sqrt{-1}):K]=4$. 
When $\ep_1=-1$, we remark that the group $G_{7,1,1}$ has a normal subgroup 
$G_{3,1,1}=\langle\ta_1,\la_1\rangle\in \mathcal{N}$ as in (\ref{listN}) and 
the corresponding cubic extension field $K(x,y,z)^{G_{3,1,1}}$ is rational 
over $K$ if and only if $[K(\sqrt{a},\sqrt{-1}):K]\leq 2$ by \cite[Theorem 7]{Yam}. 

By a result of \cite{Yam}, Theorem \ref{thmain} and Theorem \ref{thc}, we have the following:  
\begin{theorem}\label{thcor}
Let $K$ be a field of {\rm char} $K\neq 2$ and $G$ a finite group acting on 
$K(x,y,z)$ by monomial $K$-automorphisms. 
Then there exists $L=K(\sqrt{a})$ where $a\in K^\times$ such that the 
fixed field $L(x,y,z)^G$ under the action of monomial $L$-automorphisms of $G$ 
is rational over $L$. 
\end{theorem}

We give an application of the main results 
of the rationality problem under monomial group actions to linear Noether's problem. 
For Noether's problem and rationality problem, see e.g. \cite{Swa83}, \cite{CS07}. 

Let $\rho:G\rightarrow \mathrm{GL}(V)$ be a $4$-dimensional linear 
representation of a finite group $G$. 
Then $G$ acts on $K(V)=K(w_1,w_2,w_3,w_4)$ linearly via $\rho$, 
where $\{w_1,w_2,w_3,w_4\}$ is a basis of the dual space $V^*=\mathrm{Hom}_K(V,K)$. 
Now we assume that $\rho$ is monomial, i.e. an induced representation from a $1$-dimensional 
representation of a subgroup of $G$. 
Then for any $g\in G$, the corresponding matrix representation of $g$ has exactly 
one nonzero entry in each row and each column, that is, a monomial matrix. 
Hence $G$ acts on $K(V)_0:=K(x,y,z)$ with $x=w_1/w_4$, $y=w_2/w_4$, 
$z=w_3/w_4$, by monomial action. 
It is well-known that $K(V)^G$ is rational over $K(V)_0^G$ 
(see e.g. \cite{Miy71}, \cite{Kem96}, \cite{AHK00}, \cite{JLY02}). 
Therefore, by using Theorems \ref{thmain}, \ref{thc} and \ref{thcor}, 
we may obtain the rationality of 
$K(V)^G$ under the $4$-dimensional linear action of $G$ 
via the rationality of $K(V)_0^G=K(x,y,z)^G$ 
under the induced $3$-dimensional monomial action. 
\begin{corollary}\label{cor1}
Let $K$ be a field of {\rm char} $K\neq 2$ and $\rho : G\rightarrow \mathrm{GL}(V)$ a 
$4$-dimensional monomial linear representation. 
Then there exists $L=K(\sqrt{a})$ where $a\in K^\times$ such that the 
fixed fields $L(V)_0^G$ and $L(V)^G$ under the induced monomial $L$-automorphisms 
of $G$ are rational over $L$.
\end{corollary}
For $K=L=\bC$, Corollary \ref{cor1} is obtained by Prokhorov \cite[Theorem 5.1]{Pro10}. 
He solved the case $G=G_{7,1,1}$ with $\ep_1=-1$ via the Segre embedding 
(see Case $\Gamma_9^{12}$ in \cite[page 265]{Pro10}). 
For $K=\bQ$, see \cite{KY09}, \cite{Kit10}, \cite{Yam10}\footnote{See also 
the recent paper by Kang and Zhou \cite{KZ}.}.

In Section \ref{senot}, we give some group information of finite subgroups $G=G_{i,j,k}$ 
of $\mathrm{GL}(3,\bZ)$. 
The generators of the groups $G_{i,j,k}$ which we will use in the present paper 
are available in GAP \cite{GAP07} and Table $1$ of \cite{BBNWZ78}. 

In Section \ref{sepre}, we recall known results and give some theorems and lemmas 
which will be used in the proof of Theorem \ref{thmain}. 

In Sections \ref{se2} to \ref{se73}, we give the proof of Theorem \ref{thmain}. 
In Section \ref{se2} (resp. \ref{se3}), we consider the cases of $G_{2,j,k}$ 
(resp. $G_{3,j,k}$) which are abelian groups of exponent $2$. 
In Section \ref{se4A} (resp. \ref{se4B}), we study the case (4A) (resp. (4B)) of 
abelian groups $G_{4,j,k}$ which contain $\pm\caa$ (resp. $\pm\cbb$) of order $4$. 
In Section \ref{se5A} (resp. \ref{se5B}), we treat the case (5A) (resp. (5B))  of $G_{5,j,k}$ 
which have a normal subgroup $\langle\ca\rangle$ (resp. $\langle\cb\rangle$) of order $3$.  
In Section \ref{se6}, we treat the cases of $G_{6,j,k}$ which have a normal subgroup 
$\langle\ca\rangle$. 
In the cases of (5A) and $G_{6,j,k}$, we may assume that $c=1$, 
where $c$ will be given in (\ref{acts3abc}), without loss of generality 
except for the five groups $G_{5,1,2}$, $G_{5,3,2}$, $G_{5,3,3}$, $G_{6,1,1}$ and $G_{6,4,1}$. 
In Section \ref{secnot1}, we consider the case $c\neq 1$ for the exceptional five cases. 
These cases occur only over a field $K\ni\sqrt{-3}$ and some difficulties appear 
because $\langle \ca\rangle$ acts on both $K(x,y)$ and $K(z)$ faithfully. 
Finally we treat the cases of $G_{7,j,1}$, $G_{7,j,2}$ and $G_{7,j,3}$ which have a 
non-normal subgroup $\langle\cb\rangle$ 
in Sections $\ref{se71}$, $\ref{se72}$ and $\ref{se73}$, respectively. 


%
%
\section{Notation}\label{senot}

Let $\mathcal{S}_n$ (resp. $\mathcal{A}_n$, $\mathcal{D}_n$, $\mathcal{C}_n$) 
be the symmetric group (resp. the alternating group, the dihedral group, the cyclic group) 
of degree $n$ of order $n!$ (resp. $n!/2$, $2n$, $n$). 

The generators of the groups $G_{i,j,k}\subset\mathrm{GL}(3,\bZ)$ below are available 
in Table $1$ of \cite{BBNWZ78} and in GAP \cite{GAP07} via the command 
{\tt GeneratorsOfGroup(MatGroupZClass(3,i,j,k))}. 

There are $2$ generators $\ca$ and $\cb$ of $G_{i,j,k}$ of order $3$ and 
$4$ generators $\caa$, $-\caa$, $\cbb$, $-\cbb$ of $G_{i,j,k}$ of order $4$: 
\begin{align*}
\ca&=\left[\begin{array}{ccc} 0 & -1 & 0\\ 1 & -1 & 0\\ 0 & 0 & 1\end{array}\right],& 
\cb&=\left[\begin{array}{ccc} 0 & 0 & 1\\ 1 & 0 & 0\\ 0 & 1 & 0\end{array}\right],\\
\caa &= \left[\begin{array}{ccc} 0 & -1 & 0\\ 1 & 0 & 0\\ 0 & 0 & 1\end{array}\right], & 
\cbb &= \left[\begin{array}{ccc} 0 & 1 & 0\\ 0 & 1 & -1\\ -1 & 1 & 0\end{array}\right].
\end{align*}

Let $I_3$ be the $3\times 3$ identity matrix. 
For $\rho_1$, $\rho_2\in G_{i,j,k}$, we denote the commutator of 
$\rho_1$ and $\rho_2$ by $[\rho_1,\rho_2]=\rho_1^{-1}\rho_2^{-1}\rho_1\rho_2$. 

The $2$ groups $G_{1,j,k}$ of the $1$st crystal system in dimension $3$ are: 
\begin{align*}
G_{1,1,1}=&\{I_3\},& \mathcal{N}\ni G_{1,2,1}&=\langle-I_3\rangle\cong\mathcal{C}_2.
\end{align*} 

The $6$ groups $G_{2,j,k}$ of the $2$nd crystal system in dimension $3$ are 
abelian groups of exponent $2$ which are defined as 
\begin{align*}
G_{2,1,1}&=\langle \la_1\rangle\cong\mathcal{C}_2,& 
G_{2,1,2}&=\langle -\al\rangle\cong\mathcal{C}_2,\\
G_{2,2,1}&=\langle -\la_1\rangle\cong\mathcal{C}_2,& 
G_{2,2,2}&=\langle \al\rangle\cong\mathcal{C}_2,\\
\mathcal{N}\ni
G_{2,3,1}&=\langle \la_1,-I_3\rangle\cong\mathcal{C}_2\times \mathcal{C}_2,&
G_{2,3,2}&=\langle -\al,-I_3\rangle\cong\mathcal{C}_2\times \mathcal{C}_2
\end{align*} 
where
\begin{align*}
\la_1 &= \left[\begin{array}{ccc} -1 & 0 & 0\\ 0 & 1 & 0\\ 0 & 0 & -1\end{array}\right],& 
\al &= \left[\begin{array}{ccc} 0 & 1 & 0\\ 1 & 0 & 0\\ 0 & 0 & 1\end{array}\right],& 
-I_3 &=\left[\begin{array}{ccc} -1 & 0 & 0\\ 0 & -1 & 0\\ 0 & 0 & -1\end{array}\right]. 
\end{align*}

The $13$ groups $G_{3,j,k}$ of the $3$rd crystal system in dimension $3$ are abelian groups 
of exponent $2$ which are defined as 
\begin{align*}
\mathcal{N}\ni
G_{3,1,1}&=\langle \ta_1,\la_1\rangle\cong\mathcal{C}_2\times \mathcal{C}_2,& 
G_{3,1,2}&=\langle \ta_1,-\al\rangle\cong\mathcal{C}_2\times \mathcal{C}_2,\\
G_{3,1,3}&=\langle \ta_2,\la_2\rangle\cong\mathcal{C}_2\times \mathcal{C}_2,& 
G_{3,1,4}&=\langle \ta_3,\la_3\rangle\cong\mathcal{C}_2\times \mathcal{C}_2,\\
G_{3,2,1}&=\langle \ta_1,-\la_1\rangle\cong\mathcal{C}_2\times \mathcal{C}_2,&
G_{3,2,2}&=\langle \ta_1,\al\rangle\cong\mathcal{C}_2\times \mathcal{C}_2,\\  
G_{3,2,3}&=\langle -\al,\be\rangle\cong\mathcal{C}_2\times \mathcal{C}_2,&
G_{3,2,4}&=\langle \ta_2,-\la_2\rangle\cong\mathcal{C}_2\times \mathcal{C}_2,\\ 
G_{3,2,5}&=\langle \ta_3,-\la_3\rangle\cong\mathcal{C}_2\times \mathcal{C}_2,\\ 
\mathcal{N}\ni
G_{3,3,1}&=\langle \ta_1,\la_1,-I_3\rangle\cong\mathcal{C}_2\times 
\mathcal{C}_2\times \mathcal{C}_2, &
G_{3,3,2}&=\langle \ta_1,-\al,-I_3\rangle\cong\mathcal{C}_2\times 
\mathcal{C}_2\times \mathcal{C}_2, \\ 
G_{3,3,3}&=\langle \ta_2,\la_2,-I_3\rangle\cong\mathcal{C}_2\times 
\mathcal{C}_2\times \mathcal{C}_2,& 
G_{3,3,4}&=\langle \ta_3,\la_3,-I_3\rangle\cong\mathcal{C}_2\times 
\mathcal{C}_2\times \mathcal{C}_2
\end{align*} 
where
\begin{align}
\ta_1&=\left[\begin{array}{ccc} -1 & 0 & 0\\ 0 & -1 & 0\\ 0 & 0 & 1\end{array}\right], &
\ta_2&=\left[\begin{array}{ccc} 0 & 1 & 0\\ 1 & 0 & 0\\ -1 & -1 & -1\end{array}\right], & 
\ta_3&=\left[\begin{array}{ccc} 0 & 1 & -1\\ 1 & 0 & -1\\ 0 & 0 & -1\end{array}\right], 
\label{t1t3}\\ 
\la_1&=\left[\begin{array}{ccc} -1 & 0 & 0\\ 0 & 1 & 0\\ 0 & 0 & -1\end{array}\right], &
\la_2&=\left[\begin{array}{ccc} 0 & 0 & 1\\ -1 & -1 & -1\\ 1 & 0 & 0\end{array}\right], &
\la_3&=\left[\begin{array}{ccc} 0 & -1 & 1\\ 0 & -1 & 0\\ 1 & -1 & 0\end{array}\right],
\nonumber\\
\al&=\left[\begin{array}{ccc} 0 & 1 & 0\\ 1 & 0 & 0\\ 0 & 0 & 1\end{array}\right], &
\be&=\left[\begin{array}{ccc} 0 & -1 & 0\\ -1 & 0 & 0\\ 0 & 0 & 1\end{array}\right], & 
-I_3 &=\left[\begin{array}{ccc} -1 & 0 & 0\\ 0 & -1 & 0\\ 0 & 0 & -1\end{array}\right].
\nonumber
\end{align}


The $16$ groups $G_{4,j,k}$ of the $4$th crystal system in dimension $3$ contain an 
element $\caa$, $-\caa$, $\cbb$ or $-\cbb$ of order $4$ which are defined as 
\begin{align*}
G_{4,1,1}&=\langle \caa\rangle\cong\mathcal{C}_4,& 
G_{4,1,2}&=\langle \cbb\rangle\cong\mathcal{C}_4,\\
\mathcal{N}\ni
G_{4,2,1}&=\langle -\caa\rangle\cong\mathcal{C}_4,& 
\hspace*{-5mm}\mathcal{N}\ni
G_{4,2,2}&=\langle -\cbb\rangle\cong\mathcal{C}_4,\\
\mathcal{N}\ni
G_{4,3,1}&=\langle \caa,-I_3\rangle\cong\mathcal{C}_4\times \mathcal{C}_2,&
G_{4,3,2}&=\langle \cbb,-I_3\rangle\cong\mathcal{C}_4\times \mathcal{C}_2,\\  
\mathcal{N}\ni
G_{4,4,1}&=\langle \caa,\la_1\rangle\cong\mathcal{D}_4,&
G_{4,4,2}&=\langle \cbb,\la_3\rangle\cong\mathcal{D}_4,\\ 
G_{4,5,1}&=\langle \caa,-\la_1\rangle\cong\mathcal{D}_4,&  
G_{4,5,2}&=\langle \cbb,-\la_3\rangle\cong\mathcal{D}_4, \\ 
G_{4,6,1}&=\langle -\caa,\la_1\rangle\cong\mathcal{D}_4, &  
G_{4,6,2}&=\langle -\caa,-\la_1\rangle\cong\mathcal{D}_4, \\  
G_{4,6,3}&=\langle -\cbb,-\la_3\rangle\cong\mathcal{D}_4, & 
G_{4,6,4}&=\langle -\cbb,\la_3\rangle\cong\mathcal{D}_4, \\  
G_{4,7,1}&=\langle \caa,\la_1,-I_3\rangle\cong\mathcal{D}_4\times \mathcal{C}_2,& 
G_{4,7,2}&=\langle \cbb,\la_3,-I_3\rangle\cong\mathcal{D}_4\times \mathcal{C}_2
\end{align*} 
where
\begin{align*} 
\la_1 &= \left[\begin{array}{ccc} -1 & 0 & 0\\ 0 & 1 & 0\\ 0 & 0 & -1\end{array}\right], & 
\la_3 &= \left[\begin{array}{ccc} 0 & -1 & 1\\ 0 & -1 & 0\\ 1 & -1 & 0\end{array}\right]. 
\end{align*}
The generators $\pm \caa$, $\pm\cbb$, $\pm \la_1$, $\pm \la_3$ and $-I_3$ satisfy 
the following relations: 
\begin{align}
(\pm \caa)^4&=(\pm \cbb)^4=(\pm \la_1)^2=(\pm \la_3)^2=(-I_3)^2=I_3,\nonumber\\
[\pm \caa,\la_1]&=[\pm \caa,-\la_1]=\caa^2,\quad 
[\pm \cbb,\la_3]=[\pm \cbb,-\la_3]=\cbb^2,\label{eqc4g}\\
[\pm \caa,-I_3]&=[\pm \cbb,-I_3]=[\pm \al,-I_3]=[\pm \be,-I_3]=I_3.\nonumber
\end{align}
Note that $(\pm \caa)^2=\ta_1$ and $(\pm \cbb)^2=\ta_3$ where $\ta_1$ and $\ta_3$ are 
given as in (\ref{t1t3}) 


The $13$ groups $G_{5,j,k}$ of the $5$th crystal system in dimension $3$ 
have a normal subgroup $\langle\ca\rangle$ or $\langle\cb\rangle$ of order $3$, 
and are defined as 
\begin{align*}
G_{5,1,1}&=\langle \cb\rangle\cong\mathcal{C}_3,& 
G_{5,1,2}&=\langle \ca\rangle\cong\mathcal{C}_3,\\
G_{5,2,1}&=\langle \cb,-I_3\rangle\cong\mathcal{C}_6,& 
G_{5,2,2}&=\langle \ca,-I_3\rangle\cong\mathcal{C}_6,\\
G_{5,3,1}&=\langle \cb,-\al\rangle\cong\mathcal{S}_3,&
G_{5,3,2}&=\langle \ca,-\al\rangle\cong\mathcal{S}_3,& 
G_{5,3,3}&=\langle \ca,-\be\rangle\cong\mathcal{S}_3,\\
G_{5,4,1}&=\langle \cb,\al\rangle\cong\mathcal{S}_3,& 
G_{5,4,2}&=\langle \ca,\be\rangle\cong\mathcal{S}_3,& 
G_{5,4,3}&=\langle \ca,\al\rangle\cong\mathcal{S}_3,\\
G_{5,5,1}&=\langle \cb,-\al,-I_3\rangle\cong\mathcal{D}_6, &
G_{5,5,2}&=\langle \ca,-\al,-I_3\rangle\cong\mathcal{D}_6,& 
G_{5,5,3}&=\langle \ca,-\be,-I_3\rangle\cong\mathcal{D}_6
\end{align*}
where
\begin{align*}
\al&=\left[\begin{array}{ccc} 0 & 1 & 0\\ 1 & 0 & 0\\ 0 & 0 & 1\end{array}\right],&
\be&=\left[\begin{array}{ccc} 0 & -1 & 0\\ -1 & 0 & 0\\ 0 & 0 & 1\end{array}\right].
\end{align*}
The generators $\ca$, $\cb$, $\pm \al$, $\pm \be$ and $-I_3$ satisfy the following relations: 
\begin{align}
\ca^3&=\cb^3=(\pm \al)^2=(\pm \be)^2=(-I_3)^2=I_3,\nonumber\\
[\ca,\pm \al]&=[\ca,\pm \be]=\ca,\quad [\cb,\pm \al]=\cb,\label{relmat5}\\
[\ca,-I_3]&=[\cb,-I_3]=[\pm \al,-I_3]=[\pm \be,-I_3]=I_3.\nonumber
\end{align}


The $8$ groups $G_{6,j,1}$, $1\leq j\leq 7$, and $G_{6,6,2}$ of the $6$th crystal system in 
dimension $3$ have a normal subgroup $\langle\ca\rangle$ of order $3$, and are defined as 
\begin{align*}
G_{6,1,1}&=\langle \ca,\ta_1\rangle\cong\mathcal{C}_6,&
G_{6,2,1}&=\langle \ca,-\ta_1\rangle\cong\mathcal{C}_6,\\
G_{6,3,1}&=\langle \ca,\ta_1,-I_3\rangle\cong\mathcal{C}_6\times\mathcal{C}_2,\\
G_{6,4,1}&=\langle \ca,\ta_1,-\be\rangle\cong\mathcal{D}_6,& 
G_{6,5,1}&=\langle \ca,\ta_1,\be\rangle\cong\mathcal{D}_6,\\
G_{6,6,1}&=\langle \ca,-\ta_1,\be\rangle\cong\mathcal{D}_6,&
G_{6,6,2}&=\langle \ca,-\ta_1,-\be\rangle\cong\mathcal{D}_6,\\
G_{6,7,1}&=\langle \ca,\ta_1,-\be,-I_3\rangle\cong\mathcal{D}_6\times\mathcal{C}_2
\end{align*}
where
\begin{align*}
\ta_1&=\left[\begin{array}{ccc} -1 & 0 & 0\\ 0 & -1 & 0\\ 0 & 0 & 1\end{array}\right],&
\be&=\left[\begin{array}{ccc} 0 & -1 & 0\\ -1 & 0 & 0\\ 0 & 0 & 1\end{array}\right].
\end{align*}
The generators $\ca$, $\pm \ta_1$, $\pm \be$ and $-I_3$ satisfy the following relations: 
\begin{align}
\ca^3&=(\pm \ta_1)^2=(\pm \be)^2=(-I_3)^2=I_3,\nonumber\\
[\ca,\pm \ta_1]&=I_3,\quad [\ca,\pm \be]=\ca,\quad 
[\ta_1,\pm \be]=[-\ta_1,\pm \be]=I_3,\label{relmat6}\\
[\ca,-I_3]&=[\pm \ta_1,-I_3]=[\pm \be,-I_3]=I_3.\nonumber
\end{align}
Note that $\ta_1=(-\al)(-\be)=\al\be$. 


The $15$ groups $G_{7,j,k}$, $1\leq j\leq 5$, $1\leq k\leq 3$, 
of the $7$th crystal system in dimension $3$ have a non-normal subgroup $\langle\cb\rangle$ 
of order $3$ and a normal subgroup $\langle\ta_k,\la_k\rangle$ of order $4$. 
They are defined as 
\begin{align*}
\hspace*{1.5cm}
G_{7,1,k}&=\langle \ta_k,\la_k,\cb\rangle&\hspace{-1cm} 
&\cong \mathcal{A}_4&\hspace{-1cm} 
&\cong (\mathcal{C}_2\times \mathcal{C}_2)\rtimes \mathcal{C}_3, \\
G_{7,2,k}&=\langle \ta_k,\la_k,\cb,-I_3\rangle&\hspace{-1cm} 
&\cong \mathcal{A}_4\times \mathcal{C}_2&\hspace{-1cm} 
&\cong (\mathcal{C}_2\times \mathcal{C}_2\times \mathcal{C}_2)\rtimes \mathcal{C}_3,\\
G_{7,3,k}&=\langle \ta_k,\la_k,\cb,-\be_k\rangle&\hspace{-1cm} 
&\cong \mathcal{S}_4&\hspace{-1cm} 
&\cong (\mathcal{C}_2\times \mathcal{C}_2)\rtimes \mathcal{S}_3, \\
G_{7,4,k}&=\langle \ta_k,\la_k,\cb,\be_k\rangle&\hspace{-1cm} 
&\cong \mathcal{S}_4&\hspace{-1cm} 
&\cong (\mathcal{C}_2\times \mathcal{C}_2)\rtimes \mathcal{S}_3, \\
G_{7,5,k}&=\langle \ta_k,\la_k,\cb,\be_k,-I_3\rangle&\hspace{-1cm} 
&\cong \mathcal{S}_4\times \mathcal{C}_2&\hspace{-1cm} 
&\cong (\mathcal{C}_2\times \mathcal{C}_2\times \mathcal{C}_2)\rtimes \mathcal{S}_3
\end{align*}
where
\begin{align*}
\ta_1&=\left[\begin{array}{ccc}-1 & 0 & 0\\ 0 & -1 & 0\\ 0 & 0 & 1\end{array}\right],&
\la_1&=\left[\begin{array}{ccc}-1 & 0 & 0\\ 0 & 1 & 0\\ 0 & 0 & -1\end{array}\right],&
\be_1&=\left[\begin{array}{ccc}0 & -1 & 0\\ -1 & 0 & 0\\ 0 & 0 & 1\end{array}\right],\\
\ta_2&=\left[\begin{array}{ccc}0 & 1 & 0\\ 1 & 0 & 0\\ -1 & -1 & -1\end{array}\right],&
\la_2&=\left[\begin{array}{ccc}0 & 0 & 1\\ -1 & -1 & -1\\ 1 & 0 & 0\end{array}\right],&
\be_2&=\left[\begin{array}{ccc}1 & 0 & 0\\ 0 & 1 & 0\\ -1 & -1 & -1\end{array}\right],\\
\ta_3&=\left[\begin{array}{ccc}0 & 1 & -1\\ 1 & 0 & -1\\ 0 & 0 & -1\end{array}\right],&
\la_3&=\left[\begin{array}{ccc}0 & -1 & 1\\ 0 & -1 & 0\\ 1 & -1 & 0\end{array}\right],&
\be_3&=\left[\begin{array}{ccc}1 & 0 & -1\\ 0 & 1 & -1\\ 0 & 0 & -1\end{array}\right]
\end{align*}
and the generators $\ta_k$, $\la_k$, $\cb$, $\pm \be_k$, $1\leq k\leq 3$, and $-I_3$ satisfy 
the following relations: 
\begin{align}
\ta_k^2&=\la_k^2=\cb^3=(\pm \be_k)^2=(-I_3)^2=I_3,\nonumber\\
[\ta_k,\la_k]&=I_3,\quad [\ta_k,\cb]=\la_k\ta_k,\quad [\la_k,\cb]=\ta_k,\label{relmat7}\\
[\ta_k,\pm \be_k]&=I_3,\quad [\la_k,\pm \be_k]=\ta_k, \quad [\cb,\pm \be_k]=\cb\la_k,\nonumber\\
[\ta_k,-I_3]&=[\la_k,-I_3]=[\cb,-I_3]=[\pm \be_k,-I_3]=I_3.\nonumber
\end{align}

We modified the generators of $G_{7,5,k}$. 
The original generators in \cite{BBNWZ78} are obtained by just replacing 
$\be_k$ by $-\be_k=\be_k(-I_3)$. 


%
%
\section{Preliminaries}\label{sepre}

We prepare some lemmas, theorems and known results which we will use in the proof of 
Theorem \ref{thmain}.\\

\subsection{Rationality criterion and reduction to low degree}
~\\
\begin{theorem}[{\cite[Theorem 3.1]{AHK00}}]\label{thAHK}
Let $L(z)$ be the rational function field over a field $L$ with one variable $z$ 
and $G$ a group acting on $L(z)$. 
Suppose that, for any $\ta\in G$, $\ta(L)\subset L$ and $\ta(z)=a_\ta z+b_\ta$ for some
$a_\ta\in L^\times$ and $b_\ta\in L$. 
Then $L(z)^G=L^G$ or $L(z)^G = L^G(\Theta)$ for some polynomial $\Theta(z)\in L[z]$ 
with positive degree. 
In particular, if $L^G$ is rational over some subfield $M$, so is $L(z)^G$ over $M$.
\end{theorem}
\begin{theorem}\label{thInv}
Let $K$ be a field of {\rm char} $K\neq 2$ and 
$\ta$ a $K$-automorphism of $K(x,y)$ defined by
\begin{align*}
\ta(x)=-x,\quad \ta(y)=f(x^2){\rm N}_\ta(g(x))/y\quad \mathit{with}\quad f(x), g(x)\in K[x]
\end{align*}
where ${\rm N}_\ta$ is the norm map under the action of $\ta$. 
Then the fixed field $K(x,y)^{\langle\ta\rangle}$ is given by 
\[
K(x,y)^{\langle\ta\rangle}=K(X,Y,Z)\quad \mathit{with}\quad X^2-ZY^2=f(Z)
\]
where $X=(y'+f(x^2)/y')/2$, $Y=(y'-f(x^2)/y')/(2x)$ with $y'=y/g(x)$ and $Z=x^2$. 
Moreover we get:\\
{\rm (I)} When $f(x^2)$ is a constant of $K$ or quadratic in $x$, i.e. 
$f(x^2)=ax^2+b$ for $a,b\in K$ with $a\neq 0$, then $f(Z)=aZ+b$ and the fixed field 
$K(x,y)^{\langle\ta\rangle}=K(X,Y)$ is rational over $K$;\\
{\rm (II)} When $f(x^2)$ is biquadratic in $x$, i.e. $f(x^2)=ax^4+bx^2+c$ 
for $a,b,c\in K$ with $a\neq 0$, then $f(Z)=aZ^2+bZ+c$ and 
the fixed field $K(x,y)^{\langle\ta\rangle}$ is given by 
\begin{align*}
K(x,y)^{\langle\ta\rangle}=K(X,Y,Z)\quad \mathit{with}\quad X^2-ZY^2=aZ^2+bZ+c.
\end{align*}
In particular,\\
{\rm (II-i)} if $a=d^2$ is square in $K$ then $K(x,y)^{\langle\ta\rangle}=K(X+dZ,Y)$ 
is rational over $K$;\\
{\rm (II-ii)} if $c=e^2$ is square in $K$ then $K(x,y)^{\langle\ta\rangle}$ is also 
rational over $K$. 
\end{theorem}
\begin{proof}
Put $y':=y/g(x)$. 
Then $K(x,y)=K(x,y')$ and the action of $\ta$ is given by 
$\ta(\delta)=-\delta$, $\ta(y')=f(x^2)/y'$. 
It follows from $[K(x,y') : K(X,Z)]=4$ and $Y\not\in K(X,Z)$ 
that $K(x,y')^{\langle\ta\rangle}=K(X,Y,Z)$. 
The equality $X^2-ZY^2=f(Z)$ can be obtained by the definition of $X$, $Y$ and $Z$ directly. 

(I) If $f(Z)=aZ+b$ then $Z=(X^2-b)/(Y^2+a)$ and hence 
$K(x,y)^{\langle\ta\rangle}=K(X,Y,Z)=K(X,Y)$. 

(II-i) If $a=d^2$ then $X^2-aZ^2=(b+Y^2)Z+c$ is rewritten as $st=(b+Y^2)\frac{s-t}{2d}+c$ 
where $s=X+dZ$, $t=X-dZ$, so that $t\in K(s,Y)$ and 
$K(x,y)^{\langle\ta\rangle}=K(X,Y,Z)=K(s,t,Y)=K(s,Y)=K(X+dZ,Y)$. 

(II-ii) If $c=e^2$ then by putting $p:=1/x$, $q:=y'/x^2$, the action of $\ta$ on 
$K(p,q)=K(x,y')=K(x,y)$ is given by $\ta\,:\, p\mapsto -p$, $q\mapsto (cp^4+bp^2+a)/q$. 
Hence the assertion follows from (II-i). \\
\end{proof}
%

\subsection{Explicit transcendental bases}
~\\
\begin{theorem}[{\cite[Lemma 2.7]{HK94}}, {\cite[Theorem 2.4]{Kan04}}]\label{thab}
Let $K$ be any field. 
Let $-I_2\in\mathrm{GL}(2,\bZ)$ act on $K(x,y)$ by 
\begin{align*}
-I_2\,:\, x\ \mapsto\ \frac{a}{x},\ y\ \mapsto\ \frac{b}{y},\quad a\in K^\times
\end{align*}
where $b=c(x+(a/x))+d$ such that $c,d\in K$ and at least one of $c$ and $d$ 
is non-zero. 
Then $K(x,y)^{\langle \ta\rangle}=K(u_1,u_2)$ where
\begin{align*} 
u_1=\frac{x-\frac{a}{x}}{xy-\frac{ab}{xy}}=\frac{y(x^2-a)}{x^2y^2-ab},\quad 
u_2=\frac{y-\frac{b}{y}}{xy-\frac{ab}{xy}}=\frac{x(y^2-b)}{x^2y^2-ab}.
\end{align*}
\end{theorem}
\begin{lemma}[{\cite[page 1176]{HHR08}}]\label{lemaa}
Let $K$ be a field of {\rm char} $K\neq 2$ and $-I_2\in\mathrm{GL}(2,\bZ)$ act on $K(x,y)$ by 
\begin{align*}
-I_2\,:\, x\ \mapsto\ \frac{a}{x},\ y\ \mapsto\ \frac{a}{y},\quad a\in K^\times. 
\end{align*}
Then $K(x,y)^{\langle \ta\rangle}=K(t_1,t_2)$ where
\begin{align*} 
t_1=\frac{xy+a}{x+y},\quad t_2=\frac{xy-a}{x-y}.
\end{align*}
\end{lemma}

A non-$2$-group $G\subset \mathrm{GL}(3,\bZ)$ 
has a subgroup either $G_{5,1,2}=\langle\ca\rangle$ or $G_{5,1,1}=\langle\cb\rangle$ of order $3$, 
and the corresponding actions of $\ca$ and of $\cb$ on $K(x,y,z)$ are given respectively by
\begin{align*}
\ca\,&:\, x\ \mapsto\ ay,\ y\mapsto\ \frac{b}{xy},\ z\mapsto\ cz,\quad a,b,c\in K^\times,\\
\cb\,&:\, x\mapsto ay,\ y\mapsto bz,\ z\mapsto cx,\quad a,b,c\in K^\times.
\end{align*}

For $\ca$ (resp. $\cb$), we may assume that $a=1$ (resp. $a=b=c=1$) after replacing 
$ay$ by $y$ (resp. $(ay,abz)$ by $(y,z)$). 
In the case of $\ca$, we also see $c^3=1$ because $\ca^3=I_3$. 

In order to investigate the fixed field $K(x,y,z)^G$ with $G_{5,1,1}=\langle\cb\rangle\subset G$, 
we recall the following formula which is given by Masuda \cite{Mas55} when char 
$K\neq 3$ (this formula is valid also for the case of char $K=3$, 
see e.g. \cite{HK10}, \cite{Kun55}). 
\begin{lemma}\label{lemMas}
Let $K$ be any field and $G_{5,1,1}=\langle\cb\rangle$ act on $K(x,y,z)$ by 
\[
\cb\,:\, x\ \mapsto\ y\ \mapsto\ z\ \mapsto\ x.
\]
Then $K(x,y,z)^{G_{5,1,1}}=K(x,y,z)^{\langle \cb\rangle}=K(s_1,u,v)$ where 
\begin{align*}
s_1&:=\,s_1(x,y,z)\,=\,x+y+z,\\
u&:=\,u(x,y,z)\,=\,\frac{xy^2+yz^2+zx^2-3xyz}{x^2+y^2+z^2-xy-yz-zx},\\
v&:=\,v(x,y,z)\,=\,\frac{x^2y+y^2z+z^2x-3xyz}{x^2+y^2+z^2-xy-yz-zx}.
\end{align*}
\end{lemma}
Note that $u(y,x,z)=v(x,y,z)$. 
In order to study the fixed field $K(x,y,z)^G$ with $G_{5,1,2}=\langle \ca\rangle\subset G$, 
we prepare the following lemma using Lemma \ref{lemMas}: 
\begin{lemma}\label{lemMas2}
Let $K$ be any field and $\si$ a $K$-automorphism on $K(x,y)$ of order $3$ defined by 
\[
\si\,:\, x\ \mapsto\ y\ \mapsto\ \frac{b}{xy}\ \mapsto\ x,\quad b\in K^\times.
\]
Then $K(x,y)^{\langle \si\rangle}=K(\tu(b),\tv(b))$ where 
\begin{align*}
\tu(b)
=\frac{y\left(y^3x^3+bx^3-3byx^2+b^2\right)}{y^2x^4-y^3x^3+y^4x^2-byx^2-by^2x+b^2},\\
\tv(b)
=\frac{x\left(x^3y^3+by^3-3bxy^2+b^2\right)}{y^2x^4-y^3x^3+y^4x^2-byx^2-by^2x+b^2}.
\end{align*}
\end{lemma}
\begin{proof}
Put 
\begin{align*}
\ts(b):=s_1\Bigl(x,y,\frac{b}{xy}\Bigr)\,=\,\frac{yx^2+y^2x+b}{xy},\quad 
\tu(b):=u\Bigl(x,y,\frac{b}{xy}\Bigr),\quad 
\tv(b):=v\Bigl(x,y,\frac{b}{xy}\Bigr)
\end{align*}
where $s_1(x,y,z)$, $u(x,y,z)$ and $v(x,y,z)$ are given as in Lemma \ref{lemMas}. 
We write $(\ts,\tu,\tv)=(\ts(b),\tu(b),\tv(b))$ for simplicity. 
Then $\ts,\tu,\tv$ are $\si$-invariants and hence $K(x,y)^{\langle \si\rangle}\supset 
K(\ts,\tu,\tv)$. 
We also see $[K(x,y)\,:\,K(\ts,\tu,\tv)]\leq 3$ 
from the equalities
\begin{align*}
\frac{y^3-\ts y^2-b}{y}=-\frac{\ts^2\bigl(\tu+\tv\bigr)+9b}
{\ts+3\bigl(\tu+\tv\bigr)},\quad 
x=\frac{y^2-\ts y-2\tu y+\tv y+\ts\tu}{\tu+\tv-2y}.
\end{align*}
Hence $K(x,y)^{\langle \si\rangle}=K(\ts,\tu,\tv)$. 
The assertion follows from $\ts=(\tu^3+\tv^3+b)/(\tu\,\tv)$. 
\end{proof}
\begin{lemma}\label{lemV4}
Let $K$ be any field. 
Let $\ta$ and $\la$ be $K$-automorphisms on $K(x,y,z,w)$ defined by 
\begin{align*}
\ta\,&:\, x\ \mapsto\ y\ \mapsto\ x,\ z\ \mapsto\ w\ \mapsto z,\\
\la\,&:\, x \mapsto\ z\ \mapsto\ x,\ y\ \mapsto\ w\ \mapsto y. 
\end{align*}
Then $K(x,y,z,w)^{\langle \ta,\la\rangle}=K(v_0,v_1,v_2,v_3)$ where 
\begin{align*}
v_0&=x+y+z+w,\\
v_1&=\frac{x+y-z-w}{xy-zw},\quad 
v_2=\frac{x-y-z+w}{xw-yz},\quad 
v_3=\frac{x-y+z-w}{xz-yw}. 
\end{align*}
\end{lemma}
\begin{proof}
We put $(u_0,u_1,u_2,u_3):=(v_0,1/v_1,1/v_2,1/v_3)$. 
We should show that $K(x,y,z,w)^{\langle \ta,\la\rangle}=K(u_0,u_1,u_2,u_3)$. 
It follows from the definition that 
$K(x,y,z,w)^{\langle \ta,\la\rangle}\supset K(u_0,u_1,u_2,u_3)$. 
Take the elementary symmetric functions $s_1$, $s_2$, $s_3$ in $u_1$, $u_2$, $u_3$: 
\[
(s_1,s_2,s_3):=(u_1+u_2+u_3,u_1u_2+u_2u_3+u_1u_3,u_1u_2u_3).
\]
Then the elementary symmetric functions $r_1$, $r_2$, $r_3$, $r_4$ in $x$, $y$, $z$, $w$ are 
written as 
\[
r_1=u_0,\ r_2=s_1u_0-2s_2,\ r_3=s_2u_0-8s_3,\ r_4=s_2^2-4s_1s_3+s_3u_0,
\]
so that $K(r_1,r_2,r_3,r_4)\subset K(u_0,s_1,s_2,s_3)$. 
Since $[K(x,y,z,w) : K(r_1,r_2,r_3,r_4)]=24$ and $[K(u_0,u_1,u_2,u_3) : K(u_0,s_1,s_2,s_3)]=6$, 
we get $[K(x,y,z,w) : K(u_0,u_1,u_2,u_3)\leq 4$. 
\end{proof}
\begin{lemma}\label{lemV42}
Let $K$ be any field and $G_{3,1,4}=\langle\ta_3,\la_3\rangle$ act on $K(x,y,z)$ by 
\begin{align*}
\ta_3\,&:\, x\mapsto\ y\ \mapsto x,\ z\mapsto\ \frac{c}{xyz}\ \mapsto z,\\
\la_3\,&:\, x\mapsto\ z\ \mapsto x,\ y\mapsto\ \frac{c}{xyz}\ \mapsto y,\quad c\in K^\times. 
\end{align*}
Define $w=c/(xyz)$. 
Then $K(x,y,z)^{G_{3,1,4}}=K(x,y,z)^{\langle \ta_3,\la_3\rangle}=K(v_1,v_2,v_3)$ where 
\begin{align*}
v_1\ &=\ \frac{x+y-z-w}{xy-zw}\ =\ \frac{c-xyz(x+y-z)}{z(c-x^2y^2)},\\
v_2\ &=\ \frac{x-y-z+w}{xw-yz}\ =\ \frac{c-xyz(-x+y+z)}{x(c-y^2z^2)},\\
v_3\ &=\ \frac{x-y+z-w}{xz-yw}\ =\ \frac{c-xyz(x-y+z)}{y(c-x^2z^2)}.
\end{align*}
\end{lemma}
\begin{proof}
By Lemma \ref{lemV4}, we see $K(x,y,z)^{\langle \ta_3,\la_3\rangle}
=K(x,y,z,w)^{\langle \ta_3,\la_3\rangle}=K(v_0,v_1,v_2,v_3)$ 
where $v_0=x+y+z+w$. 
By the definition of $v_0,v_1,v_2,v_3$ and the equality $xyzw=c$, 
we get 
\begin{align*}
v_0=\frac{2(v_1v_2+v_2v_3+v_1v_3)-(v_1^2+v_2^2+v_3^2)+cv_1^2v_2^2v_3^2}{v_1v_2v_3}.
\end{align*}
Hence the assertion follows. 
\end{proof}

We need suitable generators $K(x,y,z)^G$ over $K$ 
not only for $G_{3,1,4}=\langle\ta_3,\la_3\rangle$ 
but also for $G=\langle\ta_3\rangle$. 
We note that $\langle\ta_3\rangle$ is conjugate to $G_{2,1,2}=\langle-\al\rangle$ 
in $\mathrm{GL}(3,\bZ)$. 
\begin{lemma}\label{lemc2t}
Let $K$ be a field of {\rm char} $K\neq 2$ and $\langle\ta_3\rangle$ act on $K(x,y,z)$ by 
\begin{align*}
\ta_3\,:\, x\ \mapsto\ y\ \mapsto\ x,\quad z\ \mapsto\ \frac{c}{xyz}\ \mapsto\ z,\quad 
c\in K^\times. 
\end{align*}
Then $K(x,y,z)^{\langle \ta_3\rangle}=K(t_1,t_2,t_3)$ where
\begin{align*} 
t_1=\frac{xy}{x+y},\quad t_2=\frac{xyz+\frac{c}{z}}{x+y},\quad t_3=\frac{xyz-\frac{c}{z}}{x-y}.
\end{align*}
\end{lemma}
\begin{proof}
Put 
\[
X:=\frac{c}{yz},\quad Y:=xz,\quad Z:=\frac{1}{x}.
\]
Then $K(x,y,z)=K(X,Y,Z)$ and the action of $\ta_3$ on $K(X,Y,Z)$ is given by 
\[
\ta_3\,:\, X\ \mapsto\ \frac{c}{X},\ Y\ \mapsto\ \frac{c}{Y},\ Z\ \mapsto\ \frac{XYZ}{c}.
\]
Using Lemma \ref{lemaa}, we see $K(X,Y,Z)^{\langle\ta_3\rangle}
=K(X,Y,t_3')^{\langle\ta_3\rangle}=K(X,Y)^{\langle\ta_3\rangle}(t_3')=K(t_1',t_2',t_3')$ where 
\begin{align*}
t_1':=\frac{XY+c}{X+Y},\quad t_2':=\frac{XY-c}{X-Y},\quad t_3':=Z+\frac{XYZ}{c}. 
\end{align*}
It is easy to see that $t_1'=1/t_2$, $t_2'=1/t_3$, $t_3'=1/t_1$. 
\end{proof}
\begin{theorem}[{\cite[Theorem 3.2]{HK97}}]\label{thHK97}
Let $K$ be a field of {\rm char} $K\neq 2$. 
If $\mathcal{S}_n$ acts on $K(x_1,\ldots,x_n)$ by either 
\begin{align*}
\mathrm{(I)}\quad 
\ta(x_i):=\begin{cases}x_{\ta(i)},\quad \mathrm{if}\ \ta\in\mathcal{A}_n,\\
1/x_{\ta(i)},\ \mathrm{if}\ \ta\in\mathcal{S}_n\backslash\mathcal{A}_n,\end{cases}
\mathrm{or}\quad 
\mathrm{(II)}\quad 
\ta(x_i):=\begin{cases}x_{\ta(i)},\quad \mathrm{if}\ \ta\in\mathcal{A}_n,\\
-x_{\ta(i)},\ \mathrm{if}\ \ta\in\mathcal{S}_n\backslash\mathcal{A}_n,\end{cases}
\end{align*}
then the fixed field $K(x_1,\ldots,x_n)^{\mathcal{S}_n}$ is rational over $K$. 
\end{theorem}
\begin{theorem}[{\cite[Theorem 1.2]{HK10}}]\label{thHK1}
Let $K$ be any field and $G_{5,3,1}=\langle\cb,-\al\rangle\cong\mathcal{S}_3$ act on 
$K(x,y,z)$ by 
\begin{align*}
\cb\ :\ &\ x\ \mapsto\ y,\ y\mapsto\ z,\ z\ \mapsto\ x,\\
-\al\ :\ &\ x\ \mapsto\ \frac{a}{y},\ y\mapsto\ \frac{a}{x},\ z\ \mapsto\ \frac{a}{z},
\quad a\in K^\times. 
\end{align*}
Then the fixed field $K(x,y,z)^{\langle\cb,-\al\rangle}$ is rational over $K$. 
Moreover we have 
\begin{align*}
K(x,y,z)^{\langle\cb,-\al\rangle}=
K\Bigl(\frac{U(s_3^2-a^3)}{u(s_3^2U-a^4)},\frac{s_3(U-a)}{s_3^2U-a^4},\frac{v}{u}\Bigr)
\end{align*}
where $s_3=xyz$, $U=u^2-uv+v^2$ and $u,v$ are given as in Lemma \ref{lemMas}. 
\end{theorem}
\begin{theorem}[{\cite[Theorem 2.4]{HK10}}]\label{thHK2}
Let $K$ be a field of {\rm char} $K\neq 2$. 
Let $\si$ and $\ta$ be $K$-automorphisms of $K(x,y,z)$ defined by 
\begin{align*}
\si\ :\ &\ x\ \mapsto\ y,\ y\mapsto\ z,\ z\mapsto\ x,\\
\ta\ :\ &\ x\ \mapsto\ \frac{-x+y+z}{ayz},\ y\mapsto\ \frac{x+y-z}{axy},\ 
z\ \mapsto\ \frac{x-y+z}{axz},\quad a\in K^\times.
\end{align*}
Then $\langle\si,\ta\rangle\cong \mathcal{S}_3$ and the fixed field 
$K(x,y,z)^{\langle\si,\ta\rangle}$ is rational over $K$. 
\end{theorem}
\begin{theorem}\label{thaaa}
Let $K$ be a field of {\rm char} $K\neq 2$ and $G_{1,2,1}=\langle-I_3\rangle$ 
act on $K(x,y,z)$ by 
\begin{align*}
-I_3\,:\, x\mapsto\ \frac{a}{x},\ y\mapsto\ \frac{a}{y},\ z\ \mapsto\ \frac{a}{z},\quad 
a\in K^\times. 
\end{align*}
Then the fixed field $K(x,y,z)^{G_{1,2,1}}=K(x,y,z)^{\langle -I_3\rangle}$ is rational over $K$. 
In particular, an explicit transcendental basis of 
$K(x,y,z)^{\langle-I_3\rangle}=K(k_1,k_2,k_3)$ over $K$ is given by 
\begin{align*} 
k_1=\frac{xy+a}{x+y},\quad 
k_2=\frac{yz+a}{y+z},\quad 
k_3=\frac{xz+a}{x+z}.
\end{align*}
\end{theorem}
\begin{proof}
We take $L:=K(\sqrt{a})$ and $\mathrm{Gal}(L/K)=\langle\rho\rangle$ where 
$\rho(\sqrt{a})=-\sqrt{a}$. 
We extend the action of $-I_3$ to $L(x,y,z)$ with trivial action on $L$. 
Then we have $K(x,y,z)^{\langle -I_3\rangle}
=L^{\langle\rho\rangle}(x,y,z)^{\langle -I_3\rangle}=L(x,y,z)^{\langle -I_3,\rho\rangle}$ 
because the actions of $\rho$ on $K(x,y,z)$ and $-I_3$ on $L$ are trivial respectively. 
Put 
\[
X:=\frac{x+\sqrt{a}}{x-\sqrt{a}},\quad Y:=\frac{y+\sqrt{a}}{y-\sqrt{a}},\quad 
Z:=\frac{z+\sqrt{a}}{z-\sqrt{a}}. 
\]
Then $L(x,y,z)=K(\sqrt{a})(X,Y,Z)$ and the actions of $-I_3$ and $\rho$ on 
$K(\sqrt{a})(X,Y,Z)$ are given by
\begin{align*}
-I_3\ :\ &\ \sqrt{a}\ \mapsto\ \sqrt{a},\ X\ \mapsto\ -X,\ Y\mapsto\ -Y,\ Z\mapsto\ -Z,\\
\rho\ :\ &\ \sqrt{a}\ \mapsto\ -\sqrt{a},\ X\ \mapsto\ \frac{1}{X},\ Y\mapsto\ \frac{1}{Y},\ 
Z\mapsto\ \frac{1}{Z}.
\end{align*}
Hence we get $L(x,y,z)^{\langle -I_3\rangle}=L(u_1,u_2,u_3)$ where 
\[
u_1:=XY,\quad u_2:=YZ,\quad u_3:=XZ. 
\]
It follows that $K(x,y,z)^{\langle -I_3\rangle}
=L^{\langle\rho\rangle}(x,y,z)^{\langle -I_3\rangle}
=(L(x,y,z)^{\langle -I_3\rangle})^{\langle\rho\rangle}=L(u_1,u_2,u_3)^{\langle\rho\rangle}$. 
The action of $\rho$ on $L(u_1,u_2,u_3)=K(\sqrt{a})(u_1,u_2,u_3)$ is given by
\[
\rho\ :\ \sqrt{a}\ \mapsto\ -\sqrt{a},\ u_1\ \mapsto\ \frac{1}{u_1},\ 
u_2\ \mapsto\ \frac{1}{u_2},\ u_3\ \mapsto\ \frac{1}{u_3}.
\]
Hence we also put 
\[
k_1':=\frac{u_1+1}{u_1-1},\quad k_2':=\frac{u_2+1}{u_2-1},\quad k_3':=\frac{u_3+1}{u_3-1}.
\]
Then $L(u_1,u_2,u_3)=L(k_1',k_2',k_3')$ and the action of $\rho$ on $L(k_1',k_2',k_3')$ 
is given by $\rho\,:\, \sqrt{a}\mapsto -\sqrt{a}$, $k_i'\mapsto -k_i'$ for $1\leq i\leq 3$. 
Therefore we get 
\[
K(x,y,z)^{\langle-I_3\rangle}=L(k_1',k_2',k_3')^{\langle\rho\rangle}
=L^{\langle\rho\rangle}(\sqrt{a}\,k_1',\sqrt{a}\,k_2',\sqrt{a}\,k_3')
=K(\sqrt{a}\,k_1',\sqrt{a}\,k_2',\sqrt{a}\,k_3').
\]
It can be evaluated from the definition that $\sqrt{a}\,k_i'=k_i$ for $1\leq i\leq 3$. 
\end{proof}
\begin{remark}(i) Theorem \ref{thaaa} is also valid for a field $K$ of char $K=2$.\\
(ii) For general $n\geq 2$, 
we can get an explicit transcendental basis of $K(x_1,\ldots,x_n)^{\langle-I_n\rangle}$ 
over $K$ under the action of $-I_n : x_i\mapsto a/x_i$, $1\leq i\leq n$ by the same manner 
as above. 
Indeed we see 
\begin{align*}
K(x_1,\ldots,x_n)^{\langle-I_n\rangle}
&=K\Bigl(\frac{x_ix_n+a}{x_i+x_n},\frac{x_n^2+a}{x_n} \,\Big{|}\, 1\leq i\leq n-1\Bigr)\\
&=\begin{cases}
K\Bigl(\displaystyle{\frac{x_ix_{i+1}+a}{x_i+x_{i+1}},\frac{x_nx_1+a}{x_n+x_1}} 
\,\Big{|}\, 1\leq i\leq n-1\Bigr),\ \mathrm{if}\  n\ \mathrm{is\ odd},\vspace*{2mm}\\
K\Bigl(\displaystyle{\frac{x_ix_{i+1}+a}{x_i+x_{i+1}},\frac{x_nx_1-a}{x_n-x_1}} 
\,\Big{|}\, 1\leq i\leq n-1\Bigr),\ \mathrm{if}\ n\ \mathrm{is\ even}.
\end{cases}
\end{align*}
\end{remark}
\begin{theorem}[{\cite[Theorem 6]{Yam}}]\label{th231}
Let $K$ be a field of {\rm char} $K\neq 2$ and 
$\langle-\ta_1,\ta_1\rangle\cong\mathcal{C}_2\times\mathcal{C}_2$ act on 
$K(x,y,z)$ by 
\begin{align*}
-\ta_1\, &:\, x\ \mapsto\ \ep_1 x,\ y\mapsto\ \ep_2 y,\ z\ \mapsto\ \frac{c}{z},\\
\ta_1\, &:\, x\ \mapsto\ \frac{a}{x},\ y\mapsto\ \frac{b}{y},\ z\ \mapsto\ \ep_3 z
\end{align*}
where $a,b,c\in K^\times$ and $\ep_1,\ep_2,\ep_3=\pm 1$. 
If $\ep_3=1$ then $K(x,y,z)^{\langle-\ta_1,\ta_1\rangle}$ is rational over $K$. 
In particular, an explicit transcendental basis of $K(x,y,z)^{\langle-\ta_1,\ta_1\rangle}
=K(u_1,u_2,u_3)$ over $K$ is given by 
\begin{align*}
u_1=\begin{cases}
t_1,\hspace*{16.6mm} \mathrm{if}\ \ep_2=1,\\
\displaystyle{t_1\Bigl(z-\frac{c}{z}\Bigr)},\ \mathrm{if}\ \ep_2=-1,\\
\end{cases}\ 
u_2=\begin{cases}
t_2,\hspace*{16.6mm} \mathrm{if}\ \ep_1=1,\\
\displaystyle{t_2\Bigl(z-\frac{c}{z}\Bigr)},\ \mathrm{if}\ \ep_1=-1,
\end{cases}\ 
u_3:=z+\frac{c}{z}
\end{align*}
where
\begin{align*}
t_1=\frac{x-\frac{a}{x}}{xy-\frac{ab}{xy}}=\frac{y(x^2-a)}{x^2y^2-ab},\quad 
t_2=\frac{y-\frac{b}{y}}{xy-\frac{ab}{xy}}=\frac{x(y^2-b)}{x^2y^2-ab}.
\end{align*}
\end{theorem}
\begin{theorem}\label{thmex}
Let $K$ be a field of {\rm char} $K\neq 2$ and $G_{3,1,1}=\langle\ta_1,\la_1\rangle$ act on 
$K(x,y,z)$ by 
\begin{align*}
\ta_1\ :\ &\ x\ \mapsto\ \frac{a}{x},\ y\ \mapsto\ \frac{a}{y},\ z\ \mapsto\ z,\\
\la_1\ :\ &\ x\ \mapsto\ \frac{a}{x},\ y\ \mapsto\ y,\ z\ \mapsto\ \frac{a}{z},
\quad a\in K^\times. 
\end{align*}
Then the fixed field $K(x,y,z)^{G_{3,1,1}}$ is rational over $K$. 
In particular, an explicit transcendental basis of 
$K(x,y,z)^{\langle\ta_1,\la_1\rangle}=K(v_1,v_2,v_3)$ over $K$ is given by 
\begin{align*}
v_1=\frac{a(-x+y+z)-xyz}{a-xy-xz+yz},\quad
v_2=\frac{a(x-y+z)-xyz}{a-xy+xz-yz},\quad
v_3=\frac{a(x+y-z)-xyz}{a+xy-xz-yz}.
\end{align*}
\end{theorem}
\begin{proof}
We take $L:=K(\sqrt{a})$ and $\mathrm{Gal}(L/K)=\langle\rho\rangle$ where 
$\rho(\sqrt{a})=-\sqrt{a}$. 
We extend the action of $\langle\ta_1,\la_1\rangle$ to $L(x,y,z)$ with trivial action on $L$. 
Then we have $K(x,y,z)^{\langle\ta_1,\la_1\rangle}
=L^{\langle\rho\rangle}(x,y,z)^{\langle\ta_1,\la_1\rangle}=L(x,y,z)^{\langle\ta_1,\la_1,\rho\rangle}$ 
because the actions of $\rho$ on $K(x,y,z)$ and $\ta_1,\la_1$ on $L$ are trivial respectively. 
Put 
\[
X:=\frac{x+\sqrt{a}}{x-\sqrt{a}},\quad Y:=\frac{y+\sqrt{a}}{y-\sqrt{a}},\quad 
Z:=\frac{z+\sqrt{a}}{z-\sqrt{a}}. 
\]
Then $L(x,y,z)=K(\sqrt{a})(X,Y,Z)$ and the actions of $\ta_1,\la_1$ and $\rho$ on 
$K(\sqrt{a})(X,Y,Z)$ are given by
\begin{align*}
\ta_1\ :\ &\ \sqrt{a}\ \mapsto\ \sqrt{a},\ X\ \mapsto\ -X,\ Y\mapsto\ -Y,\ Z\mapsto\ Z,\\
\la_1\ :\ &\ \sqrt{a}\ \mapsto\ \sqrt{a},\ X\ \mapsto\ -X,\ Y\mapsto\ Y,\ Z\mapsto\ -Z,\\
\rho\ :\ &\ \sqrt{a}\ \mapsto\ -\sqrt{a},\ X\ \mapsto\ \frac{1}{X},\ Y\mapsto\ \frac{1}{Y},\ 
Z\mapsto\ \frac{1}{Z}.
\end{align*}
Hence we get $L(x,y,z)^{\langle\ta_1,\la_1\rangle}=L(u_1,u_2,u_3)$ where 
\[
u_1:=\frac{YZ}{X},\quad u_2:=\frac{XZ}{Y},\quad u_3:=\frac{XY}{Z}. 
\]
It follows that $K(x,y,z)^{\langle\ta_1,\la_1\rangle}
=L^{\langle\rho\rangle}(x,y,z)^{\langle\ta_1,\la_1\rangle}
=(L(x,y,z)^{\langle\ta_1,\la_1\rangle})^{\langle\rho\rangle}=L(u_1,u_2,u_3)^{\langle\rho\rangle}$. 
The action of $\rho$ on $L(u_1,u_2,u_3)=K(\sqrt{a})(u_1,u_2,u_3)$ is given by
\[
\rho\ :\ \sqrt{a}\ \mapsto\ -\sqrt{a},\ u_1\ \mapsto\ \frac{1}{u_1},\ 
u_2\ \mapsto\ \frac{1}{u_2},\ u_3\ \mapsto\ \frac{1}{u_3}.
\]
Hence we also put 
\[
v_1':=\frac{u_1+1}{u_1-1},\quad v_2':=\frac{u_2+1}{u_2-1},\quad v_3':=\frac{u_3+1}{u_3-1}.
\]
Then $L(u_1,u_2,u_3)=L(v_1',v_2',v_3')$ and the action of $\rho$ on $L(v_1',v_2',v_3')$ is given 
by $\rho\,:\, \sqrt{a}\mapsto -\sqrt{a}$, $v_i'\mapsto -v_i'$ for $1\leq i\leq 3$. 
Therefore we get 
\[
K(x,y,z)^{\langle\ta_1,\la_1\rangle}=L(v_1',v_2',v_3')^{\langle\rho\rangle}
=L^{\langle\rho\rangle}(\sqrt{a}\,v_1',\sqrt{a}\,v_2',\sqrt{a}\,v_3')
=K(\sqrt{a}\,v_1',\sqrt{a}\,v_2',\sqrt{a}\,v_3').
\]
It is easy to see that $\sqrt{a}v_i'=v_i$ for $1\leq i\leq 3$. 
\end{proof}
%

%
%
\section{The case of $G_{2,j,k}$}\label{se2}
%

We treat the following six groups 
$G_{2,j,k}$ of the $2$nd crystal system in dimension $3$. 
They are abelian groups of exponent $2$ which are defined as 
\begin{align*}
G_{2,1,1}&=\langle \la_1\rangle\cong\mathcal{C}_2,& 
G_{2,1,2}&=\langle -\al\rangle\cong\mathcal{C}_2,\\
G_{2,2,1}&=\langle -\la_1\rangle\cong\mathcal{C}_2,& 
G_{2,2,2}&=\langle \al\rangle\cong\mathcal{C}_2,\\
\mathcal{N}\ni
G_{2,3,1}&=\langle \la_1,-I_3\rangle\cong\mathcal{C}_2\times \mathcal{C}_2,&
G_{2,3,2}&=\langle -\al,-I_3\rangle\cong\mathcal{C}_2\times \mathcal{C}_2.
\end{align*}
The rationality problem for $G_{2,3,1}\in\mathcal{N}$ is already solved 
by Yamasaki \cite{Yam} (see Section \ref{seintro}). 
The actions of $\la_1$, $-\al$, $-\la_1$ and $\al$ on $K(x,y,z)$ are given by 
\begin{align*} 
\la_1 &: x\ \mapsto\ \frac{a}{x},\ y\ \mapsto\ by,\ z\ \mapsto\ \frac{c}{z},& 
-\al &: x\ \mapsto\ \frac{d}{y},\ y\ \mapsto\ \frac{e}{x},\ z\ \mapsto\ \frac{f}{z},\\ 
-\la_1 &: x\ \mapsto\ gx,\ y\ \mapsto\ \frac{h}{y},\ z\ \mapsto\ iz,& 
\al &: x\ \mapsto\ jy,\ y\ \mapsto\ kx,\ z\ \mapsto\ lz. 
\end{align*} 
We may assume that $j=k=1$ by replacing $jy$ by $y$ and the other coefficients. 

By the equalities $\la_1^2=(-\al)^2=(-\la_1)^2=\al^2=I_3$, 
we have $b^2=g^2=i^2=l^2=1$ and $d=e$. 
Hence the problem may be reduced to the following cases: 
\begin{align*} 
\la_1 &: x\ \mapsto\ \frac{a}{x},\ y\ \mapsto\ \ep_1 y,\ z\ \mapsto\ \frac{c}{z},& 
-\al &: x\ \mapsto\ \frac{d}{y},\ y\ \mapsto\ \frac{d}{x},\ z\ \mapsto\ \frac{f}{z},\\ 
-\la_1 &: x\ \mapsto\ \ep_2 x,\ y\ \mapsto\ \frac{h}{y},\ z\ \mapsto\ \ep_3 z,& 
\al &: x\ \mapsto\ y,\ y\ \mapsto\ x,\ z\ \mapsto\ \ep_4 z
\end{align*} 
where $a,c,d,f,h\in K^\times$ and $\ep_1,\ep_2,\ep_3,\ep_4=\pm 1$. 

By Theorem \ref{thHaj} and Theorem \ref{thAHK}, we see that 
$K(x,y,z)^{G_{2,1,1}}$ $=$ $K(x,y,z)^{\langle\la_1\rangle}$, 
$K(x,y,z)^{G_{2,2,1}}$ $=$ $K(x,y,z)^{\langle-\la_1\rangle}$ and 
$K(x,y,z)^{G_{2,2,2}}$ $=$ $K(x,y,z)^{\langle\al\rangle}$ are rational over $K$. 
Replacing $d/y$ by $y$, we also see that 
$K(x,y,z)^{G_{2,1,2}}=K(x,y,z)^{\langle-\al\rangle}$ is rational over $K$. 
For 
\begin{align*}
G_{2,3,2}=\langle -\al,-I_3\rangle=\langle -\al,\al\rangle, 
\end{align*}
we put 
\begin{align*}
t_1:=\frac{x-y}{x+y},\quad t_2:=\frac{2}{x+y},\quad t_3:=z.
\end{align*}
Then the generators $-\al$ and $\al$ of $G_{2,3,2}$ 
act on $K(t_1,t_2,t_3)=K(x,y,z)$ respectively by 
\begin{align*} 
-\al : t_1\ \mapsto\ t_1,\ t_2\ \mapsto\ \frac{1-t_1^2}{dt_2},\ t_3\ \mapsto\ \frac{f}{t_3},
\quad \al : t_1\ \mapsto\ -t_1,\ t_2\ \mapsto\ t_2,\ t_3\ \mapsto\ \ep_4 t_3. 
\end{align*}
We take 
\begin{align*}
\begin{cases}
u_1:=t_1^2,\quad u_2:=t_2,\quad u_3:=t_3,\ \mathrm{if}\ \ep_4=1,\\
u_1:=t_1^2,\quad u_2:=t_2,\quad u_3:=t_1t_3,\ \mathrm{if}\ \ep_4=-1.
\end{cases}
\end{align*}
Then $K(t_1,t_2,t_3)^{\langle\al\rangle}=K(u_1,u_2,u_3)$ and the action of $-\al$ on 
$K(u_1,u_2,u_3)$ is given by 
\begin{align*} 
-\al : 
\begin{cases}
\displaystyle{
u_1\ \mapsto\ u_1,\ u_2\ \mapsto\ \frac{1-u_1}{du_2},\ u_3\ \mapsto\ \frac{f}{u_3}},\ 
\mathrm{if}\ \ep_4=1,\vspace*{1mm}\\
\displaystyle{
u_1\ \mapsto\ u_1,\ u_2\ \mapsto\ \frac{1-u_1}{du_2},\ u_3\ \mapsto\ \frac{fu_1}{u_3},\ 
\mathrm{if}\ \ep_4=-1}.
\end{cases}
\end{align*}
Hence it follows from Theorem \ref{thab} that 
$K(x,y,z)^{G_{2,3,2}}=K(u_1,u_2,u_3)^{\langle-\al\rangle}$ is rational over $K$.

%
\section{The case of $G_{3,j,k}$}\label{se3}
%

We treat the following $13$ groups $G=G_{3,1,k},\ 1\le k\le 4$, 
$G=G_{3,2,k},\ 1\le k\le 5$, and $G=G_{3,3,k},\ 1\le k\le 4$ of the 3rd 
crystal system in dimension 3:  
\begin{align*}
\mathcal{N}\ni
G_{3,1,1}&=\langle \ta_1,\la_1\rangle\cong\mathcal{C}_2\times \mathcal{C}_2,& 
G_{3,1,2}&=\langle \ta_1,-\al\rangle\cong\mathcal{C}_2\times \mathcal{C}_2,\\
G_{3,1,3}&=\langle \ta_2,\la_2\rangle\cong\mathcal{C}_2\times \mathcal{C}_2,& 
G_{3,1,4}&=\langle \ta_3,\la_3\rangle\cong\mathcal{C}_2\times \mathcal{C}_2,\\
G_{3,2,1}&=\langle \ta_1,-\la_1\rangle\cong\mathcal{C}_2\times \mathcal{C}_2,&
G_{3,2,2}&=\langle \ta_1,\al\rangle\cong\mathcal{C}_2\times \mathcal{C}_2,\\  
G_{3,2,3}&=\langle -\al,\be\rangle\cong\mathcal{C}_2\times \mathcal{C}_2,&
G_{3,2,4}&=\langle \ta_2,-\la_2\rangle\cong\mathcal{C}_2\times \mathcal{C}_2,\\ 
G_{3,2,5}&=\langle \ta_3,-\la_3\rangle\cong\mathcal{C}_2\times \mathcal{C}_2,\\ 
\mathcal{N}\ni
G_{3,3,1}&=\langle \ta_1,\la_1,-I_3\rangle\cong\mathcal{C}_2\times 
\mathcal{C}_2\times \mathcal{C}_2, &
G_{3,3,2}&=\langle \ta_1,-\al,-I_3\rangle\cong\mathcal{C}_2\times 
\mathcal{C}_2\times \mathcal{C}_2, \\ 
G_{3,3,3}&=\langle \ta_2,\la_2,-I_3\rangle\cong\mathcal{C}_2\times 
\mathcal{C}_2\times \mathcal{C}_2,& 
G_{3,3,4}&=\langle \ta_3,\la_3,-I_3\rangle\cong\mathcal{C}_2\times 
\mathcal{C}_2\times \mathcal{C}_2. 
\end{align*} 
%
%
\subsection{The cases of $G_{3,2,1}$ and $G_{3,2,2}$}\label{subse321}
%

We treat the cases of 
\begin{align*}
G_{3,2,1} = \langle \ta_1,-\la_1\rangle, \quad 
G_{3,2,2} = \langle \ta_1,\al\rangle.
\end{align*}
The actions of $\ta_1$, $-\la_1$ and $\al$ on $K(x,y,z)$ are given by 
\begin{align*} 
\ta_1 &: x\ \mapsto\ \frac{a}{x},\ y\ \mapsto\ \frac{b}{y},\ z\ \mapsto\ cz,& 
-\la_1 &: x\ \mapsto\ dx,\ y\ \mapsto\ \frac{e}{y},\ z\ \mapsto\ fz,\\
\al &: x\ \mapsto\ gy,\ y\ \mapsto\ hx,\ z\ \mapsto\ iz. 
\end{align*} 
We may assume that $g=h=1$ by replacing $gy$ by $y$ and the other coefficients. 

By the equalities $\ta_1^2=(-\la_1)^2=\al^2=I_3$, 
we have $c^2=d^2=f^2=i^2=1$. 
Hence the problem may be reduced to the following cases: 
\begin{align*} 
\ta_1 &: x\ \mapsto\ \frac{a}{x},\ y\ \mapsto\ \frac{b}{y},\ z\ \mapsto\ \ep_1 z,& 
-\la_1 &: x\ \mapsto\ \ep_2 x,\ y\ \mapsto\ \frac{e}{y},\ z\ \mapsto\ \ep_3 z,\\
\al &: x\ \mapsto\ y,\ y\ \mapsto\ x,\ z\ \mapsto\ \ep_4 z
\end{align*} 
where $a,b,e\in K^\times$ and $\ep_1,\ep_2,\ep_3,\ep_4=\pm 1$. 

Using Theorem \ref{thAHK}, the rationality problems of $G_{3,2,1}$ and $G_{3,2,2}$ 
may be reduced to the $2$-dimensional cases of $K(x,y)^{G_{3,2,1}}$ and $K(x,y)^{G_{3,2,2}}$ 
respectively.
Hence $K(x,y,z)^{G_{3,2,1}}$ and $K(x,y,z)^{G_{3,2,2}}$ are rational over $K$. 

%
\subsection{The case of $G_{3,2,3}$}\label{subse323}
%

We consider the case of 
\[
G_{3,2,3}=\langle -\al,\be\rangle.
\]
The actions of $-\al$ and $\be$ on $K(x,y,z)$ are given by 
\begin{align*} 
-\al &: x\ \mapsto\ \frac{a}{y},\ y\ \mapsto\ \frac{b}{x},\ z\ \mapsto\ \frac{c}{z},& 
&\be : x\ \mapsto\ \frac{d}{y},\ y\ \mapsto\ \frac{e}{x},\ z\ \mapsto\ fz.
\end{align*} 
By the equalities $(-\al)^2=\be^2=I_3$ and $(-\al)\be=\be(-\al)$, 
we have $b=a$, $d=e=\ep_1 a$ and $f=\ep_2$ where $\ep_1,\ep_2=\pm 1$. 
Hence the actions of $-\al$ and $\be$ on $K(x,y,z)$ may be reduced to the following: 
\begin{align*} 
-\al &: x\ \mapsto\ \frac{a}{y},\ y\ \mapsto\ \frac{a}{x},\ z\ \mapsto\ \frac{c}{z},& 
&\be : x\ \mapsto\ \frac{\ep_1 a}{y},\ y\ \mapsto\ \frac{\ep_1 a}{x},\ z\ \mapsto\ \ep_2 z
\end{align*}
where $a,c\in K^\times$ and $\ep_1,\ep_2=\pm 1$. 
By putting $x':=x/y$, we have
\begin{align*} 
-\al &: x'\ \mapsto\ x',\ y\ \mapsto\ \frac{a}{x'y},\ z\ \mapsto\ \frac{c}{z},& 
&\be : x'\ \mapsto\ x',\ y\ \mapsto\ \frac{\ep_1 a}{x'y},\ z\ \mapsto\ \ep_2 z.
\end{align*}
Hence the problem of $G_{3,2,3}$ may be reduced to the $2$-dimensional case of 
$K(x')(y,z)^{G_{3,2,3}}$ and $K(x,y,z)^{G_{3,2,3}}$ is rational over $K$. 

%
\subsection{The cases of $G_{3,2,4}$ and $G_{3,2,5}$}\label{subse324}
%

We consider the cases of 
\[
G_{3,2,4}=\langle\ta_2,-\la_2\rangle,\quad 
G_{3,2,5}=\langle\ta_3,-\la_3\rangle.
\]
We may replace $G_{3,2,4}$ and $G_{3,2,5}$ by new representatives which contain $-\be$ as 
\[
G_{3,2,4}':=\langle -\be,-\varphi_2\rangle\in [G_{3,2,4}], \quad 
G_{3,2,5}':=\langle -\be,-\varphi_3\rangle\in [G_{3,2,5}]
\] 
with $-\be=R_1^{-1}\ta_2R_1=R_2^{-1}\ta_3R_2$ where 
\[
R_1:=\left[\begin{array}{ccc} 0&1&1\\1&0&-1\\-1&0&0 \end{array}\right],\quad 
R_2:={}^t(R_1^{-1})=\left[\begin{array}{ccc} 0&1&0\\0&1&-1\\-1&1&-1 \end{array}\right]
\] 
and 
\begin{align*} 
\varphi_2 &:= R_1^{-1}\la_2\ta_2R_1= 
\left[ \begin{array}{ccc} -1&0&1\\0&-1&-1\\0&0&1 \end{array} \right],\quad 
\varphi_3 := R_2^{-1}\la_3\ta_3R_2= 
\left[ \begin{array}{ccc} -1&0&0\\0&-1&0\\1&-1&1 \end{array} \right]. 
\end{align*}
The actions of $-\be$, $-\varphi_2$ and $-\varphi_3$ on $K(x,y,z)$ are given  by 
\begin{align*} 
-\be &: x\ \mapsto\ ay,\ y\ \mapsto\ bx,\ z\ \mapsto\ \frac{c}{z}, \\ 
-\varphi_2 &: x\ \mapsto\ dx,\ y\ \mapsto\ ey,\ z\ \mapsto\ \frac{fy}{xz}, & 
-\varphi_3 &: x\ \mapsto\ \frac{gx}{z},\ y\ \mapsto\ hyz,\ z\ \mapsto\ \frac{i}{z}. 
\end{align*} 
It follows from the equalities $(-\be)^2=(-\varphi_2)^2=I_3$ that $ab=d^2=1$ and $e=d$. 
Hence we may assume that the actions of 
$-\be$ and $-\varphi_2$ on $K(x,y,z)$ are given as 
\begin{align*} 
-\be &: x\ \mapsto\ ay,\ y\ \mapsto\ \frac{x}{a},\ z\ \mapsto\ \frac{c}{z}, & 
-\varphi_2 &: x\ \mapsto\ \ep x,\ y\ \mapsto\ \ep y,\ z\ \mapsto\ \frac{fy}{xz}
\end{align*}
where $\ep=\pm 1$. 
By putting $x':=x/y$, we have 
\begin{align*} 
-\be &: x'\ \mapsto\ \frac{a^2}{x'},\ y\ \mapsto\ \frac{x'y}{a}, z\ \mapsto\ \frac{c}{z},\\ 
-\varphi_2 &: x'\ \mapsto\ x',\ y\ \mapsto\ \ep y, z\ \mapsto\ \frac{f}{x'z},& 
-\varphi_3 &: x'\ \mapsto\ \frac{gx'}{hz^2},\ y\ \mapsto\ hyz,\ z\ \mapsto\ \frac{i}{z}. 
\end{align*} 
By Theorem \ref{thAHK}, the rationality problem of 
$G_{3,2,4}'=\langle -\be,-\varphi_2\rangle$ and 
$G_{3,2,5}'=\langle -\be,-\varphi_3\rangle$ may be reduced to the $2$-dimensional case
of $K(x',z)^{G_{3,2,4}'}$ and $K(x',z)^{G_{3,2,5}'}$ respectively. 
Hence $K(x,y,z)^{G_{3,2,4}}$ and $K(x,y,z)^{G_{3,2,5}}$ are rational over $K$. 

In Subsection \ref{subse4BM}, we will also give an explicit transcendental basis 
of $K(x,y,z)^{G_{3,2,5}}$ over $K$ in order to show the rationality of 
$K(x,y,z)^{G_{4,6,3}}$ over $K$. 

%
\subsection{The cases of $G_{3,1,2}$ and $G_{3,3,2}$} 
%

We treat the cases of 
\begin{align*}
G_{3,1,2}=\langle \ta_1, -\al\rangle ,\quad 
G_{3,3,2}&=\langle \ta_1,-\al,-I_3\rangle.
\end{align*} 
We first see 
\begin{align*}
G_{3,1,2}=\langle \ta_1, -\be\rangle ,\quad 
G_{3,3,2}&=\langle \ta_1,-\be, \al\rangle,
\end{align*}
because $\ta_1(-\al)=-\be$ and $(-\al)(-I_3)=\al$. 
The actions of $\ta_1$, $-\be$ and $\al$ on $K(x,y,z)$ are given by 
\begin{align*} 
\ta_1 &: x\ \mapsto\ \frac{a}{x},\ y\ \mapsto\ \frac{b}{y},\ z\ \mapsto\ cz,&  
-\be &: x\ \mapsto\ dy,\ y\ \mapsto\ ex,\ z\ \mapsto\ \frac{f}{z},\\
\al &: x\ \mapsto\ gy,\ y\ \mapsto\ hx,\ z\ \mapsto\ iz. 
\end{align*} 
We may assume that $d=e=1$ by replacing $dy$ by $y$ and the other coefficients. 

By the equalities $\ta_1^2=\al^2=[\ta_1,-\be]=[\ta_1,-I_3]=[-\be,-I_3]=I_3$, 
we have $a=b$, $gh=1$ and $c^2=g^2=i^2=1$. 
Hence the problem may be reduced to the following cases: 
\begin{align*} 
\ta_1 &: x\ \mapsto\ \frac{a}{x},\ y\ \mapsto\ \frac{a}{y},\ z\ \mapsto\ \ep_1 z,&  
-\be &: x\ \mapsto\ y,\ y\ \mapsto\ x,\ z\ \mapsto\ \frac{f}{z},\\
\al &: x\ \mapsto\ \ep_2 y,\ y\ \mapsto\ \ep_2 x,\ z\ \mapsto\ \ep_3 z 
\end{align*} 
where $a,f\in K^\times$ and $\ep_1,\ep_2,\ep_3=\pm 1$. 
Now we put 
\begin{align*}
t_1:=\frac{x-y}{x+y},\quad t_2:=\frac{2}{x+y},\quad t_3:=z.
\end{align*}
Then $K(x,y,z)=K(t_1,t_2,t_3)$ and the actions of $\ta_1$, $-\be$ and $\al$ on $K(t_1,t_2,t_3)$ 
are given by 
\begin{align*} 
\ta_1 &: t_1\ \mapsto\ -t_1,\ t_2\ \mapsto\ -\frac{t_1^2-1}{a t_2},\ 
t_3\ \mapsto\ \ep_1 t_3,&  
-\be &: t_1\ \mapsto\ -t_1,\ t_2\ \mapsto\ t_2,\ t_3\ \mapsto\ \frac{f}{t_3},\\
\al &: t_1\ \mapsto\ -t_1,\ t_2\ \mapsto\ \ep_2 t_2,\ t_3\ \mapsto\ \ep_3 t_3. 
\end{align*} 
By Theorem \ref{thInv} (I), we have $K(x,y,z)^{\langle -\be\rangle}
=K(t_1,t_2,t_3)^{\langle -\be\rangle}=K(u_1,u_2,u_3')$ where 
\begin{align*}
u_1:=t_3+\frac{f}{t_3},\quad u_2:=\Bigl(t_3-\frac{f}{t_3}\Bigr)\Big{/}t_1,\quad u_3':=t_2. 
\end{align*} 
We also put $u_3:=u_3'u_2$ then $K(x,y,z)^{\langle -\be\rangle}=K(u_1,u_2,u_3)$ 
and the actions of $\ta_1$ and $\al$ on $K(u_1,u_2,u_3)$ 
are given by 
\begin{align*} 
\ta_1 &: u_1\ \mapsto\ \ep_1 u_1,\ u_2\ \mapsto\ -\ep_1 u_2,\ 
u_3\ \mapsto\ \ep_1 \frac{u_1^2-u_2^2-4f}{au_3},\\
\al &: u_1\ \mapsto\ \ep_3 u_1,\ u_2\ \mapsto\ -\ep_3 u_2,\ u_3\ \mapsto\ -\ep_2\ep_3 u_3. 
\end{align*} 
Using Theorem \ref{thInv} (I), we have 
\begin{align*} 
K(x,y,z)^{G_{3,1,2}}=K(u_1,u_2,u_3)^{\langle\ta_1\rangle}=K(v_1,v_2,v_3) 
\end{align*} 
where 
\begin{align*} 
\begin{cases}
\displaystyle{v_1:=u_3+\frac{u_1^2-u_2^2-4f}{au_3},\ 
v_2:=\Bigl(u_3-\frac{u_1^2-u_2^2-4f}{au_3}\Bigr)\Big/u_2,\ v_3:=u_1},\ \mathrm{if}\ \ep_1=1,\\
\displaystyle{v_1:=u_3-\frac{u_1^2-u_2^2-4f}{au_3},\ 
v_2:=\Bigl(u_3+\frac{u_1^2-u_2^2-4f}{au_3}\Bigr)\Big/u_1,\ v_3:=u_2},\ \mathrm{if}\ \ep_1=-1. 
\end{cases}
\end{align*} 
Hence $K(x,y,z)^{G_{3,1,2}}=K(x,y,z)^{\langle\ta_1,-\be\rangle}=K(v_1,v_2,v_3)$ 
is rational over $K$. 
The action of $\al$ on $K(x,y,z)^{G_{3,1,2}}=K(v_1,v_2,v_3)$ is given by  
\begin{align*}
\al &: \begin{cases}
v_1\ \mapsto\ -\ep_2\ep_3 v_1,\ v_2\ \mapsto\ \ep_2 v_2,\ v_3\ \mapsto\ \ep_3 v_3,\ 
\mathrm{if}\ \ep_1=1,\\
v_1\ \mapsto\ -\ep_2\ep_3 v_1,\ v_2\ \mapsto\ -\ep_2v_2,\ v_3\ \mapsto\ -\ep_3 v_3,\ 
\mathrm{if}\ \ep_1=-1.
\end{cases}
\end{align*} 
Therefore $K(x,y,z)^{G_{3,3,2}}=K(v_1,v_2,v_3)^{\langle\al\rangle}$ is rational over $K$. 

%
\subsection{The cases of $G_{3,1,3}$ and $G_{3,3,3}$}\label{subse313}
%

We treat the cases of 
\begin{align*}
G_{3,1,3}=\langle \ta_2,\la_2\rangle ,\quad G_{3,3,3}&=\langle \ta_2,\la_2,-I_3\rangle .
\end{align*} 
We will use the following representatives: 
\begin{align*}
G_{3,1,3}'=\langle -\be,\varphi_2\rangle\in [G_{3,1,3}],\quad 
G_{3,3,3}'=\langle -\be,\varphi_2,\be\varphi_2\rangle\in [G_{3,3,3}]
\end{align*}
where $\varphi_2$ is given as in subsection \ref{subse324}. 

The monomial actions of $-\be$, $\varphi_2$ and $\be\varphi_2$ on $K(x,y,z)$ are given by 
\begin{align*} 
-\be &: x\ \mapsto\ ay,\ y\ \mapsto\  bx,\ z\ \mapsto\  \frac{c}{z},& 
\varphi_2 &: x\ \mapsto\ \frac{d}{x},\ y\ \mapsto\ \frac{e}{y},\ z\ \mapsto\  \frac{fxz}{y},\\ 
\be\varphi_2 &: x\ \mapsto\ gy,\ y\ \mapsto\ hx,\ z\ \mapsto\  \frac{ixz}{y}. 
\end{align*} 
We may assume that $a=b=1$ by replacing $ay$ by $y$ and the other coefficients. 
It follows from the equalities $[-\be,\varphi_2]=[-\be,\be\varphi_2]=[\varphi_2,\be\varphi_2]=I_3$ 
that $d=e$, $g=h$ and $f^2=h^2=i^2=1$. 
Hence the actions of $-\be$, $\varphi_2$ and $\be\varphi_2$ on $K(x,y,z)$ may be reduced to 
the case 
\begin{align*} 
-\be &: x\ \mapsto\ y,\ y\ \mapsto\  x,\ z\ \mapsto\ \frac{c}{z},& 
\varphi_2 &: x\ \mapsto\ \frac{d}{x},\ y\ \mapsto\ \frac{d}{y},\ z\ \mapsto\ \frac{\ep_1 xz}{y},\\
\be\varphi_2 &: x\ \mapsto\ \ep_2 y,\ y\ \mapsto\ \ep_2 x,\ z\ \mapsto\ \frac{\ep_3 xz}{y}
\end{align*} 
where $c,d\in K^\times$ and $\ep_1,\ep_2,\ep_3=\pm 1$. 

We first consider the fixed field $K(x,y,z)^{\langle \varphi_2\rangle}$. 
By applying Lemma \ref{lemaa}, we obtain 
\begin{align*} 
K(x,y,z)^{\langle \varphi_2\rangle}=K(t_1,t_2,t_3)
\end{align*} 
where 
\begin{align*} 
t_1:=\frac{xy+\ep_1 d}{x+\ep_1 y},\quad 
t_2:=\frac{xy-\ep_1 d}{x-\ep_1 y},\quad 
t_3:=\frac{t_1+\ep_1 t_2}{2}\Bigl(1+\ep_1 \frac{x}{y}\Bigr)z. 
\end{align*} 
Then the actions of $-\be$ and $\be\varphi_2$ on $K(t_1,t_2,t_3)$ are given by 
\begin{align*}
-\be &: t_1\ \mapsto\ \ep_1 t_1,\ t_2\ \mapsto\ -\ep_1 t_2,\ 
t_3\ \mapsto\ -\ep_1 \frac{c(t_2^2-d)}{t_3},\\
\be\varphi_2& : t_1\ \mapsto\ \ep_1\ep_2 t_1,\ t_2\ \mapsto\ -\ep_1\ep_2 t_2,\ t_3\mapsto\ 
\ep_2\ep_3\frac{t_1-\ep_1t_2}{t_1+\ep_1 t_2}t_3.
\end{align*} 
It follows from Theorem \ref{thInv} (I) that 
$K(x,y,z)^{G_{3,1,3}}=K(t_1,t_2,t_3)^{\langle -\be\rangle}$ is rational over $K$. 
Indeed we get $K(x,y,z)^{G_{3,1,3}}=K(u_1,u_2,u_3)$ where 
\begin{align*}
\begin{cases}
\ \displaystyle{
u_1:=t_3-\frac{c(t_2^2-d)}{t_3},\ u_2:=\Bigl(t_3+\frac{c(t_2^2-d)}{t_3}\Bigr)\Big{/}t_2,\ 
u_3:=t_1,\ \mathrm{if}\ \ep_1=1},\\
\ \displaystyle{
u_1:=t_3+\frac{c(t_2^2-d)}{t_3},\ u_2:=\Bigl(t_3-\frac{c(t_2^2-d)}{t_3}\Bigr)\Big{/}t_1,\ 
u_3:=t_2,\ \mathrm{if}\ \ep_1=-1}.
\end{cases}
\end{align*}

In order to get $K(x,y,z)^{G_{3,3,3}}=K(u_1,u_2,u_3)^{\langle\be\varphi_2\rangle}$, we take 
\begin{align*}
v_1:=u_1-\ep_1u_2u_3,\ v_2:=u_1^2-(u_2^2-\ep_14c)u_3^2-\ep_14cd,\ v_3:=u_3.
\end{align*}
Then we have $K(x,y,z)^{G_{3,1,3}}=K(u_1,u_2,u_3)=K(v_1,v_2,v_3)$ because 
\[
u_1=\frac{v_1^2+v_2-\ep_14c(v_3^2-d)}{2v_1},\quad 
u_2=\frac{-\ep_1(v_1^2-v_2)-4c(v_3^2-d)}{2v_1v_3},\quad 
u_3=v_3.
\]
The action of $\be\varphi_2$ on $K(x,y,z)^{G_{3,1,3}}=K(v_1,v_2,v_3)$ may be evaluated as 
\begin{align*}
\be\varphi_2 : 
\begin{cases}
\ \displaystyle{v_1\ \mapsto\ -\ep_2\ep_3 v_1,\ 
v_2\ \mapsto\ \frac{v_1^4-8c(v_3^2+d)v_1^2+16c^2(v_3^2-d)^2}{v_2},\ 
v_3\mapsto\ \ep_2 v_3,\ \mathrm{if}\ \ep_1=1},\vspace*{1mm}\\
\ \displaystyle{v_1\ \mapsto\ \ep_2\ep_3 v_1,\ v_2\ \mapsto\ 
\frac{(v_1^2-4c(v_3^2-d))^2}{v_2},\ v_3\mapsto\ \ep_2 v_3,\ \mathrm{if}\ \ep_1=-1}.
\end{cases}
\end{align*}

When $\ep_1=-1$, if we put 
\[
v_2':=\frac{v_1^2-4c(v_3^2-d)+v_2}{v_1^2-4c(v_3^2-d)-v_2}
\] 
then the action of $\be\varphi_2$ on $K(x,y,z)^{G_{3,1,3}}=K(v_1,v_2',v_3)$ is given by 
\[
\be\varphi_2\,:\, v_1\ \mapsto\ \ep_2\ep_3 v_1,\ v_2'\ \mapsto\ -v_2',\ v_3\mapsto\ \ep_2 v_3.
\] 
Hence $K(x,y,z)^{G_{3,3,3}}=K(v_1,v_2',v_3)^{\langle\be\varphi_2\rangle}$ is rational over $K$. 

When $\ep_1=1$, we note that the numerator of $\be\varphi_2(v_2)$ is a monic quartic polynomial 
in $v_1$ and is also a quartic polynomial in $v_3$ with square constant term:  
\[
v_1^4-8c(v_3^2+d)v_1^2+16c^2(v_3^2-d)^2=16c^2v_3^4-8c(v_1^2+4cd)v_3^2+(v_1^2-4cd)^2.
\]
By applying Theorem \ref{thInv} (II), we conclude that 
$K(x,y,z)^{G_{3,3,3}}=K(v_1,v_2,v_3)^{\langle\be\varphi_2\rangle}$ is rational over $K$. 

%
\subsection{The cases of $G_{3,1,4}$ and $G_{3,3,4}$}\label{subse314}
%

We treat the cases of 
\begin{align*}
G_{3,1,4}=\langle \la_3,\ta_3\rangle,\quad G_{3,3,4}&=\langle \la_3,\ta_3,-I_3\rangle.
\end{align*} 
The monomial actions of $\la_3,\ta_3$ and $-I_3$ on $K(x,y,z)$ are given by 
\begin{align*} 
\la_3 &: x\ \mapsto\ az,\ y\ \mapsto\ \frac{b}{xyz},\ z\ \mapsto\  cx, &
\ta_3 &: x\ \mapsto\ dy,\ y\ \mapsto\ ex,\ z\ \mapsto\  \frac{f}{xyz}, \\ 
-I_3 &: x\ \mapsto\  \frac{g}{x},\ y\ \mapsto\  \frac{h}{y},\ z\ \mapsto\  \frac{i}{z}.
\end{align*} 
We may assume that $a=d=1$ by replacing $(dy,az)$ by $(y,z)$ and the other coefficients. 
By the equality $\la_3^2=\ta_3^2=I_3$ (resp. $[\la_3,\ta_3]=I_3)$, 
we see $c=e=1$ (resp. $b=f$). 
For $G_{3,3,4}$, it follows from $[\la_3,-I_3]=[\ta_3,-I_3]=I_3$ that 
$g=h=i$ and $b^2=g^4$. 
So the actions of $\la_3,\ta_3$ and $-I_3$ on $K(x,y,z)$ may be reduced to the following cases: 
\begin{align*} 
\la_3 &: x\ \mapsto\ z,\ y\ \mapsto\ \frac{b}{xyz},\ z\ \mapsto\ x, &
\ta_3 &: x\ \mapsto\ y,\ y\ \mapsto\ x,\ z\ \mapsto\ \frac{b}{xyz},\\ 
-I_3 &: x\ \mapsto\ \frac{g}{x},\ y\ \mapsto\ \frac{g}{y},\ z\ \mapsto\ \frac{g}{z}, 
\end{align*} 
with $b=\ep\cdot g^2\in K^\times$ and $\ep=\pm 1$ for the case of $G_{3,3,4}$. 
Put 
\begin{align*}
w:=\frac{b}{xyz}=\frac{\ep\cdot g^2}{xyz}.
\end{align*}
Then $K(x,y,z)=K(x,y,z,w)$ and the actions of $\ta_3$, $\la_3$ and $-I_3$ on $K(x,y,z,w)$ are: 
\begin{align*}
\ta_3 &: x\ \mapsto\ y,\ y\ \mapsto\ x,\ z\ \mapsto\ w,\ w\ \mapsto\ z,\\
\la_3 &: x\ \mapsto\ z,\ y\ \mapsto\ w,\ z\ \mapsto\ x,\ w\ \mapsto\ y,\\
-I_3 &: x\ \mapsto\ \frac{g}{x},\ y\ \mapsto\ \frac{g}{y},\ z\ 
\mapsto\ \frac{g}{z},\ w\ \mapsto\ \frac{g}{w}.
\end{align*}
By Lemma \ref{lemV42}, we have $K(x,y,z)^{G_{3,1,4}}
=K(x,y,z,w)^{\langle \ta_3,\la_3\rangle}=K(v_1,v_2,v_3)$ where
\begin{align*}
v_1\ :=\ \frac{x+y-z-w}{xy-zw},\quad
v_2\ :=\ \frac{x-y-z+w}{xw-yz},\quad
v_3\ :=\ \frac{x-y+z-w}{xz-yw}
\end{align*}
with $w=b/(xyz)$. 
Thus $K(x,y,z)^{G_{3,1,4}}$ is rational over $K$. 
The action of $-I_3$ on $K(x,y,z)^{\langle \ta_3,\la_3\rangle}=K(v_1,v_2,v_3)$ is given by 
\begin{align*}
-I_3 &: v_1\ \mapsto\ \frac{-v_1+v_2+v_3}{g\, v_2v_3},\ v_2\ \mapsto\ 
\frac{v_1-v_2+v_3}{g\, v_1v_3},\ v_3\ \mapsto\ \frac{v_1+v_2-v_3}{g\, v_1v_2}.
\end{align*}
We see that $K(x,y,z)^{G_{3,3,4}}=K(x,y,z)^{\langle\ta_3,\la_3,-I_3\rangle}$ 
is rational over $K$ by the following lemma: 
\begin{lemma}\label{lemtlm}
We have $K(x,y,z)^{\langle\ta_3,\la_3,-I_3\rangle}=K(t_1,t_2,t_3)$ where 
\begin{align*}
t_1&:=v_1+(-I_3)(v_1)=\frac{-v_1+v_2+v_3+g\,v_1v_2v_3}{g\,v_2v_3}\\
&=\frac{g(x+y-z-w)+xy(z+w)-zw(x+y)}{g(xy-zw)},\\
t_2&:=v_2+(-I_3)(v_2),\ t_3:=v_3+(-I_3)(v_3)\ \ \mathit{with}\ \ w=b/(xyz).
\end{align*}
\end{lemma}
\begin{proof}
We omit to explain how to get suitable generators $t_1$, $t_2$ and $t_3$, 
and only verify the assertion. 
From the definition of the $t_i$'s, we see
$K(x,y,z)^{\langle\ta_3,\la_3,-I_3\rangle}\supset K(t_1,t_2,t_3)$. 
We also get 
\begin{align*}
&g(g\,t_1t_2-4)(v_3^2-t_3v_3)-4g\,t_3(t_1+t_2-t_3)=0,\\
&v_2=\frac{v_3(g\,t_1t_2-4)-2(t_2-t_3)}{g\,t_1t_3-4},\quad 
v_1=\frac{v_3^2(g\,t_1t_2-2)+v_3(t_3-g\,t_1t_3v_2)+t_3v_2}{t_3-2v_3}.
\end{align*}
The assertion follows from $K(v_1,v_2,v_3)=K(t_1,t_2,t_3)(v_3)$ and 
$[K(v_1,v_2,v_3):K(t_1,t_2,t_3)]\leq 2$.
\end{proof}
%


%
%
\section{The case (4A): $\pm \caa\in G_{4,j,k}$}\label{se4A}

We treat the following eight groups $G=G_{4,j,1}$\ ($1\le j\le 7$) and $G=G_{4,6,2}$ of 
the 4th crystal system in dimension 3 which contain $\caa$ or $-\caa$:  
\begin{align*}
G_{4,1,1}&=\langle \caa\rangle\cong\mathcal{C}_4,& 
\hspace*{-5mm}\mathcal{N}\ni
G_{4,2,1}&=\langle -\caa\rangle\cong\mathcal{C}_4,\\
\mathcal{N}\ni
G_{4,3,1}&=\langle \caa,-I_3\rangle\cong\mathcal{C}_4\times \mathcal{C}_2,\\
\mathcal{N}\ni
G_{4,4,1}&=\langle \caa,\la_1\rangle\cong\mathcal{D}_4,&
G_{4,5,1}&=\langle \caa,-\la_1\rangle\cong\mathcal{D}_4,\\
G_{4,6,1}&=\langle -\caa,\la_1\rangle\cong\mathcal{D}_4, &  
G_{4,6,2}&=\langle -\caa,-\la_1\rangle\cong\mathcal{D}_4, \\  
G_{4,7,1}&=\langle \caa,\la_1,-I_3\rangle\cong\mathcal{D}_4\times \mathcal{C}_2.
\end{align*} 

%
\subsection{The case (4A$^+$): $\caa\in G_{4,j,k}$} 
%

We treat the cases of 
\begin{align*}
G_{4,1,1}=\langle \caa\rangle,\quad G_{4,5,1}=\langle \caa,-\la_1\rangle
\end{align*}
which have a normal subgroup $\langle\caa\rangle$ of order $4$. 
We also see 
\[
G_{4,5,1}=\langle \caa,\al\rangle
\]
where $\al=\caa(-\la_1)$ is given as in the previous section. 
By the equalities $\caa^4=\al^2=I_3$ and $[\caa,\al]=\caa^2$, 
we may assume that the actions of $\caa$ and of $\al$ on $K(x,y,z)$ are given by 
\begin{align*}
\caa &: x\ \mapsto\ y,\ y\ \mapsto\ \frac{b}{x},\ z\ \mapsto\ cz,& 
\al &: x\ \mapsto\ \ep_2y,\ y\ \mapsto\ \ep_2x,\ z\ \mapsto\ \ep_3z 
\end{align*}
where $b,c\in K^\times$ with $c^4=1$ and $\ep_2,\ep_3=\pm 1$. 
For $G_{4,5,1}$, we also see $c=\ep_1=\pm 1$. 
By Theorem \ref{thAHK}, the rationality problems for $G=G_{4,1,1}$, $G_{4,5,1}$ 
may be reduced to the $2$-dimensional monomial case $K(x,y)^G(\Theta)$ 
and hence $K(x,y,z)^G$ is rational over $K$. 
Indeed we may obtain a $G_{4,5,1}$-invariant $\Theta$ which satisfies $K(x,y,z)=K(x,y,\Theta)$. 
For example, if $(c,\ep_2,\ep_3)=(-1,-1,-1)$ then we obtain
\[
\Theta=\frac{(x+y)(xy+b)z}{(x^2-b)(y^2-b)},
\]
and if $(c,\ep_2,\ep_3)=(-1,1,-1)$ then we get 
\[
\Theta=\mathrm{Tr}_{G_{4,1,1}}(xz)=xz-yz+\frac{bz}{x}-\frac{bz}{y}=-\frac{(x-y)(b-xy)z}{xy}
\]
where $\mathrm{Tr}_{G_{4,1,1}}$ is the trace map under the action of $G_{4,1,1}$. 

For the convenience of the readers, we shall give the generators of $K(x,y)^G$. 
The action of $\caa^2=\ta_1$ on $K(x,y)$ is 
\begin{align*}
\ta_1 &: x\ \mapsto\ \frac{b}{x},\ y\ \mapsto\ \frac{b}{y},
\end{align*}
so that by Lemma \ref{lemaa} we have $K(x,y)^{\langle\ta_1\rangle}=K(t_1,t_2)$ where 
\begin{align*}
t_1:=\frac{xy+b}{x+y},\quad t_2:=\frac{xy-b}{x-y}. 
\end{align*}
The actions of $\caa$ and $\al$ on $K(x,y)^{\langle\ta_1\rangle}=K(t_1,t_2)$ are given by 
\begin{align*}
\caa : t_1\ \mapsto\ \frac{b}{t_1},\ t_2\ \mapsto\ -\frac{b}{t_2},\quad 
\al : t_1\ \mapsto\ \ep_2 t_1,\ t_2\ \mapsto\ -\ep_2 t_2. 
\end{align*}
Thus, by Theorem \ref{thab}, we get $K(x,y)^{G_{4,1,1}}=K(t_1,t_2)^{\langle\caa\rangle}
=K(u_1,u_2)$ where 
\begin{align*}
u_1:=\frac{t_2(t_1^2-b)}{t_1^2t_2^2+b^2},\quad 
u_2:=\frac{t_2(t_1^2+b)}{t_1^2t_2^2+b^2}.
\end{align*}
Because the action of $\al$ on $K(u_1,u_2)$ is given by 
\[
\al\,:\,u_1\ \mapsto\ -\ep_2u_1,\ u_2\ \mapsto\ \ep_2u_2,
\]
we get $K(x,y)^{G_{4,5,1}}=K(u_1,u_2)^{\langle\al\rangle}=K(u_1^2,u_2)$ (resp. $K(u_1,u_2^2)$) 
when $\ep_2=1$ (resp. $\ep_2=-1$). 

%
\subsection{The case (4A$^-$): $-\caa\in G_{4,j,k}$}
%

We treat the cases of 
\begin{align*}
G_{4,6,1} = \langle -\caa,\la_1\rangle,\quad 
G_{4,6,2} = \langle -\caa,-\la_1\rangle,\quad 
G_{4,7,1} = \langle \caa,\la_1,-I_3\rangle
\end{align*}
which have a normal subgroup $\langle-\caa\rangle$ of order $4$. 
We change the generators of the groups by 
\[
G_{4,6,1} = \langle -\caa,\al\rangle,\quad 
G_{4,6,2} = \langle -\caa,-\be\rangle,\quad 
G_{4,7,1} = \langle -\caa,\al,-\be\rangle
\]
where $\al$ and $-\be$ are given as in the previous section which satisfy
\begin{align*}
\al&=(-\caa)\la_1
=\left[\begin{array}{ccc} 0 & 1 & 0\\ 1 & 0 & 0\\ 0 & 0 & 1\end{array}\right], &
-\be&=(-\la_1)(-\caa)
=\left[\begin{array}{ccc} 0 & 1 & 0\\ 1 & 0 & 0\\ 0 & 0 & -1\end{array}\right].
\end{align*}
The monomial actions of $-\caa$, $\al$ and $-\be$ on $K(x,y,z)$ are given by 
\begin{align*} 
-\caa &: x\ \mapsto\ \frac{a}{y},\ y\ \mapsto\ bx,\ z\ \mapsto\ \frac{c}{z},\\
\al &: x\ \mapsto\ dy,\ y\ \mapsto\ ex,\ z\ \mapsto\ fz,& 
-\be &: x \mapsto\ gy,\ y\ \mapsto\ hx,\ z\ \mapsto\ \frac{i}{z}. 
\end{align*} 
We may assume that $b=1$ by replacing $bx$ by $x$ and the other coefficients. 
By the equality $\al^2=(-\be)^2=I_3$, we get $de=f^2=gh=1$. 
It also follows from $[-\caa,\al]=[-\caa,-\be]=\caa^2$ and $[\al,-\be]=I_3$ that 
the actions of $-\caa$, $\al$ and $-\be$ may be reduced to the following case: 
\begin{align*} 
-\caa &: x\ \mapsto\ \frac{a}{y},\ y\ \mapsto\ x,\ z\ \mapsto\ \frac{c}{z},\\
\al &: x\ \mapsto \ep_1 y,\ y\ \mapsto\ \ep_1x,\ z\ \mapsto\ \ep_2 z,& 
-\be &: x\ \mapsto\ \ep_3 y,\ y\ \mapsto\ \ep_3 x,\ z\ \mapsto\ \frac{\ep_4 c}{z}
\end{align*} 
where $a,c\in K^{\times}$ and $\ep_1,\ep_2,\ep_3,\ep_4=\pm 1$. 

Note that $(-\caa)^2=\ta_1$ where $\ta_1$ is given as in the previous section. 
The action of $\ta_1$ on $K(x,y,z)$ is given by 
\begin{align*}
\ta_1 &: x\ \mapsto\ \frac{a}{x},\ y\ \mapsto\ \frac{a}{y},\ z\ \mapsto\ z. 
\end{align*}
Although the groups $G_{4,6,1}$, $G_{4,6,2}$ and $G_{4,7,1}$ have a normal subgroup 
$\langle-\caa\rangle$, we do not consider $K(x,y,z)^{\langle -\caa\rangle}$ but 
$K(x,y,z)^{\langle (-\caa)^2\rangle}=K(x,y,z)^{\langle \ta_1\rangle}$ because 
$\langle-\caa\rangle=G_{4,2,1}$ is in $\mathcal{N}$. 
By Lemma \ref{lemaa}, we have 
\begin{align*} 
K(x,y,z)^{\langle \ta_1\rangle }=K(u_1,u_2,u_3)
\end{align*} 
where 
\begin{align*}
u_1:=\frac{xy+a}{x+y},\quad  u_2:=\frac{xy-a}{x-y},\quad  u_3:=z. 
\end{align*} 
The actions of $-\caa$, $\al$ and $-\be$ on $K(u_1,u_2,u_3)$ are 
\begin{align*} 
-\caa &: u_1\ \mapsto\ \frac{a}{u_1},\ u_2\ \mapsto\ -\frac{a}{u_2},\ 
u_3\ \mapsto\ \frac{c}{u_3},\\
\al &: u_1\ \mapsto\ \ep_1 u_1,\ u_2\ \mapsto\ -\ep_1 u_2,\ u_3\ \mapsto\ \ep_2 u_3,& 
-\be &: u_1\ \mapsto\ \ep_3 u_1,\ u_2\ \mapsto\ -\ep_3 u_2,\ u_3\ \mapsto\ \frac{\ep_4 c}{u_3}. 
\end{align*} 

For $G_{4,6,1}=\langle -\caa,\al\rangle$ and  $G_{4,7,1} = \langle -\caa,\al,-\be\rangle$, 
we consider the fixed field 
\[
K(x,y,z)^{\langle \ta_1,\al\rangle}
=K(u_1,u_2,u_3)^{\langle \al\rangle }=K(v_1,v_2,v_3)
\]
where $v_1,v_2,v_3$ are given by the following table: 

\begin{center} 
\begin{tabular}{c|ccc} 
$(\ep_1,\ep_2)$ & $v_1$ & $v_2$ & $v_3$\\ \hline 
$(1,1)$ & $u_2^2/a$ & $u_1$  & $u_3$\\ 
$(-1,1)$ & $u_1^2/a$ & $u_2$ & $u_3$\\ 
$(1,-1)$ & $cu_2/u_3$ & $u_1$ & $u_2u_3$\\ 
$(-1,-1)$ & $cu_1/u_3$ & $u_2$ & $u_1u_3$
\end{tabular} 
\end{center}

Thus we have $K(x,y,z)^{G_{4,6,1}}=K(v_1,v_2,v_3)^{\langle-\caa\rangle}$ and 
$K(x,y,z)^{G_{4,7,1}}=K(v_1,v_2,v_3)^{\langle-\caa,-\be\rangle}$. 
The actions of $-\caa$ and of $-\be$ on 
$K(x,y,z)^{\langle\ta_1,\al\rangle}=K(v_1,v_2,v_3)$ are given as follows: 

\begin{center} 
\begin{tabular}{c|ccc|ccc} 
$(\ep_1,\ep_2)$ & $-\caa(v_1)$ & $-\caa(v_2)$ & $-\caa(v_3)$ 
& $-\be(v_1)$ & $-\be(v_2)$ & $-\be(v_3)$ \\ \hline 
$(1,1)$ & $1/v_1$ & $a/v_2$ & $c/v_3$ & $v_1$ & $\ep_3 v_2$ & $\ep_4 c/v_3$ \\ 
$(-1,1)$ & $1/v_1$ & $-a/v_2$ & $c/v_3$& $v_1$ & $-\ep_3 v_2$ & $\ep_4 c/v_3$ \\ 
$(1,-1)$ & $-ac/v_1$ & $a/v_2$ & $-ac/v_3$ & $-\ep_3\ep_4v_3$ & $\ep_3 v_2$ & $-\ep_3\ep_4v_1$ \\ 
$(-1,-1)$ & $ac/v_1$ & $-a/v_2$ & $ac/v_3$ & $\ep_3\ep_4v_3$ & $-\ep_3v_2$ & $\ep_3\ep_4v_1$ 
\end{tabular} 
\end{center} 

When $\ep_2=1$, we put 
\[
v_1':=\frac{v_1+1}{v_1-1}
\]
then $K(v_1,v_2,v_3)=K(v_1',v_2,v_3)$ and $-\caa(v_1')=-v_1'$ and $-\be(v_1')=v_1'$. 
Hence, by Theorem \ref{thAHK}, the rationality problem of $G_{4,6,1}$ and of $G_{4,7,1}$ 
may be reduced to the $2$-dimensional cases $K(v_2,v_3)^{G_{4,6,1}}$ and 
$K(v_2,v_3)^{G_{4,7,1}}$ under monomial actions respectively. 

When $\ep_2=-1$, the action of $\langle-\caa\rangle$ (resp. $\langle-\caa,-\be\rangle$) 
on $K(v_1,v_2,v_3)$ is given as an affirmative case of $G_{1,2,1}=\langle-I_3\rangle$ 
(resp. a case of $G_{2,3,2}=\langle-\al,-I_3\rangle=\langle\al,-I_3\rangle$). 

We conclude that, in both cases $\ep_2=\pm 1$, the fixed fields 
$K(x,y,z)^{G_{4,6,1}}=K(v_1,v_2,v_3)^{\langle-\caa\rangle}$ 
and $K(x,y,z)^{G_{4,7,1}}=K(v_1,v_2,v_3)^{\langle-\caa,-\be\rangle}$ are rational over $K$. 
We may also get explicit transcendental bases of the fixed fields by using Lemma \ref{lemaa}. 

\begin{remark}
(i) When $\ep_2=-1$, the action of $G_{4,6,1}$ on $K(v_1,v_2,v_3)$ is a special case of 
monomial actions of $G_{1,2,1}\in\mathcal{N}$, 
which is an affirmative case of Saltman's result (cf. Theorem \ref{thSalt}).\\
(ii) When $\ep_2=1$, the action of $G_{4,7,1}$ on $K(v_1,v_2,v_3)$ 
is the special case of $G_{2,3,1}\in\mathcal{N}$, which is an affirmative case as seen 
in \cite[Theorem 6]{Yam} (see also Theorem \ref{th231}), 
although in the case of $\ep_2=-1$, the action of $G_{4,7,1}$ on $K(v_1,v_2,v_3)$ is 
equivalent to a monomial action of $G_{2,3,2}=\langle-\al,-I_3\rangle$. 
\end{remark}

For $G_{4,6,2}=\langle -\caa,-\be\rangle$, we first note that 
\[
G_{4,6,2}=\langle -\be,-\la_1\rangle
\]
because $(-\caa)(-\be)=(-\la_1)$. 
The actions of $-\be$ and of $-\la_1$ on $K(x,y,z)^{\langle\ta_1\rangle}=K(u_1,u_2,u_3)$ 
are given by 
\begin{align*} 
-\be &: u_1\ \mapsto\ \ep_3 u_1,\ u_2\ \mapsto\ -\ep_3 u_2,\ u_3\ \mapsto\ \frac{\ep_4 c}{u_3},& 
-\la_1 &: u_1\ \mapsto\ \frac{\ep_3 a}{u_1},\ u_2\ \mapsto\ \frac{\ep_3 a}{u_2},\ 
u_3\ \mapsto\ \ep_4 u_3. 
\end{align*} 
Hence the rationality problem of $G_{4,6,2}$ may be reduced to the problem of 
$\langle-\ta_1,\ta_1\rangle\in\mathcal{N}$ which is conjugate to 
$G_{2,3,1}=\langle\la_1,-I_3\rangle$ in $\mathrm{GL}(3,\bZ)$. 

However this is an affirmative case of $G_{2,3,1}$. 
We shall give the transcendental basis over $K$ as follows. 
When $\ep_4=1$, by Theorem \ref{th231}, 
$K(x,y,z)^{G_{4,6,2}}$ is rational over $K$ and we get an explicit transcendental basis over $K$. 

When $\ep_4=-1$, we also get an explicit transcendental basis of $K(x,y,z)^{G_{4,6,2}}$ 
over $K$ as follows: 
First we may assume that $\ep_3=1$ by interchanging $u_1$, $u_2$ and $a$, $-a$. 
Put 
\[
u_1':=u_3,\quad u_2':=\frac{u_1+u_2}{u_1-u_2},\quad u_3':=u_1,\quad 
\Bigl(u_2=\frac{(u_2'-1)u_3'}{u_2'+1}\Bigr).
\]
Then $K(x,y,z)^{\langle\ta_1\rangle}=K(u_1,u_2,u_3)=K(u_1',u_2',u_3')$ and 
the actions of $-\la_1$ and of $-\be$ on $K(u_1',u_2',u_3')$ are given by 
\begin{align*}
-\la_1 &: u_1'\ \mapsto\ -u_1',\ u_2'\ \mapsto\ -u_2',\ u_3'\ \mapsto\ \frac{a}{u_3'},& 
-\be &: u_1'\ \mapsto\ -\frac{c}{u_1'},\ u_2'\ \mapsto\ \frac{1}{u_2'},\ u_3'\ \mapsto\ u_3'.
\end{align*}
Therefore we can also apply Theorem \ref{th231} to the case $\ep_4=-1$.  


%
\section{The case (4B): $\pm\cbb\in G_{4,j,k}$}\label{se4B}

We treat the following eight groups $G=G_{4,j,2}$\ ($1\le j\le 5$, $j=7$) and 
$G=G_{4,6,k}$\ ($3\le k\le 4$) of the 4th crystal system in dimension 3 which 
contain $\cbb$ or $-\cbb$:  
\begin{align*}
G_{4,1,2}&=\langle \cbb\rangle\cong\mathcal{C}_4,&
\mathcal{N}\ni G_{4,2,2}&=\langle -\cbb\rangle\cong\mathcal{C}_4,\\
G_{4,3,2}&=\langle \cbb,-I_3\rangle\cong\mathcal{C}_4\times \mathcal{C}_2,\\  
G_{4,4,2}&=\langle \cbb,\la_3\rangle\cong\mathcal{D}_4,& 
G_{4,5,2}&=\langle \cbb,-\la_3\rangle\cong\mathcal{D}_4, \\   
G_{4,6,3}&=\langle -\cbb,-\la_3\rangle\cong\mathcal{D}_4, & 
G_{4,6,4}&=\langle -\cbb,\la_3\rangle\cong\mathcal{D}_4, \\
G_{4,7,2}&=\langle \cbb,\la_3,-I_3\rangle\cong\mathcal{D}_4\times \mathcal{C}_2.
\end{align*} 

%
\subsection{The case (4B$^+$): $\cbb\in G_{4,j,k}$}\label{subse4p}
%

We treat the cases of 
\begin{align*}
G_{4,1,2}&=\langle \cbb\rangle\cong\mathcal{C}_4,&
G_{4,3,2}&=\langle \cbb,-I_3\rangle\cong\mathcal{C}_4\times \mathcal{C}_2,& 
G_{4,4,2}&=\langle \cbb,\la_3\rangle\cong\mathcal{D}_4,\\
G_{4,5,2}&=\langle \cbb,-\la_3\rangle\cong\mathcal{D}_4, &  
G_{4,7,2}&=\langle \cbb,\la_3,-I_3\rangle\cong\mathcal{D}_4\times \mathcal{C}_2
\end{align*} 
which have a normal subgroup $\langle\cbb\rangle$. 
The actions of $\cbb$, $-I_3$, $\la_3$ and $-\la_3$ on $K(x,y,z)$ are given by 
\begin{align*}
\cbb &: x\ \mapsto\ \frac{a}{z},\ y\ \mapsto\ bxyz,\ z\ \mapsto\ \frac{c}{y},&
-I_3 &: x\ \mapsto\ \frac{d}{x},\ y\ \mapsto\ \frac{e}{y},\ z\ \mapsto\ \frac{f}{z},\\
\la_3 &: x\ \mapsto\ gz,\ y\ \mapsto\ \frac{h}{xyz},\ z\ \mapsto\ ix,& 
-\la_3 &: x\ \mapsto\ \frac{j}{z},\ y\ \mapsto\ kxyz,\ z\ \mapsto\ \frac{l}{x}.
\end{align*}
By replacing $ay/c$ and $z/a$ by $y$ and $z$ respectively and the other coefficients, 
the action of $\cbb$ on $K(x,y,z)$ is given by 
$\cbb : x\ \mapsto\ 1/z,\ y\ \mapsto\ abxyz,\ z\ \mapsto\ 1/y$. 
Also by replacing $ab$ by $b$, we may assume that the action of $\cbb$ on 
$K(x,y,z)$ is given by 
\[
\cbb : x\ \mapsto\ \frac{1}{z},\ y\ \mapsto\ bxyz,\ z\ \mapsto\ \frac{1}{y},
\]
i.e. $a=c=1$. 
By the equalities $\cbb^4=(\pm \la_3)^2=I_3$, we have $b^2=gi=j^2k^2=1$ and $j=l$. 

By the relations of the generators of the $G$'s as in (\ref{eqc4g}), 
the problem reduces to the following cases: 
\begin{align*}
\cbb &: x\ \mapsto\ \frac{1}{z},\ y\ \mapsto\ \ep_1 xyz,\ z\ \mapsto\ \frac{1}{y},&
-I_3 &: x\ \mapsto\ \frac{1}{fx},\ y\ \mapsto\ \frac{1}{fy},\ z\ \mapsto\ \frac{f}{z},\\
\la_3 &: x\ \mapsto\ gz,\ y\ \mapsto\ \frac{\ep_1 g}{xyz},\ z\ \mapsto\ \frac{x}{g},& 
-\la_3 &: x\ \mapsto\ \frac{\ep_2}{z},\ y\ \mapsto\ \ep_1\ep_2 xyz,\ z\ \mapsto\ \frac{\ep_2}{x} 
\end{align*}
with $f,g\in K^\times$ and $\ep_1,\ep_2=\pm 1$. 
Note that $fg=\ep_2=\pm 1$ for $G_{4,7,2}$. 
Put 
\[
X:=\frac{1}{yz},\quad Y:=xz,\quad Z:=\frac{1}{x}.
\]
Then $K(x,y,z)=K(X,Y,Z)$ and the actions of 
$\cbb$, $-I_3$, $\la_3$ and $-\la_3$ on $K(X,Y,Z)$ are given by 
\begin{align*}
\cbb &: X\ \mapsto\ \frac{\ep_1}{Y},\ Y\ \mapsto\ X,\ Z\ \mapsto\ YZ,&
-I_3 &: X\ \mapsto\ \frac{1}{X},\ Y\ \mapsto\ \frac{1}{Y},\ Z\ \mapsto\ \frac{f}{Z},\\
\la_3 &: X\ \mapsto\ \frac{\ep_1}{X},\ Y\ \mapsto\ Y,\ Z\ \mapsto\ \frac{1}{gYZ},& 
-\la_3 &: X\ \mapsto\ \ep_1 X,\ Y\ \mapsto\ \frac{1}{Y},\ Z\ \mapsto\ \ep_2 YZ. 
\end{align*}
It follows from Theorem \ref{thAHK} that $K(x,y,z)^{G_{4,1,2}}$ and $K(x,y,z)^{G_{4,5,2}}$ 
are rational over $K$. 
Indeed we get explicit transcendental bases of $K(X,Y,Z)^{G_{4,1,2}}$ and of 
$K(X,Y,Z)^{G_{4,5,2}}$ over $K$ respectively as follows: 
First we take a $G_{4,1,2}$-invariant $\Theta$ as 
\begin{align*}
\Theta:=\mathrm{Tr}_{\langle\cbb\rangle}(\ep_1 XZ)=(\ep_1 XZ+Z+YZ+XYZ)=Z(XY+\ep_1 X+Y+1)
\end{align*}
where $\mathrm{Tr}_{\langle\cbb\rangle}$ is the trace map under the action of $\langle\cbb\rangle$. 
Then we have $K(X,Y,Z)=K(X,Y,\Theta)$ and the action of $\ta_3=\cbb^2$ on $K(X,Y,\Theta)$ 
is given by 
\begin{align*}
\ta_3=\cbb^2\ :\ X\ \mapsto\ \frac{\ep_1}{X},\ Y\ \mapsto\ \frac{\ep_1}{Y},\ 
\Theta\ \mapsto\ \Theta. 
\end{align*}
By using Lemma \ref{lemaa}, we get $K(X,Y,Z)^{\langle\ta_3\rangle}=K(t_1,t_2,t_3)$ where 
\begin{align*}
t_1:=\frac{XY+1}{X+\ep_1 Y},\quad t_2:=\frac{XY-1}{X-\ep_1 Y},\quad t_3:=\Theta.
\end{align*}
Then the actions of $\cbb$, $-I_3$, $\la_3$ and $-\la_3$ on 
$K(x,y,z)^{\langle\ta_3\rangle}=K(t_1,t_2,t_3)$ are given by 
\begin{align*}
\cbb &: t_1\ \mapsto\ \frac{1}{t_1},\ t_2\ \mapsto\ \frac{-1}{t_2},\ t_3\ \mapsto\ t_3,\\
-I_3 &: t_1\ \mapsto\ \ep_1 t_1,\ t_2\ \mapsto\ \ep_1 t_2,\ t_3\ \mapsto\ 
-\ep_1\frac{4f(t_1+1)(t_1+\ep_1)(t_2^2-\ep_1)}{(t_1^2-t_2^2)t_3},\\
\la_3 &: t_1\ \mapsto\ \frac{\ep_1}{t_1},\ t_2\ \mapsto\ \frac{\ep_1}{t_2},\ 
t_3\ \mapsto\ -\ep_1\frac{4(t_1+1)(t_1+\ep_1)(t_2^2-\ep_1)}{g(t_1^2-t_2^2)t_3},\\
-\la_3 &: t_1\ \mapsto\ \frac{1}{t_1},\ t_2\ \mapsto\ \frac{1}{t_2},\ t_3\ \mapsto\ \ep_2 t_3. 
\end{align*}
Now we put $t_1':=t_1/t_2$ then the action of $\cbb$ on $K(t_1,t_2,t_3)=K(t_1',t_2,t_3)$ is 
given by 
\begin{align*}
\cbb &: t_1'\ \mapsto\ \frac{-1}{t_1'},\ t_2\ \mapsto\ \frac{-1}{t_2},\ t_3\ \mapsto\ t_3.
\end{align*}
Applying Lemma \ref{lemaa} again, we have 
$K(x,y,z)^{\langle\cbb\rangle}=K(t_1',t_2,t_3)^{\langle\cbb\rangle}=K(u_1,u_2,u_3)$ where 
\begin{align*}
u_1:=\frac{t_1't_2-1}{t_1'+t_2}=\frac{t_1-1}{(t_1/t_2)+t_2},
\quad u_2:=\frac{t_1't_2+1}{t_1'-t_2}=\frac{t_1+1}{(t_1/t_2)-t_2},\quad u_3:=t_3.
\end{align*}
Thus the set $\{u_1,u_2,u_3\}$ becomes a transcendental basis of $K(x,y,z)^{G_{4,1,2}}$ over $K$ 
and hence $K(x,y,z)^{G_{4,1,2}}=K(x,y,z)^{\langle\cbb\rangle}$ is rational over $K$.

The actions of $-I_3$, $\la_3$ and $-\la_3$ on 
$K(x,y,z)^{\langle\cbb\rangle}=K(u_1,u_2,u_3)$ are given by 
\begin{align*}
-I_3 &: \begin{cases}
\ \displaystyle{u_1\ \mapsto\ u_1,\ u_2\ \mapsto\ u_2,\ 
u_3\ \mapsto\ -\frac{16fu_2^2(u_1^2+1)(u_1u_2-1)}{(u_1+u_2)^2(u_1u_2+1)u_3}},\ 
\mathrm{if}\ \ep_1=1,\vspace*{1mm}\\
\ \displaystyle{u_1\ \mapsto\ -u_2,\ u_2\ \mapsto\ -u_1,\ 
u_3\ \mapsto\ \frac{16fu_1u_2(u_1^2+1)(u_2^2+1)}{(u_1+u_2)^2(u_1u_2+1)u_3}},\ 
\mathrm{if}\ \ep_1=-1,
\end{cases}\\
\la_3 &: 
\begin{cases}
\ \displaystyle{u_1\ \mapsto\ -u_1,\ u_2\ \mapsto\ -u_2,\ 
u_3\ \mapsto\ -\frac{16u_2^2(u_1^2+1)(u_1u_2-1)}{g(u_1+u_2)^2(u_1u_2+1)u_3}},\ 
\mathrm{if}\ \ep_1=1,\vspace*{1mm}\\
\ \displaystyle{u_1\ \mapsto\ u_2,\ u_2\ \mapsto\ u_1,\ 
u_3\ \mapsto\ \frac{16u_1u_2(u_1^2+1)(u_2^2+1)}{g(u_1+u_2)^2(u_1u_2+1)u_3}},\ 
\mathrm{if}\ \ep_1=-1,
\end{cases}\\
-\la_3 &:\hspace*{5.3mm} u_1\ \mapsto\ -u_1,\ u_2\ \mapsto\ -u_2,\ u_3\ \mapsto\ \ep_2 u_3
\end{align*}
where $fg=\ep_2=\pm 1$. 
Hence $K(x,y,z)^{G_{4,5,2}}=K(u_1,u_2,u_3)^{\langle-\la_3\rangle}$ is rational over $K$ 
and we get an explicit transcendental basis of $K(x,y,z)^{G_{4,5,2}}$ over $K$. 

Next we consider the remaining three cases $G_{4,3,2}=\langle \cbb,-I_3\rangle$, 
$G_{4,4,2}=\langle \cbb,\la_3\rangle$ and $G_{4,7,2}=\langle \cbb,\la_3,-\la_3\rangle$. 
We put 
\begin{align*}
\begin{cases}
\ \displaystyle{v_1:=u_1u_2,\quad v_2:=u_1,\quad v_3:=\frac{(u_1+u_2)(u_1u_2+1)u_3}{4u_2}},\ 
\mathrm{if}\ \ep_1=1,\vspace*{1mm}\\
\ \displaystyle{v_1:=u_1+u_2,\quad v_2:=u_1-u_2,\quad v_3:=\frac{8u_1(u_1^2+1)}{(u_1+u_2)u_3}},\ 
\mathrm{if}\ \ep_1=-1.
\end{cases}
\end{align*}
Then $K(x,y,z)^{\langle\cbb\rangle}=K(u_1,u_2,u_3)=K(v_1,v_2,v_3)$ and the actions of 
$-I_3$, $\la_3$ and $-\la_3$ on $K(v_1,v_2,v_3)$ are given by 
\begin{align*}
-I_3 &: \begin{cases}
\ \displaystyle{v_1\ \mapsto\ v_1,\ v_2\ \mapsto\ v_2,\ 
v_3\ \mapsto\ -\frac{f(v_1^2-1)(v_2^2+1)}{v_3}},\ 
\mathrm{if}\ \ep_1=1,\vspace*{1mm}\\
\ \displaystyle{v_1\ \mapsto\ -v_1,\ v_2\ \mapsto\ v_2,\ 
v_3\ \mapsto\ \frac{v_1^2-v_2^2+4}{fv_3}},\ 
\mathrm{if}\ \ep_1=-1,
\end{cases}\\
\la_3 &: 
\begin{cases}
\ \displaystyle{v_1\ \mapsto\ v_1,\ v_2\ \mapsto\ -v_2,\ 
v_3\ \mapsto\ -\frac{(v_1^2-1)(v_2^2+1)}{gv_3}},\ 
\mathrm{if}\ \ep_1=1,\vspace*{1mm}\\
\ \displaystyle{v_1\ \mapsto\ v_1,\ v_2\ \mapsto\ -v_2,\ 
v_3\ \mapsto\ \frac{g(v_1^2-v_2^2+4)}{v_3}},\ 
\mathrm{if}\ \ep_1=-1,
\end{cases}\\
-\la_3 &:\hspace*{5.3mm} 
v_1\ \mapsto\ \ep_1 v_1,\ v_2\ \mapsto\ -v_2,\ v_3\ \mapsto\ \ep_2 v_3
\end{align*}
where $fg=\ep_2=\pm 1$. 
It follows from Theorem \ref{thInv} (I) that 
$K(x,y,z)^{G_{4,3,2}}=K(v_1,v_2,v_3)^{\langle -I_3\rangle}$ and 
$K(x,y,z)^{G_{4,4,2}}=K(v_1,v_2,v_3)^{\langle \la_3\rangle}$ are rational over $K$. 

Indeed we have 
$K(x,y,z)^{G_{4,4,2}}=K(v_1,v_2,v_3)^{\langle \la_3\rangle}=K(w_1,w_2,w_3)$ where 
\begin{align*}
\begin{cases}
\ \displaystyle{w_1:=v_1,\ w_2:=v_3-\frac{(v_1^2-1)(v_2^2+1)}{gv_3},\ w_3:=
\Bigl(v_3+\frac{(v_1^2-1)(v_2^2+1)}{gv_3}\Bigr)\Big/v_2},\ \mathrm{if}\ \ep_1=1,\vspace*{1mm}\\
\ \displaystyle{w_1:=v_1,\ w_2:=v_3+\frac{g(v_1^2-v_2^2+4)}{v_3},\ w_3:=
\Bigl(v_3-\frac{g(v_1^2-v_2^2+4)}{v_3}\Bigr)\Big/v_2},\ \mathrm{if}\ \ep_1=-1.
\end{cases}
\end{align*}
The action of $-\la_3$ on $K(w_1,w_2,w_3)$ is given by 
\begin{align*}
-\la_3 :  w_1\ \mapsto\ \ep_1 w_1,\ w_2\ \mapsto\ \ep_2 w_2,\ w_3\ \mapsto\ -\ep_2 w_3.
\end{align*}
Hence $K(x,y,z)^{G_{4,7,2}}=K(w_1,w_2,w_3)^{\langle-\la_3\rangle}$ is rational over $K$. 

%
\subsection{The case (4B$^-$) (i): $-\cbb\in G_{4,6,3}$}\label{subse4BM}
%

We treat the case of 
\begin{align*} 
G_{4,6,3} = \langle -\cbb,-\la_3\rangle.
\end{align*} 
The monomial actions of $-\cbb$ and of $-\la_3$ are given by 
\begin{align*} 
-\cbb &: x\ \mapsto az,\ y\ \mapsto\ \frac{b}{xyz},\ z\ \mapsto\ cy,& 
-\la_3 &: x\ \mapsto\ \frac{d}{z},\ y\ \mapsto\ exyz,\ z\ \mapsto\  \frac{f}{x}.
\end{align*}
By replacing $(acy,az)$ by $(y,z)$, we may assume that $a=c=1$. 
From $(-\la_3)^2=I_3$, we see $d=f$ and $d^2e^2=1$. 
By the equality $[-\cbb,-\la]=(-\cbb)^2$, the actions of $-\cbb$ 
and of $-\la_3$ on $K(x,y,z)$ may be reduced to the following cases:
\begin{align*}
-\cbb &: x\ \mapsto z,\ y\ \mapsto\ \frac{b}{xyz},\ z\ \mapsto\ y,& 
-\la_3 &: x\ \mapsto\ \frac{d}{z},\ y\ \mapsto\ \frac{xyz}{\ep_1 d},\ z\ \mapsto\ \frac{d}{x}
\end{align*}
where $b=\ep_1 d^2\in K^\times$ and $\ep_1=\pm 1$. 
The action of $\ta_3=(-\cbb)^2$ on $K(x,y,z)$ is given by 
\begin{align*}
\ta_3 &: x\ \mapsto y,\ y\ \mapsto\ x,\ z\ \mapsto\ \frac{b}{xyz}.
\end{align*}

By Lemma \ref{lemc2t}, we get $K(x,y,z)^{\langle\ta_3\rangle}=K(t_1,t_2,t_3)$ where 
\begin{align*}
t_1:=\frac{xy}{d(x+y)},\quad 
t_2:=\frac{xyz+\frac{b}{z}}{d(x+y)},\quad 
t_3:=\frac{xyz-\frac{b}{z}}{d(x-y)}.
\end{align*}
The actions of $-\cbb$ and of $-\la_3$ on $K(x,y,z)^{\langle\ta_3\rangle}=K(t_1,t_2,t_3)$ are 
given by 
\begin{align*}
-\cbb &: t_1\ \mapsto -\frac{\ep_1 (t_2^2-t_3^2)}{4dt_1t_2(t_3^2-\ep_1)},\ 
t_2\ \mapsto\ \frac{\ep_1}{t_2},\ t_3\ \mapsto\ -\frac{\ep_1}{t_3},& 
-\la_3 &: t_1\ \mapsto\ \frac{t_1}{t_2},\ t_2\ \mapsto\ \frac{1}{t_2},\ 
t_3\ \mapsto\  \frac{1}{t_3}
\end{align*}
where $\ep_1=\pm 1$. 
We put 
\begin{align*}
u_1:=\frac{t_2+t_3}{2t_1(t_3-1)},\quad u_2:=\frac{t_2+1}{t_2-1},\quad u_3:=\frac{t_3+1}{t_3-1}.
\end{align*}
Then $K(x,y,z)^{\langle\ta_3\rangle}=K(t_1,t_2,t_3)=K(u_1,u_2,u_3)$ and 
the actions of $-\cbb$ and of $-\la_3$ on $K(u_1,u_2,u_3)$ are given by 
\begin{align*}
-\cbb &: 
\begin{cases}
\ \displaystyle{u_1\ \mapsto -\frac{d}{u_1},\ u_2\ \mapsto\ -u_2,\ 
u_3\ \mapsto\ -\frac{1}{u_3}},\ \mathrm{if}\ \ep_1=1,\vspace*{1mm}\\
\ \displaystyle{u_1\ \mapsto -\frac{d(u_3^2+1)}{2u_1},\ u_2\ \mapsto\ -\frac{1}{u_2},\ 
u_3\ \mapsto\ -u_3},\ \mathrm{if}\ \ep_1=-1,
\end{cases}\\
-\la_3 &: u_1\ \mapsto\ -u_1,\ u_2\ \mapsto\ -u_2,\ u_3\ \mapsto\ -u_3.
\end{align*}
Because $G_{4,2,2}=\langle-\cbb\rangle\in\mathcal{N}$, 
we consider the fixed field $K(x,y,z)^{\langle\ta_3,-\la_3\rangle}=K(v_1,v_2,v_3)$ where 
\begin{align*}
v_1:=u_1u_3,\quad v_2:=u_2u_3,\quad v_3:=u_3^2.
\end{align*}
Then the action of $-\cbb$ on $K(x,y,z)^{\langle\ta_3,-\la_3\rangle}=K(v_1,v_2,v_3)$ 
is given by 
\begin{align*}
-\cbb &: 
\begin{cases}
\ \displaystyle{v_1\ \mapsto \frac{d}{v_1},\ v_2\ \mapsto\ \frac{v_2}{v_3},\ 
v_3\ \mapsto\ \frac{1}{v_3}},\ \mathrm{if}\ \ep_1=1,\vspace*{1mm}\\
\ \displaystyle{v_1\ \mapsto \frac{dv_3(v_3+1)}{2v_1},\ v_2\ \mapsto\ \frac{v_3}{v_2},\ 
v_3\ \mapsto\ v_3},\ \mathrm{if}\ \ep_1=-1.
\end{cases}
\end{align*}
By Theorem \ref{thAHK}, the rationality problem reduces to the case of 
$2$-dimensional monomial actions and hence 
$K(x,y,z)^{G_{4,6,3}}=K(v_1,v_2,v_3)^{\langle-\cbb\rangle}$ is rational over $K$. 
Indeed we get an explicit transcendental basis of $K(x,y,z)^{G_{4,6,3}}$ over 
$K$ by using Theorem \ref{thab}. 

%
\subsection{The case (4B$^-$) (ii): $-\cbb\in G_{4,6,4}$} 
%

We treat the cases of 
\begin{align*} 
G_{4,4,2}=\langle \cbb,\la_3\rangle,\quad 
G_{4,6,4}=\langle -\cbb,\la_3\rangle,\quad
G_{4,7,2}=\langle \cbb,\la_3,-I_3\rangle
\end{align*} 
which have a normal subgroup $G_{3,1,4}=\langle\ta_3,\la_3\rangle$ where $\ta_3=(\pm \cbb)^2$. 

Although we treated the groups $G_{4,4,2}$ and $G_{4,7,2}$ in Subsection \ref{subse4p}, 
we will give an another proof for these two groups in order to describe the difference 
between three groups $G_{4,4,2}$, $G_{4,6,4}$ and $G_{4,7,2}$. 

The actions of $\cbb$, $-\cbb$, $\la_3$ and $-I_3$ on $K(x,y,z)$ are given by 
\begin{align*}
\cbb &: x\ \mapsto\ \frac{a}{z},\ y\ \mapsto\ bxyz,\ z\ \mapsto\ \frac{c}{y},&
-\cbb &: x\ \mapsto\ a'z,\ y\ \mapsto\ \frac{b'}{xyz},\ z\ \mapsto\ c'y,\\
\la_3 &: x\ \mapsto\ dz,\ y\ \mapsto\ \frac{e}{xyz},\ z\ \mapsto\ fx,& 
-I_3 &: x\ \mapsto\ \frac{g}{x},\ y\ \mapsto\ \frac{h}{y},\ z\ \mapsto\ \frac{i}{z}.
\end{align*}
For $G_{4,4,2}$ and $G_{4,7,2}$, by replacing $(ay/c,dz)$ by $(y,z)$ and the other coefficients, 
we may assume that $a=c$ and $d=f=1$ without loss of generality. 
For $G_{4,6,4}$, by replacing $(a'c'y,dz)$ by $(y,z)$ and the other coefficients, 
we may assume that $a'c'=d=f=1$. 
By $\cbb^4=I_3$, we see $a^2b^2=1$. 
From the relations of the generators of the $G$'s as in (\ref{eqc4g}), 
the problem may be reduced to the following case: 
\begin{align*}
\cbb &: x\ \mapsto\ \frac{a}{z},\ y\ \mapsto\ \frac{xyz}{\ep_1 a},\ 
z\ \mapsto\ \frac{a}{y},&
-\cbb &: x\ \mapsto\ \ep_2 z,\ y\ \mapsto\ \frac{\ep_2 e}{xyz},\ z\ \mapsto\ \ep_2 y,\\
\la_3 &: x\ \mapsto\ z,\ y\ \mapsto\ \frac{e}{xyz},\ z\ \mapsto\ x,& 
-I_3 &: x\ \mapsto\ \frac{\ep_2 a}{x},\ y\ \mapsto\ \frac{\ep_2 a}{y},\ 
z\ \mapsto\ \frac{\ep_2 a}{z}
\end{align*}
where $a,e\in K^\times$ and $\ep_1,\ep_2=\pm 1$ with 
$e=\ep_1 a^2$ for $G_{4,4,2}$ and $G_{4,7,2}$. 
Note that the action of $\ta_3=(\pm \cbb)^2$ on $K(x,y,z)$ is given by 
\begin{align*}
\ta_3=(\pm \cbb)^2 : x\ \mapsto\ y,\ y\ \mapsto\ x,\ z\ \mapsto\ \frac{e}{xyz}
\end{align*}
where $e=\ep_1 a^2\in K^\times$ for $G_{4,4,2}$ and $G_{4,7,2}$. 
Now we put 
\begin{align*}
w:=\frac{e}{xyz}\ \Bigl(=\frac{\ep_1 a^2}{xyz}\ \text{for}\ 
G_{4,4,2}\ \text{and}\ G_{4,7,2}\Bigr)
\end{align*}
as in Lemma \ref{lemV42} (cf. also Subsection \ref{subse314}). 
Then $K(x,y,z)=K(x,y,z,w)$ and the actions of $\cbb$, $-\cbb$, $\ta_3$, $\la_3$ and 
$-I_3$ on $K(x,y,z,w)$ are given by 
\begin{align*}
\cbb &: x\ \mapsto\ \frac{a}{z},\ y\ \mapsto\ \frac{a}{w},\ 
z\ \mapsto\ \frac{a}{y},\ w\ \mapsto\ \frac{a}{x},\\
-\cbb &: x\ \mapsto\ \ep_2 z,\ y\ \mapsto\ \ep_2 w,\ z\ \mapsto\ \ep_2 y,\ 
w\ \mapsto \ep_2 x,\\
\la_3 &: x\ \mapsto\ z,\ y\ \mapsto\ w,\ z\ \mapsto\ x,\ w\ \mapsto\ y,\\
\ta_3 &: x\ \mapsto\ y,\ y\ \mapsto\ x,\ z\ \mapsto\ w,\ w\ \mapsto\ z,\\
-I_3 &: x\ \mapsto\ \frac{\ep_2 a}{x},\ y\ \mapsto\ \frac{\ep_2 a}{y},\ 
z\ \mapsto\ \frac{\ep_2 a}{z},\ w\ \mapsto\ \frac{\ep_2 a}{w}. 
\end{align*}
By Lemma \ref{lemV42}, we have $K(x,y,z)^{\langle \ta_3,\la_3\rangle}
=K(x,y,z,w)^{\langle \ta_3,\la_3\rangle}=K(v_1,v_2,v_3)$ where
\begin{align*}
v_1\ :=\ \frac{x+y-z-w}{xy-zw},\quad
v_2\ :=\ \frac{x-y-z+w}{xw-yz},\quad
v_3\ :=\ \frac{x-y+z-w}{xz-yw}
\end{align*}
with $w=e/(xyz)$. 
Note that $G_{3,1,4}=\langle\ta_3,\la_3\rangle$ as we treated in Subsection \ref{subse314}. 
The actions of $\cbb$, $-\cbb$ and $-I_3$ on 
$K(x,y,z)^{\langle \ta_3,\la_3\rangle}=K(v_1,v_2,v_3)$ are given by 
\begin{align*}
\cbb &: v_1\ \mapsto\ \frac{-v_1+v_2+v_3}{a v_2v_3},\ v_2\ \mapsto\ 
\frac{v_1+v_2-v_3}{a v_1v_2},\ v_3\ \mapsto\ \frac{v_1-v_2+v_3}{a v_1v_3},\\
-\cbb &: v_1\ \mapsto\ \ep_2 v_1,\ v_2\ \mapsto\ \ep_2 v_3,\ v_3\ \mapsto\ \ep_2 v_2,\\
-I_3 &: v_1\ \mapsto\ \ep_2\frac{-v_1+v_2+v_3}{a v_2v_3},\ v_2\ \mapsto\ 
\ep_2\frac{v_1-v_2+v_3}{a v_1v_3},\ v_3\ \mapsto\ \ep_2\frac{v_1+v_2-v_3}{a v_1v_2}.
\end{align*}
Hence $K(x,y,z)^{G_{4,6,4}}=K(x,y,z)^{\langle-\cbb,\la_3\rangle}
=K(v_1,v_2,v_3)^{\langle-\cbb\rangle}$ is rational over $K$. 
We also see that $K(x,y,z)^{G_{4,4,2}}=K(x,y,z)^{\langle\cbb,\la_3\rangle}$ 
is rational over $K$ by the following lemma: 
\begin{lemma}\label{lemuc2}
We have $K(x,y,z)^{G_{4,4,2}}=K(v_1,v_2,v_3)^{\langle\cbb\rangle}=K(t_1,t_2,t_3)$ where 
\begin{align*}
t_1&:=\frac{v_2-v_3}{v_1},&
t_2&:=\frac{av_1v_3(v_1+v_2-v_3-av_1v_2^2)}{v_1^2(1-a^2v_2^2v_3^2)-(v_2-v_3)^2},&
t_3&:=\frac{av_1v_2(v_1-v_2+v_3-av_1v_3^2)}{v_1^2(1-a^2v_2^2v_3^2)-(v_2-v_3)^2}. 
\end{align*}
\end{lemma}
\begin{proof}
We see $K(v_1,v_2,v_3)=K(t_1,v_2,v_3)$ and the action of $\cbb$ on $K(t_1,v_2,v_3)$ is given by 
\begin{align*}
\cbb :  t_1\ \mapsto\ t_1,\ v_2\ \mapsto\ \frac{1+t_1}{a v_2},\ 
v_3\ \mapsto\ \frac{1-t_1}{a v_3}. 
\end{align*}
By Theorem \ref{thab}, we have
\begin{align*}
K(t_1,v_2,v_3)^{\langle\cbb\rangle}=K\Bigl(t_1,\frac{v_3(v_2^2-A)}{v_2^2v_3^2-AB},
\frac{v_2(v_3^2-B)}{v_2^2v_3^2-AB}\Bigr)\quad \mathrm{with}\quad 
(A,B)=\Bigl(\frac{1+t_1}{a},\frac{1-t_1}{a}\Bigr). 
\end{align*}
Hence the assertion follows by direct calculation.
\end{proof}

For $G_{4,7,2}=\langle\cbb,\la_3,-I_3\rangle$, 
the action of $-I_3$ on $K(x,y,z)^{G_{4,4,2}}=K(t_1,t_2,t_3)$ is given by 
\begin{align*}
-I_3 :  t_1\ \mapsto\ -t_1,\ t_2\ \mapsto\ \ep_2 t_3,\ t_3\ \mapsto\ \ep_2 t_2. 
\end{align*}
Therefore we conclude that 
$K(x,y,z)^{G_{4,7,2}}=K(t_1,t_2,t_3)^{\langle-I_3\rangle}$ is rational over $K$. 
We can also check the rationality of $K(x,y,z)^{G_{4,7,2}}$ by Lemma \ref{lemtlm} in Subsection 
\ref{subse314} because $K(x,y,z)^{G_{4,7,2}}=(K(x,y,z)^{G_{3,3,4}})^{\langle\cbb\rangle}
=(K(x,y,z)^{\langle\ta_3,\la_3,-I_3\rangle})^{\langle\cbb\rangle}$. 

\begin{remark}
For the case $(4B^-)$, i.e. two groups $G_{4,6,3}=\langle-\cbb,-\la_3\rangle$ and 
$G_{4,6,4}=\langle-\cbb,\la_3\rangle$, 
to get suitable generators of $K(x,y,z)^{\langle\ta_3\rangle}$ over $K$ is essential, 
because $G_{4,2,2}=\langle-\cbb\rangle\in\mathcal{N}$. 
Although we can show the rationality of $K(x,y,z)^{G_{4,6,4}}$ by using 
$K(x,y,z)^{\langle\ta_3\rangle}=K(t_1,t_2,t_3)$ via Lemma \ref{lemc2t}, 
we choose a method as in this subsection in order to compare the result with some cases 
$G\supset G_{3,1,4}=\langle\ta_3,\la_3\rangle$ (cf. Subsection \ref{subse314} and 
Section \ref{se73}).
\end{remark}


%
%
\section{The case (5A): $\ca\in G_{5,j,k}$}\label{se5A}

In this section, we consider the following eight groups $G=G_{5,j,k}$, $2\leq k\leq 3$, 
which have a normal subgroup $\langle\ca\rangle$: 
\begin{align*}
G_{5,1,2}&=\langle \ca\rangle\cong\mathcal{C}_3,& 
G_{5,2,2}&=\langle \ca,-I_3\rangle\cong\mathcal{C}_6,\\
G_{5,3,2}&=\langle \ca,-\al\rangle\cong\mathcal{S}_3,&
G_{5,3,3}&=\langle \ca,-\be\rangle\cong\mathcal{S}_3,\\
G_{5,4,2}&=\langle \ca,\be\rangle\cong\mathcal{S}_3,& 
G_{5,4,3}&=\langle \ca,\al\rangle\cong\mathcal{S}_3,\\
G_{5,5,2}&=\langle \ca,\al,-I_3\rangle\cong\mathcal{D}_6,& 
G_{5,5,3}&=\langle \ca,\be,-I_3\rangle\cong\mathcal{D}_6.
\end{align*}
Note that we changed the generator $-\al$ (resp. $-\be$) of $G_{5,5,2}$ (resp. $G_{5,5,3}$) 
by $\al=(-\al)(-I_3)$ (resp. $\be=(-\be)(-I_3)$). 
The actions of $\ca$, $-I_3$, $\al$, $-\al$, $\be$ and $-\be$ on $K(x,y,z)$ are given by
\begin{align*}
\ca &: x\ \mapsto\ ay,\ y\ \mapsto\ \frac{b}{xy},\ z\ \mapsto\ cz,& 
-I_3 &: x\ \mapsto\ \frac{d}{x},\ y\ \mapsto\ \frac{e}{y},\ z\ \mapsto\ \frac{f}{z},\\
\al &: x\ \mapsto\ gy,\ y\ \mapsto\  hx,\ z\ \mapsto\ iz,& 
-\al &: x\ \mapsto\ \frac{j}{y},\ y\ \mapsto\  \frac{k}{x},\ z\ \mapsto\ \frac{l}{z},\\
\be &: x\ \mapsto\ \frac{m}{y},\ y\ \mapsto\  \frac{n}{x},\ z\ \mapsto\ oz,& 
-\be &: x\ \mapsto\ py,\ y\ \mapsto\ qx,\ z\ \mapsto\ \frac{r}{z}.
\end{align*}

We may assume that $a=1$ by replacing $ay$ by $y$ and the other coefficients. 
By the equalities $\ca^3=(-I_3)^2=\al^2=\be^2=I_3$, 
we have $c^3=gh=i^2=o^2=pq=1$, $m=n$, $j=k$. 

By the relations of the generators of the $G$'s as in (\ref{relmat5}), 
we have the following lemma: 
\begin{lemma}
{\rm (i)} If $\ca,-I_3\in G$ then $c=1$, $b^2=d^3$, $d=e\in K^{\times 2}$;\\
{\rm (ii)} If $\ca,\al\in G$ then $c=1$, $g=h=1$;\\
{\rm (iii)} If $\ca,-\al\in G$ then $b^2=j^3$, hence $j\in K^{\times 2}$;\\
{\rm (iv)} If $\ca,\be\in G$ then $c=1$, $b^2=m^3$, hence $m\in K^{\times 2}$;\\
{\rm (v)} If $\ca,-\be\in G$ then $p=q=1$.
\end{lemma}
Thus we may reduce the monomial actions to the following form: 
\begin{align}
\ca^{(b,c)}:=\ca &: x\ \mapsto\ y,\ y\ \mapsto\ \frac{b}{xy},\ z\ \mapsto\ cz,& 
-I_3 &: x\ \mapsto\ \frac{1}{x},\ y\ \mapsto\ \frac{1}{y},\ z\ \mapsto\ \frac{f}{z},
\label{acts3abc}\\
\al &: x\ \mapsto\ y,\ y\ \mapsto\ x,\ z\ \mapsto\ \ep z,& 
-\al &: x\ \mapsto\ \frac{1}{y},\ y\ \mapsto\ \frac{1}{x},\ z\ \mapsto\ \frac{l}{z},\nonumber\\
\be &: x\ \mapsto\ \frac{1}{y},\ y\ \mapsto\ \frac{1}{x},\ z\ \mapsto\ \ep z,& 
-\be &: x\ \mapsto\ y,\ y\ \mapsto\ x,\ z\ \mapsto\ \frac{r}{z}\nonumber
\end{align}
where $b,c,f,l,r\in K^\times$, $c^3=1$ and $\ep=\pm 1$. 
We have $c=1$ except for the groups $G_{5,1,2}$, $G_{5,3,2}$ and $G_{5,3,3}$, and 
$b=1$ except for the groups $G_{5,1,2}$, $G_{5,3,3}$ and $G_{5,4,3}$. 
For $-\al\in G_{5,5,2}$ and $-\be\in G_{5,5,3}$, we have $l=\ep f$ and $r=\ep f$ 
since $-\al=\al(-I_3)$ and $-\be=\be(-I_3)$, respectively. 

\subsection{The cases of $G_{5,1,2}$, $G_{5,4,2}$, $G_{5,4,3}$}

We treat the cases of 
\begin{align*}
G_{5,1,2}=\langle \ca^{(b,c)}\rangle,\quad 
G_{5,4,2}=\langle \ca^{(1,1)},\be\rangle,\quad 
G_{5,4,3}=\langle \ca^{(b,1)},\al\rangle
\end{align*}
where $\ca^{(b,c)}$ is given as in (\ref{acts3abc}).

By applying Theorem \ref{thAHK} to $L=K(x,y)$ and $M=K$, 
the rationality problems may be reduced to the $2$-dimensional case of $K(x,y)^{G_{5,1,2}}$, 
$K(x,y)^{G_{5,4,2}}$ and $K(x,y)^{G_{5,4,3}}$ under monomial actions.  
Hence $K(x,y,z)^{G_{5,1,2}}$, $K(x,y,z)^{G_{5,4,2}}$ and $K(x,y,z)^{G_{5,4,3}}$ 
are rational over $K$.

\subsection{The cases of $G_{5,2,2}$, $G_{5,5,2}$, $G_{5,5,3}$}

We treat the cases of 
\begin{align*}
G_{5,2,2}=\langle \ca^{(1,1)},-I_3\rangle,\quad 
G_{5,5,2}=\langle \ca^{(1,1)},\al,-I_3\rangle,\quad 
G_{5,5,3}=\langle \ca^{(1,1)},\be,-I_3\rangle.
\end{align*}
In these cases, the groups $G$ have a non-trivial center which includes $\langle-I_3\rangle$ 
respectively. 
We first consider the fixed field $K(x,y,z)^{G_{5,2,2}}
=K(x,y,z)^{\langle \ca^{(1,1)}, -I_3\rangle}$. 

By Lemma \ref{lemMas2}, we have $K(x,y,z)^{\langle \ca^{(1,1)}\rangle}=K(\tu(1),\tv(1),z)$. 
We put 
\[
t_1:=\tu(1),\quad t_2:=\tv(1),\quad t_3:=z.
\]
Then the actions of $-I_3$, $\al$ and $\be$ on $K(x,y,z)^{\langle \ca^{(1,1)}\rangle}
=K(t_1,t_2,t_3)$ are given by 
\begin{align*}
-I_3\,&:\, t_1\ \mapsto\ \frac{t_2}{t_1^2-t_1t_2+t_2^2},\ 
t_2\ \mapsto\ \frac{t_1}{t_1^2-t_1t_2+t_2^2},\ t_3\ \mapsto\ \frac{f}{t_3},\\
\al\,&:\, t_1\ \mapsto\ t_2\ \mapsto t_1,\ t_3\ \mapsto\ \ep t_3,\\
\be\,&:\, t_1\ \mapsto\ \frac{t_1}{t_1^2-t_1t_2+t_2^2},\ 
t_2\ \mapsto\ \frac{t_2}{t_1^2-t_1t_2+t_2^2},\ t_3\ \mapsto\ \ep t_3
\end{align*}
where $\ep=\pm 1$. 
We put 
\[
u_1:=\frac{t_1-t_2}{t_1+t_2},\quad u_2:=\frac{2}{t_1+t_2},\quad u_3:=t_3.
\]
Then we see $K(x,y,z)^{\langle \ca^{(1,1)}\rangle}=K(t_1,t_2,t_3)=K(u_1,u_2,u_3)$ and 
the actions of $-I_3$, $\al$ and $\be$ on $K(u_1,u_2,u_3)$ are given by 
\begin{align*}
-I_3\,&:\, u_1\ \mapsto -u_1,\ u_2\ \mapsto\ \frac{3u_1^2+1}{u_2},\ 
u_3\ \mapsto\ \frac{f}{u_3},\\
\al\,&:\, u_1\ \mapsto -u_1,\ u_2\ \mapsto\ u_2,\ u_3\ \mapsto\ \ep u_3,\\
\be\,&:\, u_1\ \mapsto\ u_1,\ 
u_2\ \mapsto\ \frac{3u_1^2+1}{u_2},\ u_3\ \mapsto\ \ep u_3.
\end{align*}
We get an explicit transcendental basis of $K(x,y,z)^{G_{5,2,2}}$ over $K$ as follows: 
\begin{lemma}\label{lem522}
We have $K(x,y,z)^{G_{5,2,2}}=K(u_1,u_2,u_3)^{\langle -I_3\rangle}=K(r_1,r_2,r_3)$ where 
\begin{align*}
r_1=\frac{u_3^2+f}{u_3},\quad 
r_2=\frac{u_1u_2}{3 u_1^2-u_2^2+1},\quad 
r_3=\frac{(3u_1^2+u_2^2+2 u_2+1)(u_3^2-f)}{(3u_1^2-u_2^2+1)u_3}.
\end{align*}
\end{lemma}
\begin{proof}
We put 
\[
w_1:=u_3+\frac{f}{u_3},\ 
w_2:=u_1\Big/\Bigl(u_3-\frac{f}{u_3}\Bigr),\ 
w_3:=u_2+\frac{3u_1^2+1}{u_2},\ 
w_4:=u_1\Big/\Bigl(u_2-\frac{3u_1^2+1}{u_2}\Bigr).
\]
Then we have $K(u_1,u_2,u_3)^{\langle -I_3\rangle}=K(w_1,w_2,w_3,w_4)$ 
because $[K(u_1,u_2,u_3) : K(w_1,w_3,w_4)]=4$ and $w_2\not\in K(w_1,w_3,w_4)$. 
Since $u_1^2=(w_1^2-4f)w_2^2=w_4^2(w_3^2-4)/(12w_4^2+1)$, putting 
$w_2':=w_4(w_3+2)/w_2$, $w_3':=w_4(w_3-2)/w_2$, 
we have $w_2'w_3'=(w_1^2-4f)(12w_4^2+1)$, 
so that $K(w_1,w_2,w_3,w_4)=K(w_1,w_2',w_3',w_4)=K(w_1,w_2',w_4)$. 
It can be checked by the definition that $r_1=w_1$, $r_2=-w_4$, $r_3=-w_2'$. 
\end{proof}
\begin{remark}
The rationality of $K(u_1,u_2,u_3)^{\langle-I_3\rangle}$ is a result of $R(a,b,c)$ 
in \cite[Lemma 2]{Yam}.
\end{remark}
In particular, $K(x,y,z)^{G_{5,2,2}}$ is rational over $K$. 
The actions of $\al$ and $\be$ on $K(x,y,z)^{G_{5,2,2}}=K(r_1,r_2,r_3)$ are given by 
\begin{align*}
\al\,&:\, r_1\ \mapsto \ep r_1,\ r_2\ \mapsto\ -r_2,\ r_3\ \mapsto\ \ep r_3,\\
\be\,&:\, r_1\ \mapsto\ \ep r_1,\ r_2\ \mapsto\ -r_2,\ r_3\ \mapsto\ -\ep r_3.
\end{align*}
Thus we conclude that $K(x,y,z)^{G_{5,5,2}}=K(r_1,r_2,r_3)^{\langle\al\rangle}$ 
and $K(x,y,z)^{G_{5,5,3}}=K(r_1,r_2,r_3)^{\langle\be\rangle}$ are rational over $K$. 

\subsection{The cases of $G_{5,3,2}$, $G_{5,3,3}$}\label{subse532}
In this subsection, we assume that $c=1$. 
The cases of $G_{5,3,2}$ and $G_{5,3,3}$ with $c\neq 1$ will be discussed 
in Section \ref{secnot1}.

We treat the cases of 
\begin{align*}
G_{5,3,2}=\langle \ca^{(1,c)},-\al\rangle,\quad 
G_{5,3,3}=\langle \ca^{(b,c)},-\be\rangle. 
\end{align*}
Recall that the actions of $\ca^{(b,c)}$, $-\al$ and $-\be$ on $K(x,y,z)$ are given as 
\begin{align*}
\ca^{(b,c)} &: x\ \mapsto\ y,\ y\ \mapsto\ \frac{b}{xy},\ z\ \mapsto\ cz,\\
-\al &: x\ \mapsto\ \frac{1}{y},\ y\ \mapsto\ \frac{1}{x},\ z\ \mapsto\ \frac{l}{z},\quad 
-\be : x\ \mapsto\ y,\ y\ \mapsto\ x,\ z\ \mapsto\ \frac{r}{z}
\end{align*}
where $b,c,l,r\in K^\times$ and $c^3=1$. 
We assume that $c=1$ in this subsection. 

By the same way as in the previous subsection, 
we get $K(x,y,z)^{\langle \ca^{(b,1)}\rangle}=K(t_1,t_2,t_3)$ where $t_1=\tu(b)$, 
$t_2=\tv(b)$, $t_3=z$ and $\tu(b)$, $\tv(b)$ are given in Lemma \ref{lemMas2}, and we put 
\begin{align}
u_1:=\frac{t_1-t_2}{t_1+t_2},\quad u_2:=\frac{2}{t_1+t_2},\quad u_3:=t_3.\label{defu123}
\end{align}
Then the actions of $-\al$ and of $-\be$ on $K(x,y,z)^{\langle \ca^{(b,1)}\rangle}
=K(u_1,u_2,u_3)$ are given by 
\begin{align*}
-\al\,&:\, u_1\ \mapsto u_1,\ u_2\ \mapsto\ \frac{3u_1^2+1}{u_2},\ 
u_3\ \mapsto\ \frac{l}{u_3},\\
-\be\,&:\, u_1\ \mapsto\ -u_1,\ u_2\ \mapsto\ u_2,\ u_3\ \mapsto\ \frac{r}{u_3}
\end{align*}
where we adopt $b=1$ in the case of $G_{5,3,2}=\langle \ca^{(1,1)},-\al\rangle$. 

Hence the problems reduce to the case of $2$-dimensional monomial actions. 
It follows by Theorem \ref{thHaj} that 
$K(x,y,z)^{G_{5,3,3}}=K(u_1,u_2,u_3)^{\langle-\be\rangle}$ and 
$K(x,y,z)^{G_{5,3,2}}=K(u_1,u_2,u_3)^{\langle-\al\rangle}$are rational over $K$.  


%
%
\section{The case (5B): $\cb\in G_{5,j,1}$}\label{se5B}

In this section, we consider the following five groups $G=G_{5,j,1}$, $1\leq j\leq 5$, 
which have a normal subgroup $\langle\cb\rangle$: 
\begin{align*}
G_{5,1,1}&=\langle \cb\rangle\cong\mathcal{C}_3,&
G_{5,2,1}&=\langle \cb,-I_3\rangle\cong\mathcal{C}_6,\\
G_{5,3,1}&=\langle \cb,-\al\rangle\cong\mathcal{S}_3,&
G_{5,4,1}&=\langle \cb,\al\rangle\cong\mathcal{S}_3,&
G_{5,5,1}&=\langle \cb,\al,-I_3\rangle\cong\mathcal{D}_6.
\end{align*}
Note that we changed the generator $-\al$ of $G_{5,5,1}$ by $\al=(-\al)(-I_3)$. 
The actions of $\cb$, $-I_3$, $\al$ and $-\al$ on $K(x,y,z)$ are given by
\begin{align*}
\cb &: x\ \mapsto\ ay,\ y\ \mapsto\ bz,\ z\ \mapsto\  cx,& 
-I_3 &: x\ \mapsto\ \frac{d}{x},\ y\ \mapsto\ \frac{e}{y},\ z\ \mapsto\ \frac{f}{z},\\
\al &: x\ \mapsto\ gy,\ y\ \mapsto\  hx,\ z\ \mapsto\ iz,& 
-\al &: x\ \mapsto\ \frac{j}{y},\ y\ \mapsto\  \frac{k}{x},\ z\ \mapsto\ \frac{l}{z}.
\end{align*}

We may assume that $a=b=c=1$ by replacing $(ay,abz)$ by $(y,z)$ and the other coefficients. 
By the equalities $\cb^3=\al^2=(-\al)^2=(-I_3)^2=I_3$, we see $gh=i^2=1$ and $j=k$. 

From the relations of the generators of the $G$'s as in (\ref{relmat5}), 
we also see $d=e=f$, $g=h=i=\pm 1$, $j=k=l$.  
Hence we have 
\begin{align*}
\cb &: x\ \mapsto\ y,\ y\ \mapsto\ z,\ z\ \mapsto\  x,& 
-I_3 &: x\ \mapsto\ \frac{d}{x},\ y\ \mapsto\ \frac{d}{y},\ z\ \mapsto\ \frac{d}{z},\\
\al &: x\ \mapsto\ \ep y,\ y\ \mapsto\ \ep x,\ z\ \mapsto\ \ep z,& 
-\al &: x\ \mapsto\ \frac{j}{y},\ y\ \mapsto\ \frac{j}{x},\ z\ \mapsto\ \frac{j}{z}
\end{align*}
where $d,j\in K^\times$ and $\ep=\pm 1$. 
For $G_{5,5,1}=\langle\cb,\al,-I_3\rangle$, we have $j=\ep d$ because $-\al=\al(-I_3)$. 

\subsection{The cases of $G_{5,1,1}$, $G_{5,3,1}$, $G_{5,4,1}$}

We treat the cases of 
\begin{align*}
G_{5,1,1}=\langle \cb\rangle,\quad 
G_{5,3,1}=\langle \cb,-\al\rangle,\quad 
G_{5,4,1}=\langle \cb,\al\rangle.
\end{align*}

It follows from Lemma \ref{lemMas} that $K(x,y,z)^{G_{5,1,1}}$ $=$ 
$K(x,y,z)^{\langle \cb\rangle}$ $=$ $K(s_1,u,v)$ is rational over $K$, where 
$s_1$, $u$, $v\in K(x,y,z)$ are as in Lemma \ref{lemMas}. 

The action of $G_{5,3,1}=\langle\cb,-\al\rangle$ on $K(x,y,z)$ 
is the twisted action of $\mathcal{S}_3$ as in Theorem \ref{thHK1}. 
Hence it follows from Theorem \ref{thHK1} that $K(x,y,z)^{G_{5,3,1}}$ is rational over $K$. 
We can also get an explicit transcendental basis of $K(x,y,z)^{G_{5,3,1}}$ over $K$ 
by Theorem \ref{thHK1}.

The fixed field $K(x,y,z)^{G_{5,4,1}}=K(s_1,u,v)^{\langle\al\rangle}$ is rational over $K$ 
because the action of $\al$ on $K(x,y,z)^{\langle\cb\rangle}=K(s_1,u,v)$ is given by 
\begin{align*}
\al &: s_1\ \mapsto\ \ep s_1,\ u\ \mapsto\ \ep v,\ v\ \mapsto\ \ep u.
\end{align*}

\subsection{The cases of $G_{5,2,1}$, $G_{5,5,1}$}

We treat the cases of 
\begin{align*}
G_{5,2,1}=\langle \cb,-I_3\rangle,\quad 
G_{5,5,1}=\langle \cb,\al,-I_3\rangle. 
\end{align*}

Since the center of the groups $G_{5,2,1}$ and $G_{5,5,1}$ includes $\langle -I_3\rangle$ 
respectively, we first consider the field $K(x,y,z)^{\langle-I_3\rangle}$. 
By Theorem \ref{thaaa}, we have $K(x,y,z)^{\langle-I_3\rangle}=K(k_1,k_2,k_3)$ 
where 
\begin{align*}
k_1:=\frac{xy+d}{x+y},\quad 
k_2:=\frac{yz+d}{y+z},\quad 
k_3:=\frac{xz+d}{x+z}
\end{align*}
and the actions of $\cb$ and of $\al$ on $K(k_1,k_2,k_3)$ are given by 
\begin{align*}
\cb &: k_1\mapsto\ k_2\ \mapsto\ k_3\ \mapsto\ k_1,\\
\al &: k_1\ \mapsto\ \ep k_1,\ k_2\ \mapsto\ \ep k_3,\ k_3\ \mapsto\ \ep k_2.
\end{align*}
Hence $K(x,y,z)^{G_{5,2,1}}=K(k_1,k_2,k_3)^{\langle\cb\rangle}$ is rational over $K$ 
by Lemma \ref{lemMas}. 
When $\ep=1$, the action of $\langle\cb,\al\rangle$ on 
$K(k_1,k_2,k_3)$ is the permutation of $\mathcal{S}_3$ 
and hence $K(k_1,k_2,k_3)^{\langle\cb,\al\rangle}$ is rational over $K$. 
When $\ep=-1$, it follows from Theorem \ref{thHK97} that 
$K(k_1,k_2,k_3)^{\langle\cb,\al\rangle}$ is rational over $K$. 
Therefore we conclude that $K(x,y,z)^{G_{5,5,1}}
=K(k_1,k_2,k_3)^{\langle\cb,\alpha\rangle}$ is rational over $K$. 


%
%
\section{The case of $G_{6,j,k}$}\label{se6}

We treat the following eight groups $G=G_{6,j,1}$, $1\leq j\leq 7$, and $G=G_{6,6,2}$ of 
the $6$th crystal system in dimension $3$ which have a normal subgroup $\langle\ca\rangle$: 
\begin{align*}
G_{6,1,1}&=\langle \ca,\ta_1\rangle\cong\mathcal{C}_6,&
G_{6,2,1}&=\langle \ca,-\ta_1\rangle\cong\mathcal{C}_6,\\
G_{6,3,1}&=\langle \ca,\ta_1,-I_3\rangle\cong\mathcal{C}_6\times\mathcal{C}_2,\\
G_{6,4,1}&=\langle \ca,\ta_1,-\be\rangle\cong\mathcal{D}_6,&
G_{6,5,1}&=\langle \ca,\ta_1,\be\rangle\cong\mathcal{D}_6,\\
G_{6,6,1}&=\langle \ca,-\ta_1,\be\rangle\cong\mathcal{D}_6,&
G_{6,6,2}&=\langle \ca,-\ta_1,-\be\rangle\cong\mathcal{D}_6,\\
G_{6,7,1}&=\langle \ca,\ta_1,\be,-I_3\rangle\cong\mathcal{D}_6\times\mathcal{C}_2.
\end{align*}
Note that we changed the generator $-\be$ of $G_{6,7,1}$ by $\be=(-\be)(-I_3)$. 
The actions of $\ca$, $-I_3$, $\ta_1$, $-\ta_1$, $\be$ and $-\be$ on $K(x,y,z)$ are given by
\begin{align*}
\ca &: x\ \mapsto\ ay,\ y\ \mapsto\ \frac{b}{xy},\ z\ \mapsto\ cz,& 
-I_3 &: x\ \mapsto\ \frac{d}{x},\ y\ \mapsto\ \frac{e}{y},\ z\ \mapsto\ \frac{f}{z},\\
\ta_1 &: x\ \mapsto\ \frac{g}{x},\ y\ \mapsto\  \frac{h}{y},\ z\ \mapsto\ iz,& 
-\ta_1 &: x\ \mapsto\ jx,\ y\ \mapsto\ ky,\ z\ \mapsto\ \frac{l}{z},\\
\be &: x\ \mapsto\ \frac{m}{y},\ y\ \mapsto\  \frac{n}{x},\ z\ \mapsto\ oz,& 
-\be &: x\ \mapsto\ py,\ y\ \mapsto\ qx,\ z\ \mapsto\ \frac{r}{z}.
\end{align*}

We may assume that $a=1$ by replacing $ay$ by $y$ and the other coefficients. 
By the equalities $\ca^3=(-I_3)^2=\ta_1^2=\be^2=I_3$, we have $c^3=i^2=j^2=k^2=o^2=pq=1$, $m=n$. 

By the relations of the generators of the $G$'s as in (\ref{relmat6}), we see the following lemma: 
\begin{lemma}
{\rm (i)} If $\ca,-I_3\in G$ then $c=1$, $b^2=d^3$, $d=e\in K^{\times 2}$;\\
{\rm (ii)} If $\ca,\ta_1\in G$ then $b^2=g^3$, $g=h\in K^{\times 2}$;\\
{\rm (iii)} If $\ca,-\ta_1\in G$ then $c=1$, $j=k=1$;\\
{\rm (iv)} If $\ca,\be\in G$ then $c=1$, $b^2=m^3$, $m=n\in K^{\times 2}$;\\
{\rm (v)} If $\ca,-\be\in G$ then $p=q=1$.
\end{lemma}
Thus, the monomial actions may be reduced to the following cases: 
\begin{align*}
\ca=\ca^{(b,c)} &: x\ \mapsto\ y,\ y\ \mapsto\ \frac{b}{xy},\ z\ \mapsto\ cz,& 
-I_3 &: x\ \mapsto\ \frac{1}{x},\ y\ \mapsto\ \frac{1}{y},\ z\ \mapsto\ \frac{f}{z},\\
\ta_1 &: x\ \mapsto\ \frac{1}{x},\ y\ \mapsto\  \frac{1}{y},\ z\ \mapsto\ \ep_1 z,& 
-\ta_1 &: x\ \mapsto\ x,\ y\ \mapsto\ y,\ z\ \mapsto\ \frac{l}{z},\\
\be &: x\ \mapsto\ \frac{1}{y},\ y\ \mapsto\  \frac{1}{x},\ z\ \mapsto\ \ep_2 z,& 
-\be &: x\ \mapsto\ y,\ y\ \mapsto\ x,\ z\ \mapsto\ \frac{r}{z}
\end{align*}
where $b,c,f,l,r\in K^\times$, $c^3=1$ and $\ep_1,\ep_2=\pm 1$. 
We have $c=1$ except for the groups $G_{6,1,1}$ and $G_{6,4,1}$, and 
$b=1$ except for the groups $G_{6,2,1}$ and $G_{6,6,2}$. 
Furthermore, $l=\ep_1f$ for the groups $G_{6,3,1}$ and $G_{6,7,1}$, 
$r=\ep_3l$ where $\ep_3=\pm 1$ for the groups $G_{6,6,2}$ and $G_{6,7,1}$ 
with $\ep_3=\ep_1\ep_2$ for $G_{6,7,1}$. 

\subsection{The cases of $G_{6,1,1}$, $G_{6,5,1}$}

We treat the cases of 
\begin{align*}
G_{6,1,1}=\langle \ca^{(1,c)},\ta_1\rangle,\quad 
G_{6,5,1}=\langle \ca^{(1,1)},\ta_1,\be\rangle.
\end{align*}

By applying Theorem \ref{thAHK} to $L=K(x,y)$ and $M=K$, the rationality problems 
may be reduced to the $2$-dimensional cases of $L=K(x,y)$. 
Hence $K(x,y,z)^{G_{6,1,1}}$ and $K(x,y,z)^{G_{6,5,1}}$ are rational over $K$.

\subsection{The cases of $G_{6,3,1}$, $G_{6,7,1}$}

We treat the cases of 
\begin{align*}
G_{6,3,1}=\langle \ca^{(1,1)},\ta_1,-I_3\rangle,\quad 
G_{6,7,1}=\langle \ca^{(1,1)},\ta_1,\be,-I_3\rangle.
\end{align*}
In these cases, the groups $G$ have non-trivial centers which include $\langle-I_3\rangle$ 
respectively. 
Hence we first consider $K(x,y,z)^{G_{5,2,2}}=K(x,y,z)^{\langle \ca^{(1,1)},-I_3\rangle}$.

As in the previous section, we put $t_1:=\tu(1)$, $t_2:=\tv(1)$, $t_3:=z$ where 
$\tu(b)$, $\tv(b)$ are given in Lemma \ref{lemMas2} and also put 
$u_1:=(t_1-t_2)/(t_1+t_2)$, $u_2:=2/(t_1+t_2)$, $u_3:=t_3$. 

Then, by Lemma \ref{lem522}, we have 
$K(x,y,z)^{G_{5,2,2}}=K(u_1,u_2,u_3)^{\langle -I_3\rangle}=K(r_1,r_2,r_3)$ where 
\begin{align*}
r_1=\frac{u_3^2+f}{u_3},\quad 
r_2=\frac{u_1u_2}{3 u_1^2-u_2^2+1},\quad 
r_3=\frac{(3u_1^2+u_2^2+2 u_2+1)(u_3^2-f)}{(3u_1^2-u_2^2+1)u_3}.
\end{align*}
The actions of $\ta_1$ and $\be$ on $K(u_1,u_2,u_3)^{\langle -I_3\rangle}=K(r_1,r_2,r_3)$ 
are given by 
\begin{align*}
\ta_1\,&:\, r_1\ \mapsto \ep_1 r_1,\ r_2\ \mapsto\ r_2,\ r_3\ \mapsto\ -\ep_1 r_3,\\
\be\,&:\, r_1\ \mapsto\ \ep_2 r_1,\ r_2\ \mapsto\ -r_2,\ r_3\ \mapsto\ -\ep_2 r_3.
\end{align*}
Hence $K(x,y,z)^{G_{6,3,1}}=K(r_1,r_2,r_3)^{\langle\ta_1\rangle}$ and 
$K(x,y,z)^{G_{6,7,1}}=K(r_1,r_2,r_3)^{\langle\ta_1,\be\rangle}$ are rational over $K$. 

\subsection{The cases of $G_{6,2,1}$, $G_{6,6,1}$, $G_{6,6,2}$}

We treat the cases of 
\begin{align*}
G_{6,2,1}&=\langle \ca^{(b,1)},-\ta_1\rangle,\quad 
G_{6,6,1}=\langle \ca^{(1,1)},-\ta_1,\be\rangle,\\
G_{6,6,2}&=\langle \ca^{(b,1)},-\ta_1,-\be\rangle
=\langle \ca^{(b,1)},-\ta_1,\al\rangle.
\end{align*}
Recall that the actions of $\ca^{(b,1)}$, $-\ta_1$, $\al$, $\be$ on $K(x,y,z)$ 
are given as 
\begin{align*}
\ca=\ca^{(b,1)} &: x\ \mapsto\ y,\ y\ \mapsto\ \frac{b}{xy},\ z\ \mapsto\ z,& 
-\ta_1 &: x\ \mapsto\ x,\ y\ \mapsto\ y,\ z\ \mapsto\ \frac{l}{z},\\
\al=(-\ta_1)(-\be) &: x\mapsto y,\ y\ \mapsto x,\, z\mapsto \ep_3 z,& 
\be &: x\ \mapsto\ \frac{1}{y},\ y\ \mapsto\  \frac{1}{x},\ z\ \mapsto\ \ep_2 z. 
\end{align*}

As in the previous section, we put $t_1:=\tu(b)$, $t_2:=\tv(b)$, $t_3:=z$ where 
$\tu(b)$, $\tv(b)$ are given in Lemma \ref{lemMas2} and also put 
$u_1:=(t_1-t_2)/(t_1+t_2)$, $u_2:=2/(t_1+t_2)$, $u_3:=t_3$. 

Then the actions of $-\ta_1$, $\be$ and $\al$ on $K(x,y,z)^{\langle \ca^{(b,1)}\rangle}
=K(u_1,u_2,u_3)$ are given by 
\begin{align*}
-\ta_1\,&:\, u_1\ \mapsto u_1,\ u_2\ \mapsto\ u_2,\ u_3\ \mapsto\ \frac{l}{u_3},\\
\be\,&:\, u_1\ \mapsto\ u_1,\ u_2\ \mapsto\ \frac{3u_1^2+1}{u_2},\ u_3\ \mapsto\ \ep_2 u_3,& 
\al\,&:\, u_1\ \mapsto\ -u_1,\ u_2\ \mapsto\ u_2,\ u_3\ \mapsto\ \ep_3 u_3
\end{align*}
where we adopt $b=1$ in the case of $G_{6,6,1}=\langle \ca^{(1,1)},-\ta_1,\be\rangle$. 
Therefore $K(x,y,z)^{G_{6,2,1}}=K(u_1,u_2,u_3)^{\langle-\ta_1\rangle}=K(v_1,v_2,v_3)$ 
is rational over $K$ where 
\begin{align*}
v_1:=u_1,\quad v_2:=u_2,\quad v_3:=u_3+\frac{l}{u_3}. 
\end{align*}
The actions of $\be$ and $\al$ on $K(x,y,z)^{G_{6,2,1}}=K(v_1,v_2,v_3)$ are given by 
\begin{align*}
\be\,&:\, v_1\ \mapsto\ v_1,\ v_2\ \mapsto\ \frac{3v_1^2+1}{v_2},\ v_3\ \mapsto\ \ep_2 v_3,& 
\al\,&:\, v_1\ \mapsto\ -v_1,\ v_2\ \mapsto\ v_2,\ v_3\ \mapsto\ \ep_3 v_3.
\end{align*}
Hence both $K(x,y,z)^{G_{6,6,1}}=K(v_1,v_2,v_3)^{\langle\be\rangle}$ and 
$K(x,y,z)^{G_{6,6,2}}=K(v_1,v_2,v_3)^{\langle\al\rangle}$ are rational over $K$. 

\subsection{The case of $G_{6,4,1}$}\label{subse641}
In this subsection, we assume that $c=1$. 
The case of $G_{6,4,1}$ with $c\neq 1$ will be discussed in Section \ref{secnot1}.

We treat the cases of 
\begin{align*}
G_{6,4,1}=\langle \ca^{(1,c)},\ta_1,-\be\rangle. 
\end{align*}
Recall that the actions of $\ca^{(1,c)}$, $\ta_1$ and $-\be$ on $K(x,y,z)$ are given as 
\begin{align*}
\ca^{(1,c)} &: x\ \mapsto\ y,\ y\ \mapsto\ \frac{1}{xy},\ z\ \mapsto\ cz,\\
\ta_1 &: x\ \mapsto\ \frac{1}{x},\ y\ \mapsto\  \frac{1}{y},\ z\ \mapsto\ \ep_1 z,\quad 
-\be : x\ \mapsto\ y,\ y\ \mapsto\ x,\ z\ \mapsto\ \frac{r}{z}. 
\end{align*}

We assume that $c=1$ in this subsection.

As in the previous subsection, we have 
$K(x,y,z)^{\langle \ca^{(1,1)}\rangle}=K(u_1,u_2,u_3)$ 
and the actions of $\ta_1$ and of $-\be$ on $K(u_1,u_2,u_3)$ are given by 
\begin{align*}
\ta_1\,&:\, u_1\ \mapsto -u_1,\ u_2\ \mapsto\ \frac{3u_1^2+1}{u_2},\ 
u_3\ \mapsto\ \ep_1 u_3,\\
-\be\,&:\, u_1\ \mapsto\ -u_1,\ u_2\ \mapsto\ u_2,\ u_3\ \mapsto\ \frac{r}{u_3}.
\end{align*}
\begin{lemma}\label{lemqq}
We have $K(x,y,z)^{G_{6,1,1}}=K(u_1,u_2,u_3)^{\langle \ta_1\rangle}=K(q_1,q_2,q_3)$ 
where 
\begin{align*}
q_1&:=u_2+\ta_1(u_2)=\frac{3u_1^2+u_2^2+1}{u_2},\\ 
q_2&:=\frac{u_2-\ta_1(u_2)}{u_1}=-\frac{3u_1^2-u_2^2+1}{u_1u_2},\quad 
q_3:=\begin{cases}
u_3,\ \mathrm{if}\ \ep_1=1,\vspace*{1mm}\\
\displaystyle{\frac{u_3}{u_1}},\ \mathrm{if}\ \ep_1=-1.
\end{cases}
\end{align*}
\end{lemma}
\begin{proof}
The assertion follows from Theorem \ref{thInv} (I). 
Indeed we can check it directly as follows: 
First we see $K(u_1,u_2,u_3)^{\langle\ta_1\rangle}=K(u_1^2,q_1,q_2,q_3)$ because 
$[K(u_1,u_2,u_3) : K(u_1^2,q_1,q_3)]=4$ and $q_2\not\in K(u_1^2,q_1,q_3)$. 
Since $u_1^2=(q_1^2-4(3u_1^2+1))/q_2^2$, we have $u_1^2\in K(q_1,q_2)$ so that 
$K(u_1^2,q_1,q_2,q_3)=K(q_1,q_2,q_3)$. 
\end{proof}

The action of $-\be$ on $K(q_1,q_2,q_3)$ is given by 
\begin{align*}
-\be\,&:\, q_1\mapsto q_1,\ q_2\mapsto -q_2,\ 
\begin{cases}
\displaystyle{q_3\mapsto \frac{r}{q_3},\ \mathrm{if}\ \ep_1=1},\vspace*{1mm}\\
\displaystyle{q_3\mapsto -\frac{r(q_2^2+12)}{(q_1^2-4)q_3},\ \mathrm{if}\ \ep_1=-1}.
\end{cases}
\end{align*}

When $\ep_1=1$, by Theorem \ref{thInv} (I), 
$K(q_1,q_2,q_3)^{\langle-\be\rangle}=K(q_1,q_3+r/q_3,(q_3-r/q_3)/q_2)$ 
is rational over $K$. 
When $\ep_1=-1$, by Theorem \ref{thInv} (I) again, 
$K(q_1,q_2,q_3)^{\langle-\be\rangle}$ is rational over $K$. 
Thus we conclude that $K(x,y,z)^{G_{6,4,1}}=K(q_1,q_2,q_3)^{\langle-\be\rangle}$ 
is rational over $K$. 

%
%
\section{The case where $c\neq 1$}\label{secnot1}

There are $16$ groups $G\subset \mathrm{GL}(3,\bZ)$ 
which contain $\ca$ as in Section \ref{se5A} and Section \ref{se6}. 
The action of $\ca$ on $K(x,y,z)$ is given by 
\[
\ca : x\ \mapsto\ y,\ y\ \mapsto\ \frac{b}{xy},\ z\ \mapsto\ cz
\]
where $b,c\in K^\times$ and $c^3=1$. 
By results in Section \ref{se5A} and Section \ref{se6}, we have $c=1$ except for 
$G=G_{5,1,2}$, $G_{5,3,2}$, $G_{5,3,3}$, $G_{6,1,1}$, $G_{6,4,1}$. 
Note that $K(x,y,z)^{G_{5,1,2}}$ and $K(x,y,z)^{G_{6,1,1}}$ are rational over $K$ 
by Theorem \ref{thAHK}. 

In this section, we consider the case $c\neq 1$ for $G=G_{5,3,2}$, $G_{5,3,3}$, $G_{6,4,1}$. 
Namely we assume that char $K\neq 2$ and 
$K$ contains a primitive cube root of unity $\om$ and $c=\om$. 
We note that $\om^2+\om+1=0$ and $\sqrt{-3}\in K$. 

From results in Section \ref{se5A} and Section \ref{se6}, we have
\begin{align*}
G_{5,3,2}=\langle \ca^{(1,\om)},-\al\rangle,\quad  
G_{5,3,3}=\langle \ca^{(b,\om)},-\be\rangle,\quad 
G_{6,4,1}=\langle \ca^{(1,\om)},\ta_1,-\be\rangle
\end{align*}
where 
\begin{align*}
\ca=\ca^{(b,\om)} &: x\ \mapsto\ y,\ y\ \mapsto\ \frac{b}{xy},\ z\ \mapsto\ \om z,& 
\ta_1 &: x\ \mapsto\ \frac{1}{x},\ y\ \mapsto\  \frac{1}{y},\ z\ \mapsto\ \ep z,\\
-\al &: x\ \mapsto\ \frac{1}{y},\ y\ \mapsto\ \frac{1}{x},\ z\ \mapsto\ \frac{l}{z},& 
-\be &: x\ \mapsto\ y,\ y\ \mapsto\ x,\ z\ \mapsto\ \frac{r}{z}
\end{align*}
with $b,l,r\in K^\times$ and $\ep=\pm 1$. 

We take the Lagrange resolvent 
\[
\Theta:=x+\om\cdot y+\om^2\cdot\frac{b}{xy}=\frac{\om^2b+x^2y+\om xy^2}{xy}
\]
with respect to the $\ca$-orbit of $x$ then we have 
\[
\ca^{(b,\om)}(\Theta)=\om^2 \Theta.
\]
Thus $\Theta z$ is an invariant under the action of $\ca^{(b,\om)}$. 

By results in Subsection \ref{subse532}, we have 
$K(x,y)^{\langle \ca^{(b,\om)}\rangle}=K(u_1,u_2)$ where $u_1$, $u_2$ are given 
as in (\ref{defu123}), and the actions of $\ta_1$, $-\al$, $-\be$ on $K(u_1,u_2)$ are 
\begin{align*}
\ta_1\,&:\, u_1\ \mapsto -u_1,\ u_2\ \mapsto\ \frac{3u_1^2+1}{u_2},\\ 
-\al\,&:\, u_1\ \mapsto\ u_1,\ u_2\ \mapsto\ \frac{3u_1^2+1}{u_2},\quad 
-\be\,:\, u_1\ \mapsto -u_1,\ u_2\ \mapsto\ u_2. 
\end{align*}

Putting
\[
t_3:=\Theta z=\frac{\om^2b+x^2y+\om xy^2}{xy}z,
\]
we have $K(x,y,z)^{\langle \ca^{(b,\om)}\rangle}=K(u_1,u_2,t_3)$ and the actions of 
$\ta_1$, $-\al$, $-\be$ on $K(t_3)$ are as follows: 
\[
\ta_1\,:\, t_3\ \mapsto\ \ep\frac{\ta_1(\Theta)}{\Theta}t_3,\quad 
-\al\,:\, t_3\ \mapsto\ \frac{l}{t_3}{\rm N}_{-\al}(\Theta),\quad 
-\be\,:\, t_3\ \mapsto\ \frac{r}{t_3}{\rm N}_{-\be}(\Theta)
\]
where {\rm N} means the norm operator, so that ${\rm N}_\tau(\Theta)=\tau(\Theta)\Theta$ 
when $\tau$ is of order $2$. 

We consider the case of $G_{5,3,2}=\langle\ca^{(b,\om)},-\al\rangle$ with $b=1$. 
Although $\Theta$ is not $\ca^{(b,\om)}$-invariant, the relations 
$\ca^{(b,\om)}(\Theta)=\om^2\Theta$ and $[\ca^{(b,\om)},-\al]=\ca^{(b,\om)}$ imply that 
$\ca^{(b,\om)}\bigl((-\al)(\Theta)\bigr)=\om(-\al)(\Theta)$. 
Hence
\[
F:=\frac{\Theta^2}{(-\al)(\Theta)}=\frac{(\om^2b+x^2y+\om xy^2)^2}{xy(x+\om y+\om^2bx^2y^2)}
\]
is $\ca^{(b,\om)}$-invariant, so $F\in K(u_1,u_2)$ and 
${\rm N}_{-\al}(F)={\rm N}_{-\al}(\Theta)$. 
Putting $u_3:=t_3/F$, we have $K(x,y,z)^{\langle\ca^{(b,\om)}\rangle}=K(u_1,u_2,u_3)$ 
and 
\[
-\al\,:\, u_1\ \mapsto\ u_1,\ u_2\ \mapsto\ \frac{3u_1^2+1}{u_2},\ u_3\ \mapsto 
\frac{l}{u_3}. 
\]
Hence $K(x,y,z)^{G_{5,3,2}}=K(u_1,u_2,u_3)^{\langle-\al\rangle}$ is 
rational over $K$ (cf. Subsection \ref{subse532}). 

For $G_{5,3,3}=\langle\ca^{(b,\om)},-\be\rangle$, 
by replacing $-\al$ with $-\be$, we can show that 
$K(x,y,z)^{\langle\ca^{(b,\om)}\rangle}=K(u_1,u_2,u_3')$ and 
\[
-\be\,:\, u_1\ \mapsto -u_1,\ u_2\ \mapsto\ u_2,\ u_3'\ \mapsto \frac{r}{u_3'}
\]
where $u_3'=t_3/F'$ and $F'=\Theta^2/\bigl((-\be)(\Theta)\bigr)$. 
Thus $K(x,y,z)^{G_{5,3,3}}=K(u_1,u_2,u_3')^{\langle-\be\rangle}$ is rational over $K$ 
(cf. Subsection \ref{subse532}). 

For $G_{6,4,1}=\langle\ca^{(b,\om)},\ta_1,-\be\rangle$ with $b=1$, we put
\begin{align*}
u_3''&:=\Bigl(\frac{\tau_1(\Theta)}{\Theta}+1\Bigr)t_3
=\Bigl(\frac{\om x+y+b\om ^2x^2y^2}{b\om ^2+x^2y+\om xy^2}+1\Bigr)t_3\\
&\ =\frac{(b\om ^2+\om x+y+x^2y+\om xy^2+b\om ^2x^2y^2)z}{xy},
\end{align*}
then we have $\ta_1 : u_3''\mapsto \ep u_3''$. 
Note that $\ta_1(\Theta)/\Theta$ is $\ca^{(b,\om)}$-invariant, so that 
$\ta_1(\Theta)/\Theta\in K(u_1,u_2)$ and 
$K(x,y,z)^{\langle\ca^{(b,\om)}\rangle}=K(u_1,u_2,u_3'')$. 
By Lemma \ref{lemqq}, we get 
\[
K(x,y,z)^{G_{6,1,1}}=K(u_1,u_2,u_3'')^{\langle \ta_1\rangle}=K(q_1,q_2,q_3)
\]
where 
\begin{align*}
q_1&:=u_2+\ta_1(u_2)=\frac{3u_1^2+u_2^2+1}{u_2},\\ 
q_2&:=\frac{u_2-\ta_1(u_2)}{u_1}=-\frac{3u_1^2-u_2^2+1}{u_1u_2},\quad 
q_3:=\begin{cases}
u_3'',\ \mathrm{if}\ \ep=1,\vspace*{1mm}\\
\displaystyle{\frac{u_3''}{u_1}},\ \mathrm{if}\ \ep=-1.
\end{cases}
\end{align*}
The action of $-\be$ on $K(q_1,q_2,q_3)$ is given by 
\begin{align*}
-\be\,&:\, q_1\ \mapsto\ q_1,\ q_2\ \mapsto\ -q_2,\ 
\begin{cases}
\displaystyle{q_3\ \mapsto\ \frac{r}{q_3}{\rm N}_{-\be}\Bigl(\Theta+\ta_1(\Theta)\Bigr),\ 
\mathrm{if}\ \ep=1},\vspace*{1mm}\\
\displaystyle{q_3\ \mapsto\ 
-\frac{r(q_2^2+12)}{(q_1^2-4)q_3}{\rm N}_{-\be}\Bigl(\Theta+\ta_1(\Theta)\Bigr),\ 
\mathrm{if}\ \ep=-1}.
\end{cases}
\end{align*}
Since 
\[
F'':=\frac{(\Theta+\ta_1(\Theta))^2}{(-\be)(\Theta+\ta_1(\Theta))}
=\frac{(b\om ^2+\om x+y+x^2y+\om xy^2+b\om ^2x^2y^2)^2}{xy(b\om ^2+x+\om y+\om x^2y+xy^2+b\om ^2x^2y^2)}
\]
is $G_{6,1,1}$-invariant, we have $F''\in K(q_1,q_2)$ and 
${\rm N}_{-\be}(\Theta+\ta_1(\Theta))={\rm N}_{-\be}(F'')$. 

So putting $q_3':=q_3/F''$, we have $K(x,y,z)^{G_{6,1,1}}=K(q_1,q_2,q_3')$ and 
\begin{align*}
-\be\,&:\, q_1\ \mapsto\ q_1,\ q_2\ \mapsto\ -q_2,\ 
\begin{cases}
\displaystyle{q_3'\ \mapsto\ \frac{r}{q_3'},\ \mathrm{if}\ \ep=1},\vspace*{1mm}\\
\displaystyle{q_3'\ \mapsto\ 
-\frac{r(q_2^2+12)}{(q_1^2-4)q_3'},\ \mathrm{if}\ \ep=-1}.
\end{cases}
\end{align*}
By Theorem \ref{thInv} (I), 
$K(x,y,z)^{G_{6,4,1}}=K(q_1,q_2,q_3')^{\langle-\be\rangle}$ is rational over $K$ 
(cf. Subsection \ref{subse641}). 
 

%
%
\section{The case of $G_{7,j,1}$}\label{se71}

In this section, we consider the following five groups $G=G_{7,j,1}$, $1\leq j\leq 5$, 
which have a normal subgroup $\langle\ta_1,\la_1\rangle$: 
\begin{align*}
\hspace*{1.5cm}
G_{7,1,1}&=\langle \ta_1,\la_1,\cb\rangle&\hspace{-1cm} 
&\cong \mathcal{A}_4&\hspace{-1cm} 
&\cong (\mathcal{C}_2\times \mathcal{C}_2)\rtimes \mathcal{C}_3, \\
G_{7,2,1}&=\langle \ta_1,\la_1,\cb,-I_3\rangle&\hspace{-1cm} 
&\cong \mathcal{A}_4\times \mathcal{C}_2&\hspace{-1cm} 
&\cong (\mathcal{C}_2\times \mathcal{C}_2\times \mathcal{C}_2)\rtimes \mathcal{C}_3,\\
G_{7,3,1}&=\langle \ta_1,\la_1,\cb,-\be_1\rangle&\hspace{-1cm} 
&\cong \mathcal{S}_4&\hspace{-1cm} 
&\cong (\mathcal{C}_2\times \mathcal{C}_2)\rtimes \mathcal{S}_3, \\
G_{7,4,1}&=\langle \ta_1,\la_1,\cb,\be_1\rangle&\hspace{-1cm} 
&\cong \mathcal{S}_4&\hspace{-1cm} 
&\cong (\mathcal{C}_2\times \mathcal{C}_2)\rtimes \mathcal{S}_3, \\
G_{7,5,1}&=\langle \ta_1,\la_1,\cb,\be_1,-I_3\rangle&\hspace{-1cm} 
&\cong \mathcal{S}_4\times \mathcal{C}_2&\hspace{-1cm} 
&\cong (\mathcal{C}_2\times \mathcal{C}_2\times \mathcal{C}_2)\rtimes \mathcal{S}_3.
\end{align*}

Note that $\ta_1=\ta$ and $\be_1=\be$ (cf. Sections \ref{se5A}, \ref{se5B}, \ref{se6}). 

The actions of $\ta_1$, $\la_1$, $\cb$, $-\be_1$, $\be_1$ and $-I_3$ on $K(x,y,z)$ are given by
\begin{align*}
\ta_1 &: x\ \mapsto\ \frac{a}{x},\ y\ \mapsto\ \frac{b}{y},\ z\ \mapsto\  cz,& 
\la_1 &: x\ \mapsto\ \frac{d}{x},\ y\ \mapsto\ ey,\ z\ \mapsto\ \frac{f}{z},\\
\cb &: x\ \mapsto\ gy,\ y\ \mapsto\  hz,\ z\ \mapsto\ ix,& 
-\be_1 &: x\ \mapsto\ jy,\ y\ \mapsto\  kx,\ z\ \mapsto\ \frac{l}{z},\\
\be_1 &: x\ \mapsto\ \frac{m}{y},\ y\ \mapsto\ \frac{n}{x},\ z\ \mapsto\  oz,& 
-I_3 &: x\ \mapsto\ \frac{p}{x},\ y\ \mapsto\ \frac{q}{y},\ z\ \mapsto\ \frac{r}{z}.
\end{align*}

We may assume that $g=h=i=1$ by replacing $(gy,ghz)$ by $(y,z)$ and the other coefficients.

By the equalities $\ta_1^2=\la_1^2=\cb^3=(-\be_1)^2=\be_1^2=(-I_3)^2=I_3$, 
we have $c^2=e^2=jk=o^2=1$ and $m=n$. 

By the relations of the generators of $G_{7,1,1}$ as in (\ref{relmat7}), 
the monomial action of $G_{7,1,1}$ is written as 
\begin{align}
\ta_1 &: x\ \mapsto\ \frac{a}{x},\ y\ \mapsto\ \frac{\ep_1 a}{y},\ z\ \mapsto\ \ep_1 z,&
\la_1 &: x\ \mapsto\ \frac{\ep_1 a}{x},\ y\ \mapsto\ \ep_1 y,\ z\ \mapsto\ \frac{a}{z},
\label{act711}\\
\cb &: x\ \mapsto\  y,\ y\ \mapsto\  z,\ z\ \mapsto\ x,\nonumber
\end{align}
where $a\in K^\times$, $\ep_1=\pm 1$. 

By the relations of $-\be_1$, $\be_1$, $-I_3$ with the generators of $G_{7,1,1}$, 
the monomial actions of $-\be_1$, $\be_1$, $-I_3$ are written as 
\begin{align}
-\be_1 &: x\ \mapsto\ jy,\ y\ \mapsto\ \frac{x}{j},\ z\ \mapsto\ \frac{ja}{z},& 
\be_1 &: x\ \mapsto\ \frac{\ep_3 a}{y},\ y\ \mapsto\ \frac{\ep_3 a}{x},\ 
z\ \mapsto\ \ep_3 z,& \label{act712}\\
-I_3 &: x\ \mapsto\ \frac{\ep_4 a}{x},\ y\ \mapsto\ \frac{\ep_4 a}{y},\ z\ \mapsto\ 
\frac{\ep_4 a}{z}\nonumber
\end{align}
where 
\[
j=\begin{cases}\ep_2,\hspace*{10.3mm} \mathrm{if}\ \ep_1=1,\\
\ep_2\sqrt{-1},\ \mathrm{if}\ \ep_1=-1\end{cases}
\]
and $\ep_2,\ep_3,\ep_4=\pm 1$. 

The action of $\be_1$ is possible only when $\ep_1=1$, so that the case $\ep_1=-1$ does 
not appear for $G=G_{7,4,1}$ and $G_{7,5,1}$, and also for $G_{7,3,1}$ if 
$\sqrt{-1}\not\in K$. 
For the group $G_{7,5,1}$ with $\ep_1=1$, we have $\ep_2=\ep_3\ep_4$. 
When $\ep_1=-1$, we may assume that $\ep_2=1$ by replacing $-\sqrt{-1}$ by $\sqrt{-1}$. 

\subsection{The case of $\ep_1=1$}\label{subse71p}

We treat the case where $\ep_1=1$ in this subsection. 
Although $G_{3,1,1}=\langle\ta_1,\la_1\rangle\in\mathcal{N}$, by Theorem \ref{thmex}, we have 
\begin{align*}
K(x,y,z)^{\langle \ta_1,\la_1\rangle}=K(v_1,v_2,v_3)
\end{align*}
where
\begin{align}
v_1:=\frac{a(-x+y+z)-xyz}{a-xy-xz+yz},\ \ 
v_2:=\frac{a(x-y+z)-xyz}{a-xy+xz-yz},\ \ 
v_3:=\frac{a(x+y-z)-xyz}{a+xy-xz-yz}.\label{defvv}
\end{align}

The actions of $\cb$, $-\be_1$, $\be_1$ and $-I_3$ on 
$K(v_1,v_2,v_3)$ are given by 
\begin{align*}
\cb &: v_1\ \mapsto\ v_2,\ v_2\ \mapsto\  v_3,\ v_3\ \mapsto\ v_1,& 
-\be_1 &: v_1\ \mapsto\ \frac{\ep_2 a}{v_2},\ v_2\ \mapsto\ \frac{\ep_2 a}{v_1},\ 
v_3\ \mapsto\ \frac{\ep_2 a}{v_3},\\
\be_1 &: v_1\ \mapsto\ \ep_3 v_2,\ v_2\ \mapsto\ \ep_3 v_1,\ v_3\ \mapsto\ \ep_3 v_3,& 
-I_3 &: v_1\ \mapsto\ \frac{\ep_4 a}{v_1},\ v_2\ \mapsto\ \frac{\ep_4 a}{v_2},\ 
v_3\ \mapsto\ \frac{\ep_4 a}{v_3}.
\end{align*}

Hence the groups $G_{7,j,1}$, $1\leq j\leq 5$, act on 
$K(x,y,z)^{\langle\ta_1,\la_1\rangle}$ by monomial actions as $\mathcal{C}_3$, 
$\mathcal{C}_6$, $\mathcal{S}_3$, $\mathcal{S}_3$, $\mathcal{D}_6$ respectively. 
We already treated these cases as in Section \ref{se5B}. 
Therefore the fixed fields $K(x,y,z)^{G_{7,j,1}}$, $1\leq j\leq 5$, 
are rational over $K$. 

We can also get explicit transcendental bases of the fixed fields over $K$ 
by results of Section \ref{se5B}. 
For example, by Theorem \ref{thaaa}, we see 
\begin{align*}
K(x,y,z)^{\langle \ta_1,\la_1,-I_3\rangle}&=K(v_1,v_2,v_3)^{\langle -I_3\rangle}
=K(k_1,k_2,k_3)
\end{align*}
where
\begin{align*}
k_1:=\frac{2(v_2v_3+\ep_4 a)}{v_2+v_3},\quad 
k_2:=\frac{2(v_1v_3+\ep_4 a)}{v_1+v_3},\quad 
k_3:=\frac{2(v_1v_2+\ep_4 a)}{v_1+v_2}. 
\end{align*}

By the definition of $v_1$, $v_2$ and $v_3$, we see 
\begin{align*}
\begin{cases}
\displaystyle{k_1=\frac{x^2+a}{x},\quad k_2=\frac{y^2+a}{y},\quad k_3=\frac{z^2+a}{z},\quad 
\mathrm{if}\ \ep_4=1},\vspace*{1mm}\\
\displaystyle{k_1=\frac{(x^2-a)(a^2+ay^2+az^2-4ayz+y^2z^2)}{x(y^2-a)(z^2-a)},\quad 
k_2=\cb(k_1),\quad k_3=\cb^2(k_1),\quad \mathrm{if}\ \ep_4=-1}.
\end{cases}
\end{align*}
Because the action of $\cb$ on $K(k_1,k_2,k_3)$ is given by 
\begin{align*}
\cb &: k_1\ \mapsto\ k_2,\ k_2\ \mapsto\  k_3,\ k_3\ \mapsto\ k_1,
\end{align*}
we can get an explicit transcendental basis of 
$K(x,y,z)^{G_{7,2,1}}=K(k_1,k_2,k_3)^{\langle \cb\rangle}$ 
over $K$ by Lemma \ref{lemMas}.

\subsection{The cases of $G_{7,1,1}$ and $G_{7,3,1}$ with $\ep_1=-1$}

We treat the case where $\ep_1=-1$. 
In this case, $G_{7,4,1}$ and $G_{7,5,1}$ do not appear and 
we first treat the cases of 
\[
G_{7,1,1}=\langle \ta_1,\la_1,\cb\rangle,\quad 
G_{7,3,1}=\langle \ta_1,\la_1,\cb,-\be_1\rangle.
\]
For $G_{7,2,1}$, see the next subsection. 
The actions of $\ta_1$ ,$\la_1$, $\cb$ and $-\be_1$ on $K(x,y,z)$ are given 
as in (\ref{act711}) and (\ref{act712}). 
Note that $G_{7,3,1}$ appears only in the case where $K\ni\sqrt{-1}$. 
We will show that $K(x,y,z)^{G_{7,1,1}}$ is rational over $K$ 
if $[K(\sqrt{a},\sqrt{-1}) : K]\leq 2$ and $K(\sqrt{-1})(x,y,z)^{G_{7,3,1}}$ is rational over 
$K(\sqrt{-1})$. 

When $[K(\sqrt{a},\sqrt{-1}) : K]=4$, we do not know whether $K(x,y,z)^{G_{7,1,1}}$ 
is rational over $K$ or not. 
Note that the group $G_{7,1,1}$ has a normal subgroup 
$\langle \ta_1,\la_1\rangle=G_{3,1,1}\in\mathcal{N}$ and $K(x,y,z)^{G_{3,1,1}}$ is rational 
over $K$ if and only if $[K(\sqrt{a},\sqrt{-1}) : K]\leq 2$ (see also Section \ref{seintro}). 
The group $G_{7,1,1}$ has no normal subgroup other than $G_{3,1,1}$, 
and hence the reduction of the problem to a factor group does not work for $G_{7,1,1}$ 
if $[K(\sqrt{a},\sqrt{-1}) : K]=4$. 

On the contrary, the group $G_{7,2,1}$ has a normal subgroup 
$G_{3,3,1}=\langle\ta_1,\la_1,-I_3\rangle$, and the monomial action of 
$G_{3,3,1}$ with $\ep_1=-1$ is an affirmative case of 
$G_{3,3,1}\in\mathcal{N}$ (see the next subsection). 
\begin{lemma}\label{lemex2}
Let $L=K(\sqrt{a},\sqrt{-1})$. 
The fixed field $L(x,y,z)^{\langle\ta_1,\la_1\rangle}$ is rational over $L$ 
and an explicit transcendental basis of 
$L(x,y,z)^{\langle\ta_1,\la_1\rangle}=L(w_1,w_2,w_3)$ over $L$ is given by 
\begin{align*}
w_1&=\frac{(a+\sqrt{a}\sqrt{-1}x+\sqrt{a}y-\sqrt{-1}xy)
(a+\sqrt{a}\sqrt{-1}y-\sqrt{a}z+\sqrt{-1}yz)}
{(a-\sqrt{a}\sqrt{-1}x-\sqrt{a}y-\sqrt{-1}xy)
(-\sqrt{-1}a+\sqrt{a}y-\sqrt{a}\sqrt{-1}z-yz)},\\
w_2&=\frac{(a-\sqrt{a}x-\sqrt{a}\sqrt{-1}z-\sqrt{-1}xz)
(a+\sqrt{a}\sqrt{-1}y-\sqrt{a}z+\sqrt{-1}yz)}
{(a-\sqrt{a}x+\sqrt{a}\sqrt{-1}z+\sqrt{-1}xz)
(a+\sqrt{a}\sqrt{-1}y+\sqrt{a}z-\sqrt{-1}yz)},\\
w_3&=\frac{(a+\sqrt{a}\sqrt{-1}x+\sqrt{a}y-\sqrt{-1}xy)
(a-\sqrt{a}x-\sqrt{a}\sqrt{-1}z-\sqrt{-1}xz)}
{(-\sqrt{-1}a-\sqrt{a}x+\sqrt{a}\sqrt{-1}y-xy)
(a-\sqrt{a}x+\sqrt{a}\sqrt{-1}z+\sqrt{-1}xz)}.
\end{align*}
\end{lemma}
\begin{proof}
Put 
\[
X:=\frac{x+\sqrt{a}}{x-\sqrt{a}},\quad 
Y:=\frac{y+\sqrt{a}\sqrt{-1}}{y-\sqrt{a}\sqrt{-1}},\quad 
Z:=\frac{z}{\sqrt{a}}. 
\]
Then $L(x,y,z)=L(X,Y,Z)$ and the actions of $\ta_1$ and $\la_1$ on $L(X,Y,Z)$ are given by
\begin{align*}
\ta_1\ :\ X\ \mapsto\ -X,\ Y\mapsto\ -Y,\ Z\mapsto\ -Z,\quad 
\la_1\ :\ X\ \mapsto\ -\frac{1}{X},\ Y\mapsto\ \frac{1}{Y},\ Z\mapsto\ \frac{1}{Z}.
\end{align*}
Hence we get $L(x,y,z)^{\langle\ta_1\rangle}=L(u_1,u_2,u_3)$ where 
\[
u_1:=\sqrt{-1}XY,\quad u_2:=YZ,\quad u_3:=\sqrt{-1}XZ
\]
and the action of $\la_1$ on $L(u_1,u_2,u_3)$ is given by 
\[
\la_1\ :\ u_1\ \mapsto\ \frac{1}{u_1},\ u_2\mapsto\ \frac{1}{u_2},\ u_3\mapsto\ \frac{1}{u_3}.
\]
Put 
\[
v_1:=\frac{u_1+1}{u_1-1},\quad v_2:=\frac{u_2+1}{u_2-1},\quad v_3:=\frac{u_3+1}{u_3-1}.
\]
Then the action of $\la_1$ on $L(x,y,z)^{\langle\ta_1\rangle}=L(u_1,u_2,u_3)=L(v_1,v_2,v_3)$ is 
given by 
\begin{align*}
\la_1\ :\ v_1\ \mapsto\ -v_1,\ v_2\mapsto\ -v_2,\ v_3\mapsto\ -v_3.
\end{align*}
Hence we get $L(x,y,z)^{\langle\ta_1,\la_1\rangle}=L(w_1,w_2,w_3)$ where 
\[
w_1:=v_1v_2,\quad w_2:=v_2v_3,\quad w_3:=v_3v_1. 
\]
We can evaluate the ${w_i}$'s in terms of $x,y,z$ by the definitions above. 
\end{proof}

Now we take 
\[
L:=K(\sqrt{a},\sqrt{-1})
\]
and $\mathrm{Gal}(L/K)=\langle\rho_{a},\rho_{-1}\rangle$ 
where 
\begin{align*}
\rho_a\ :\ \sqrt{a}\mapsto -\sqrt{a},\ \sqrt{-1}\mapsto \sqrt{-1},\\
\rho_{-1}\ :\ \sqrt{a}\mapsto \sqrt{a},\ \sqrt{-1}\mapsto -\sqrt{-1}. 
\end{align*}
Put $\rho_{-a}=\rho_a\circ\rho_{-1}$. 
We extend the actions of $G_{7,1,1}$ and $G_{7,3,1}$ to $L(x,y,z)$ with trivial actions on $L$. 
The actions of $\cb$, $-\be_1$, $\rho_{a}$, $\rho_{-1}$ and $\rho_{-a}$ 
on $L(x,y,z)^{\langle\ta_1,\la_1\rangle}=L(w_1,w_2,w_3)$ are given by 
\begin{align}
\cb &: w_1\, \mapsto\, \frac{w_2-w_1w_3}{w_2(w_1-1)(w_3-1)},\ 
w_2\, \mapsto\, \frac{w_3-w_1w_2}{w_3(w_1-1)(w_2-1)},\ 
w_3\, \mapsto\, \frac{w_1-w_2w_3}{w_1(w_2-1)(w_3-1)},\label{actww}\\
-\be_1 &: 
w_1\, \mapsto\, \frac{1}{w_3},\ 
w_2\, \mapsto\, \frac{1}{w_2},\ 
w_3\, \mapsto\, \frac{1}{w_1},\nonumber\\
\rho_a &: 
w_1\, \mapsto\, -\frac{w_1-w_2}{w_1(w_3-1)},\ 
w_2\, \mapsto\, -\frac{w_3(w_1-1)(w_1-w_2)}{w_1(w_3-1)(w_2-w_3)},\ 
w_3\, \mapsto\, \frac{w_2(w_1-1)}{w_1(w_2-w_3)},\nonumber\\
\rho_{-1} &: 
w_1\, \mapsto\, \frac{w_2(w_3-1)(w_1-w_3)}{w_3(w_2-1)(w_1-w_2)},\ 
w_2\, \mapsto\, \frac{w_1(w_3-1)}{w_3(w_1-w_2)},\ 
w_3\, \mapsto\, \frac{w_1-w_3}{w_3(w_2-1)},\nonumber\\
\rho_{-a} &: 
w_1\, \mapsto\, -\frac{w_3(w_2-1)}{w_2(w_1-w_3)},\ 
w_2\, \mapsto\, -\frac{w_2-w_3}{w_2(w_1-1)},\ 
w_3\, \mapsto\, \frac{w_1(w_2-1)(w_2-w_3)}{w_2(w_1-1)(w_1-w_3)}\nonumber
\end{align}
(we omit the display of the actions on $L$).
\begin{lemma}\label{lemtri}
The following three conditions are equivalent:\\
{\rm (i)} $K(\sqrt{-1})(x,y,z)^{G_{7,1,1}}=L(w_1,w_2,w_3)^{\langle\cb,\rho_{a}\rangle}$ 
is rational over $K(\sqrt{-1})=L^{\langle\rho_{a}\rangle}$;\\
{\rm (ii)} $K(\sqrt{a})(x,y,z)^{G_{7,1,1}}=L(w_1,w_2,w_3)^{\langle\cb,\rho_{-1}\rangle}$ 
is rational over $K(\sqrt{a})=L^{\langle\rho_{-1}\rangle}$;\\
{\rm (iii)} $K(\sqrt{-a})(x,y,z)^{G_{7,1,1}}=L(w_1,w_2,w_3)^{\langle\cb,\rho_{-a}\rangle}$ 
is rational over $K(\sqrt{-a})=L^{\langle\rho_{-a}\rangle}$. 
\end{lemma}
\begin{proof}
We take the $L$-automorphism 
\[
\eta : w_1\mapsto w_2\mapsto w_3\mapsto w_1
\]
on $L(w_1,w_2,w_3)$. Then the assertion follows from the equalities
\[
\eta^{-1}\cb\eta=\cb,\quad \eta^{-1}\rho_a\eta=\rho_{-1},\quad 
\eta^{-1}\rho_{-1}\eta=\rho_{-a},\quad \eta^{-1}\rho_{-a}\eta=\rho_a
\]
as $K$-automorphisms. 
\end{proof}
Hence we should consider only $\rho_a$ instead of $\rho_{-1}$, $\rho_{-a}$. 
In order to linearlize the action of $\cb$, 
we take the $\cb$-orbit of $w_1$ as 
\[
q_1:=w_1,\quad q_2:=\cb(w_1)=\frac{w_2-w_1w_3}{w_2(w_1-1)(w_3-1)},\quad 
q_3:=\cb^2(w_1)=\frac{(w_1-1)(w_2-1)}{w_1w_2-w_3}.
\]
Then we see 
\[
[L(w_1,w_2,w_3):L(q_1,q_2,q_3)]=2
\]
because $L(w_1,w_2,w_3)=L(q_1,q_2,q_3)(w_3)$ and 
\[
f(w_3):=q_2q_3w_3^2-(q_1-q_2-q_3+q_1q_2+q_2q_3-q_1q_3)w_3+q_1q_2-q_2+1=0.
\]
We take the square root of the discriminant of the quadratic polynomial $f(w_3)$: 
\begin{align*}
q_4&:=\sqrt{(q_1-q_2-q_3+q_1q_2+q_2q_3-q_1q_3)^2-4q_2q_3(q_1q_2-q_2+1)}\\
&\ =\frac{w_1^2w_2-w_1w_2^2-w_1w_3+w_2^2w_3+w_1w_3^2-w_1w_2w_3^2}
{w_2(w_3-1)(w_1w_2-w_3)}. 
\end{align*}
\begin{lemma}
Let $L=K(\sqrt{-1},\sqrt{a})$. 
The fixed field $L(x,y,z)^{G_{3,1,1}}$ is given by 
\[
L(x,y,z)^{G_{3,1,1}}=L(w_1,w_2,w_3)=L(q_1,q_2,q_3)(q_4)
\]
where $q_1,q_2,q_3,q_4$ satisfy the equality 
\begin{align}
q_4^2=(q_1-q_2-q_3+q_1q_2+q_2q_3-q_1q_3)^2-4q_2q_3(q_1q_2-q_2+1).\label{eqq4}
\end{align}
\end{lemma}
We note that $q_4$ is $\cb$-invariant. 
The actions of $\cb$, $-\be_1$ and $\rho_a$ on 
$L(x,y,z)^{\langle\ta_1,\la_1\rangle}=L(q_1,\ldots,q_4)$ are given by
\begin{align}
\cb &: \ q_1\ \mapsto\ q_2,\ q_2\ \mapsto\ q_3,\ q_3\ \mapsto\ q_1,\ q_4\ \mapsto\ q_4,\nonumber\\
-\be_1 &:\ 
q_1\ \mapsto\ \frac{q_1q_2+q_2q_3-q_3q_1+q_1-q_2-q_3-q_4}{2(q_1q_2-q_2+1)},\label{actq}\\ 
&\quad q_2\ \mapsto\ \frac{q_1q_2-q_2q_3+q_3q_1-q_1-q_23+q_3-q_4}{2(q_3q_1-q_1+1)},\nonumber\\ 
&\quad q_3\ \mapsto\ \frac{-q_1q_2+q_2q_3+q_3q_1-q_1+q_2-q_3-q_4}{2(q_2q_3-q_3+1)},\nonumber\\ 
&\quad q_4\ \mapsto\ \frac{C_q+(1-q_1-q_2-q_3+q_1q_2+q_1q_3+q_2q_3-2q_1q_2q_3)q_4}
{2(q_1q_2-q_2+1)(q_3q_1-q_1+1)(q_2q_3-q_3+1)},\nonumber\\
\rho_a &:\ \sqrt{a}\ \mapsto\ -\sqrt{a},\nonumber\\ 
&\quad 
q_1\ \mapsto\ -\frac{q_1q_2^2-q_3q_1^2+q_1^2-2q_1q_2-q_2q_3q_1+q_2q_3-q_1+q_2-q_3-q_4q_1+q_4}
{2q_1(q_3q_1-q_1+1)},\nonumber\\
&\quad 
q_2\ \mapsto\ \frac{q_1q_2^2-q_3q_2^2-q_2^2+q_1q_3q_2+2q_3q_2-q_1q_3+q_1+q_2-q_3+q_4q_2-q_4}
{2q_2(q_1q_2-q_2+1)},\nonumber\\
&\quad  
q_3\ \mapsto\ \frac{-q_1q_3^2+q_2q_3^2-q_3^2+2q_1q_3+q_1q_2q_3-q_1q_2-q_1+q_2+q_3+q_4q_3-q_4}
{2q_3(q_2q_3-q_3+1)},\nonumber\\
&\quad  
q_4\ \mapsto\ \frac{(q_1q_2q_3+1)(-q_1q_2+2q_2q_3q_1-q_2q_3-q_3q_1+q_1+q_2+q_3-q_4-2)}
{2(q_1q_2-q_2+1)(q_3q_1-q_1+1)(q_2q_3-q_3+1)}\nonumber
\end{align}
where 
\begin{align*}
C_q=&-q_1+q_1^2-q_2+q_1q_2+q_2^2-2q_1q_2^2+q_1^2q_2^2-q_3+q_1q_3-2q_1^2q_3\\
&+q_2q_3-4q_1q_2q_3+2q_1^2q_2q_3+2q_1q_2^2q_3-2q_1^2q_2^2q_3+q_3^2+q_1^2q_3^2\\
&-2q_2q_3^2+2q_1q_2q_3^2-2q_1^2q_2q_3^2+q_2^2q_3^2-2q_1q_2^2q_3^2+2q_1^2q_2^2q_3^2. 
\end{align*}
\begin{lemma}Let $L=K(\sqrt{a},\sqrt{-1})$. 
Then we have 
\[
L(x,y,z)^{G_{7,1,1}}=L(w_1,w_2,w_3)^{\langle\cb\rangle}
=L(q_1,\ldots,q_4)^{\langle\cb\rangle}=L(s,u,v)
\]
where $s=q_4+s_1(1-u+v)$ and $s_1$, $u$, $v$ are given by 
\begin{align*}
s_1&=s_1(q_1,q_2,q_3)=q_1+q_2+q_3,\\
u&=u(q_1,q_2,q_3)
=\frac{q_1q_2^2+q_2q_3^2+q_3q_1^2-3q_1q_2q_3}{q_1^2+q_2^2+q_3^2-q_1q_2-q_2q_3-q_3q_1},\\
v&=v(q_1,q_2,q_3)
=\frac{q_1^2q_2+q_2^2q_3+q_3^2q_1-3q_1q_2q_3}{q_1^2+q_2^2+q_3^2-q_1q_2-q_2q_3-q_3q_1}. 
\end{align*}
In particular, if $\sqrt{-1},\sqrt{a}\in K$ then $K(x,y,z)^{G_{7,1,1}}$ is rational over $K$. 
\end{lemma}
\begin{proof}
We see $L(x,y,z)^{G_{7,1,1}}=L(q_1,q_2,q_3,q_4)^{\langle\cb\rangle}$ and the action of $\cb$ 
on $L(q_1,q_2,q_3,q_4)$ is given by $\cb : q_1\mapsto q_2\mapsto q_3\mapsto q_1$, 
$q_4\mapsto q_4$. 
By Lemma \ref{lemMas}, we get 
\[
L(x,y,z)^{G_{7,1,1}}=L(s_1,u,v,q_4). 
\] 
The relation (\ref{eqq4}) becomes 
\[
q_4^2=s_1^2-4s_2+2s_1s_2+s_2^2-4s_3-4s_1s_3-4(q_1q_2^2+q_2q_3^2+q_3q_1^2)
\]
where $s_1=q_1$, $s_2=q_1q_2+q_2q_3+q_3q_1$, $s_3=q_1q_2q_3$ are the elementary 
symmetric functions in $q_1$, $q_2$, $q_3$. 
We also see 
\begin{align*}
&s_2=s_1(u+v)-3(u^2-uv+v^2),\\
&s_3=s_1uv-(u^3+v^3),\\
&q_1q_2^2+q_2q_3^2+q_3q_1^2=s_1^2u-3s_1u^2+3(2u-v)(u^2-uv+v^2)
\end{align*}
(see, for example, \cite[Theorem 2.2]{HK10}). 
Thus we get 
\begin{align}
q_4^2&=s_1^2(1-u+v)^2-2s_1(2u-3u^2+u^3+2v-uv+3v^2+v^3)\label{eqq4s1}\\
&\quad +(u^2-uv+v^2)(12-20u+9u^2+16v-9uv+9v^2).\nonumber
\end{align}
Hence we put 
\[
s:=q_4+s_1(1-u+v),\ t:=q_4-s_1(1-u+v),\quad 
\Bigl(s_1=\frac{t-s}{2(1-u+v)},\ q_4=\frac{s+t}{2}\Bigr). 
\]
Then $L(s_1,u,v,q_4)=L(s,t,u,v)$ and the relation (\ref{eqq4s1}) becomes linear in 
both $s$ and $t$. 
Thus we get $L(x,y,z)^{G_{7,1,1}}=L(s,t,u,v)=L(s,u,v)$.
\end{proof}
We evaluate the actions of $-\be_1$ and $\rho_a$ on $L(s,u,v)$ as follows: 
First we see 
\begin{align}
q_1&=\frac{q_3v-u^2+uv-v^2}{q_3-u}, 
q_2=\frac{q_3u-u^2+uv-v^2}{q_3-v},\label{qtos}\\
q_4&=\Bigl(q_3^3(-1+u-v)+q_3^2s+q_3(-su+3u^2-3u^3-sv-3uv+6u^2v+3v^2\nonumber\\
&\quad\ -6uv^2+3v^3)-u^3+u^4+suv-u^3v-v^3+uv^3-v^4\Bigr)\Big/\Bigl((q_3-u)(q_3-u)\Bigr).\nonumber
\end{align}
Then $L(q_1,\ldots,q_4)=L(s,u,v)(q_3)$ is a cyclic cubic extension over $L(s,u,v)$. 
The minimal polynomial of $q_3$ over $L(s,u,v)$ is given by 
\begin{align}
&2(-s+2u+su-3u^2+u^3+2v-sv-uv+3v^2+v^3)q_3^3\label{eqq3}\\
&+(s^2-12u^2+20u^3-9u^4+12uv-36u^2v+18u^3v-12v^2+36uv^2-27u^2v^2\nonumber\\
&\quad\ -16v^3+18uv^3-9v^4)q_3^2\nonumber\\
&+(-s^2u+6su^2-6su^3-2u^4+3u^5-s^2v-6suv+12su^2v+4u^3v-3u^4v\nonumber\\
&\quad\ +6sv^2-12suv^2-6u^2v^2+3u^3v^2+6sv^3+4uv^3+3u^2v^3-2v^4-3uv^4+3v^5)q_3\nonumber\\
&+(su-2u^2+u^3+2uv-3u^2v-2v^2+3uv^2-2v^3)(-2u^2+2u^3+sv+2uv\nonumber\\
&\quad\ -3u^2v-2v^2+3uv^2-v^3)\nonumber
\end{align}
with square discriminant $(u-v)^2 f(s,u,v)^2$ where $f(s,u,v)\in L[s,u,v]$ is 
a polynomial of degree $4$, $8$, $8$ with respect to $s$, $u$, $v$ respectively. 

By using (\ref{actq}), (\ref{qtos}) and (\ref{eqq3}), we can evaluate the actions of 
$-\be_1$ and $\rho_a$ on $L(s,u,v)$. 
Indeed the action of $-\be_1$ on $L(s,u,v)$ is given by 
\begin{align*}
-\be_1 :\ &s\ \mapsto\ \frac{3U^2-6U(u+v)+s(s+6u-6v-4)}{D},\\
&u\ \mapsto\ \frac{3U^2-2U(4u-2v-3)+s^2-2s-4sv+4v}{D},\\
&v\ \mapsto\ \frac{2(U(u-2v-3)+su-s+2u)}{D}
\end{align*}
where
$D=3U^2-6U(2u-v-3)+s^2-2s-6sv-12u+12v+4$
and $U=u^2-uv+v^2$. 
However the action of $\rho_a$ on $L(s,u,v)$ becomes more complicated. 

We normalize the action of $-\be_1$ as follows: 
First we see that (i) the denominators of $-\be_1(u)$ and $-\be_1(v)$ coincide, and 
(ii) the numerator of $-\be_1(v)$ is linear in $s$. 
We also see 
\[
-\be_1(u)=\frac{2(U(2u-v-6)+sv+6u-4v-2)}{D}+1. 
\]
Put 
\[
r_1:=\frac{u-1}{v}.
\]
Then both of the numerator and the denominator of 
\[
r_2:=-\be_1(r_1)=\frac{U(2u-v-6)+sv+6u-4v-2}{U(u-2v-3)+su-s+2u}
\]
becomes linear in $s$. 
Therefore we see that 
\[
L(s,u,v)=L(r_1,r_2,v).
\]
The action of $-\be_1$ on $L(r_1,r_2,v)$ is given by 
\[
-\be_1 : r_1\ \mapsto\ r_2,\ r_2\ \mapsto\ r_1,\ 
v\ \mapsto\ \frac{r_1r_2-1}{(r_1^2-r_1+1)(r_2^2-r_2+1)v}. 
\]
We also put 
\[
r_3:=v(r_1^2-r_1+1), 
\]
then we get 
\[
L(x,y,z)^{G_{7,1,1}}=L(s,u,v)=L(r_1,r_2,r_3). 
\]
The action of $-\be_1$ on $L(r_1,r_2,r_3)$ is given by 
\[
-\be_1 : r_1\ \mapsto\ r_2,\ r_2\ \mapsto\ r_1,\ r_3\ \mapsto\ \frac{r_1r_2-1}{r_3}.
\]
We also see that the action of $\rho_a$ is given by 
\begin{align*}
\rho_a &: r_1\ \mapsto\ \frac{R_1}{R_2},\ r_2\ \mapsto\ \frac{R_3}{R_4},\ 
r_3\ \mapsto\ \frac{(r_1r_2-1)R_5}{r_3R_2}
\end{align*}
where 
\begin{align*}
R_1&=r_1+r_2-2r_1r_2-r_1^2r_2-r_1r_2^2+2r_1^2r_2^2-2r_2r_3+2r_1r_2^2r_3+r_2r_3^2,\\
R_2&=1+r_1-3r_1r_2-r_1^2r_2+2r_1^2r_2^2+r_3-2r_2r_3-r_1r_2r_3+2r_1r_2^2r_3-r_3^2+r_2r_3^2,\\
R_3&=-r_1+r_1^2r_2-2r_1r_3+2r_1^2r_2r_3-r_1r_3^2-r_2r_3^2+2r_1r_2r_3^2,\\
R_4&=1-r_1-r_1r_2+r_1^2r_2+r_3-2r_1r_3-r_1r_2r_3+2r_1^2r_2r_3-r_3^2-r_2r_3^2+2r_1r_2r_3^2,\\
R_5&=-1+r_1r_2-r_3+r_1r_3+r_2r_3+r_3^2.
\end{align*}
\begin{lemma}
We have $L(x,y,z)^{G_{7,1,1}}=L(r_1,r_2,r_3)=L(t_1,t_2,t_3)$ where
\[
t_1=\frac{R_1}{r_1R_5},\ t_2=\frac{R_3}{r_2R_5},\ 
t_3=\frac{r_2-r_1r_2^2-r_1r_3-r_2r_3+r_1r_2r_3-R_3}{r_2R_5}.
\]
\end{lemma}
\begin{proof}
The assertion can be checked directly as 
\begin{align*}
(r_1,r_2,r_3)=\Bigl(\frac{t_3T_1}{T_2},\frac{T_1}{t_3T_3},\frac{t_3T_4}{T_2}\Bigr)
\end{align*}
where
\begin{align*}
T_1&=-1-t_2-t_1t_2-t_1t_2^2-3t_3-3t_1t_2t_3-t_3^2-t_1t_3^2,\\
T_2&=1+2t_1t_2+t_1^2t_2^2+3t_3+t_1t_3-t_2t_3+3t_1t_2t_3+t_1^2t_2t_3
-t_1t_2^2t_3+t_3^2+2t_1t_3^2-t_1t_2t_3^2,\\
T_3&=1+2t_2-t_1t_2+3t_3-t_1t_3+t_2t_3+t_3^2,\\
T_4&=t_1+t_2+t_1^2t_2+t_1t_2^2+2t_1t_3+2t_1t_2t_3+t_1t_3^2.
\end{align*}
\end{proof}
The actions of $-\be_1$ and $\rho_a$ on $L(x,y,z)^{G_{7,1,1}}=L(t_1,t_2,t_3)$ are given by 
\begin{align*}
-\be_1 &: t_1\ \mapsto\ t_2,\ t_2\ \mapsto\ t_1,\ t_3\ \mapsto\ \frac{t_1t_2+1}{t_3},\\
\rho_a &: \sqrt{a}\ \mapsto\ -\sqrt{a},\ 
t_1\ \mapsto\ \frac{1}{t_2},\ t_2\ \mapsto\ \frac{1}{t_1},\ 
t_3\ \mapsto\ \frac{t_1t_2+1}{t_1t_3}. 
\end{align*}
We take 
\[
u_1:=\frac{t_1t_2-1}{t_1t_2+1},\ u_2:=\frac{t_1t_2+1}{t_1t_3},\ u_3:=t_3. 
\]
Then $L(x,y,z)^{G_{7,1,1}}=L(t_1,t_2,t_3)=L(u_1,u_2,u_3)$ and the actions of $-\be_1$ 
and $\rho_a$ on $L(x,y,z)^{G_{7,1,1}}$ $=$ $L(u_1,u_2,u_3)$ are given by 
\begin{align*}
-\be_1 &: u_1\ \mapsto\ u_1,\ u_2\ \mapsto\ \frac{2}{(u_1+1)u_2},\ 
u_3\ \mapsto\ -\frac{2}{(u_1-1)u_3},\\ 
\rho_a &: \sqrt{a}\ \mapsto\ -\sqrt{a},\ 
u_1\ \mapsto\ -u_1,\ u_2\ \mapsto\ u_3,\ u_3\ \mapsto\ u_2. 
\end{align*}
Hence we get 
\[
K(\sqrt{-1})(x,y,z)^{G_{7,1,1}}=L(u_1,u_2,u_3)^{\langle\rho_a\rangle}
=K(\sqrt{-1})(u_1\sqrt{a},u_2+u_3,(u_2-u_3)\sqrt{a})
\]
is rational over $K(\sqrt{-1})$. By Lemma \ref{lemtri}, we conclude that 
$K(x,y,z)^{G_{7,1,1}}$ is rational over $K$ if $[K(\sqrt{a},\sqrt{-1}):K]\leq 2$. 

By Theorem \ref{thab}, we have
\[
L(x,y,z)^{G_{7,3,1}}=L(u_1,u_2,u_3)^{\langle-\be_1\rangle}=L(m_1,m_2,m_3)
\]
where 
\begin{align*}
m_1=u_1,\ 
m_3&=\frac{u_2-(\frac{2}{u_1+1})/u_2}{u_2u_3-(\frac{2}{u_1+1})(\frac{-2}{u_1-1})/(u_2u_3)}
=\frac{(u_1-1)(-2+u_2^2+u_1u_2^2)u_3}{4-u_2^2u_3^2+u_1^2u_2^2u_3^2},\\
m_2&=\frac{u_3-(\frac{-2}{u_1-1})/u_3}{u_2u_3-(\frac{2}{u_1+1})(\frac{-2}{u_1-1})/(u_2u_3)}
=\frac{(u_1+1)(2-u_3^2+u_1u_3^2)u_2}{4-u_2^2u_3^2+u_1^2u_2^2u_3^2}.
\end{align*}
The action of $\rho_a$ on $L(m_1,m_2,m_3)$ is given by 
\[
\rho_a : \sqrt{a}\ \mapsto\ -\sqrt{a},\ 
m_1\ \mapsto\ -m_1,\ m_2\ \mapsto\ m_3,\ m_3\ \mapsto\ m_2. 
\]
Hence we get 
\[
K(\sqrt{-1})(x,y,z)^{G_{7,3,1}}=L(m_1,m_2,m_3)^{\langle\rho_a\rangle}
=K(\sqrt{-1})(m_1\sqrt{a},m_2+m_3,(m_2-m_3)\sqrt{a})
\]
is rational over $K(\sqrt{-1})$. 

\subsection{The case of $G_{7,2,1}$ with $\ep_1=-1$}

We consider the case of 
\[
G_{7,2,1}=\langle \ta_1,\la_1,\cb,-I_3\rangle
\]
with $\ep_1=-1$. 
The actions of $\ta_1$, $\la_1$, $\cb$ and $-I_3$ on $K(x,y,z)$ are given 
as in (\ref{act711}) and (\ref{act712}). 
Define 
\begin{align*}
-I_3^{(+1)} &: x\ \mapsto\ \frac{a}{x},\ 
y\ \mapsto\ \frac{a}{y},\ z\ \mapsto\ \frac{a}{z},\\
-I_3^{(-1)} &: x\ \mapsto\ -\frac{a}{x},\ 
y\ \mapsto\ -\frac{a}{y},\ z\ \mapsto\ -\frac{a}{z}
\end{align*}
and take $\rho\in\mathrm{Aut}_K K(x,y,z)$ as 
\[
\rho:=(-I_3^{(+1)})(-I_3^{(-1)})\ :\ x\ \mapsto\ -x,\ y\ \mapsto\ -y,\ z\ \mapsto\ -z. 
\]
Then we get the following lemma and we will use (ii) in the next section. 
\begin{lemma}\label{lemp1}
Let $\rho=(-I_3^{(+1)})(-I_3^{(-1)})\in \mathrm{Aut}_K K(x,y,z)$ as above. 
Then we have\\
{\rm (i)} $K(x,y,z)^{\langle \ta_1,\la_1,-I_3^{(\pm 1)}\rangle}=K(p_1^\pm,p_2^\pm,p_3^\pm)$,\\
{\rm (ii)} $K(x,y,z)^{\langle\ta_1,\la_1,\rho\rangle}=K(p_1^+/p_1^-,p_2^+/p_2^-,p_3^+/p_3^-)$, 
where
\begin{align*}
p_1^+:=\frac{(x^2-a)(y^2+a)}{x(y^2-a)},\quad p_2^+:=\frac{(y^2-a)(z^2+a)}{y(z^2-a)},\quad 
p_3^+:=\frac{(z^2-a)(x^2+a)}{z(x^2-a)},\\
p_1^-:=\frac{(x^2-a)(y^2+a)}{y(x^2+a)},\quad p_2^-:=\frac{(y^2-a)(z^2+a)}{z(y^2+a)},\quad 
p_3^-:=\frac{(z^2-a)(x^2+a)}{x(z^2+a)}. 
\end{align*}
\end{lemma}
\begin{proof}
Here we show only the case $K(x,y,z)^{\langle \ta_1,\la_1,-I_3^{(+1)}\rangle}$ and 
omit the cases $K(x,y,z)^{\langle \ta_1,\la_1,-I_3^{(-1)}\rangle}$ 
and $K(x,y,z)^{\langle\ta_1,\la_1,\rho\rangle}$ because one can prove them by a similar way. 
From the definition, we have $K(x,y,z)^{\langle \ta_1,\la_1,-I_3^{(+1)}\rangle}\supset 
K(p_1^+,p_2^+,p_3^+)$. 
Write $(p_1,p_2,p_3):=(p_1^+,p_2^+,p_3^+)$. 
Put $t_1=(x-\frac{a}{x})^2$, $t_2=(y-\frac{a}{y})^2$, $t_3=(z-\frac{a}{z})^2$, 
then we have $[K(x,y,z) : K(t_1,t_2,t_3)]=4^3=64$. 
On the other hand, we have 
\[
p_1^2=\frac{t_1(t_2+4a)}{t_2},\quad 
p_2^2=\frac{t_2(t_3+4a)}{t_3},\quad 
p_3^2=\frac{t_3(t_1+4a)}{t_1}, 
\]
from which we see $K(p_1^2,p_2^2,p_3^2)=K(t_1,t_2,t_3)$. 
Since $[K(p_1,p_2,p_3) : K(p_1^2,p_2^2,p_3^2)]=8$, we get 
$[K(x,y,z) : K(p_1,p_2,p_3)]=64/8=8$, and the assertion follows. 
\end{proof}
The action of $\cb$ on $K(p_1^\pm,p_2^\pm,p_3^\pm)$ is given by 
$\cb\,:\, p_1^\pm \mapsto p_2^\pm \mapsto p_3^\pm \mapsto p_1^\pm$. 
Hence we conclude that $K(x,y,z)^{G_{7,2,1}}$ $=$ 
$K(p_1^\pm,p_2^\pm,p_3^\pm)^{\langle \cb\rangle}$ is rational over $K$. 
\begin{remark}
We may use Lemma \ref{lemex2} to obtain another proof of Lemma \ref{lemp1} (i). 
It follows from Lemma \ref{lemex2} that 
\begin{align*}
K(x,y,z)^{\langle\ta_1,\la_1,-I_3^{(\pm 1)}\rangle}
=L(w_1,w_2,w_3)^{\langle\rho_a,\rho_{-1},-I_3^{(\pm 1)}\rangle}=K(p_1^\pm, p_2^\pm, p_3^\pm).
\end{align*}
Note that the actions of $\rho_a$ and $\rho_{-1}$ on $L(w_1,w_2,w_3)$ are 
given as in (\ref{actww}) and the action of $-I_3^{(\pm 1)}$ is given by 
\begin{align*}
-I_3^{(+1)} &: 
w_1\, \mapsto\, -\frac{w_3(w_2-1)}{w_2(w_1-w_3)},\ 
w_2\, \mapsto\, -\frac{w_2-w_3}{w_2(w_1-1)},\ 
w_3\, \mapsto\, \frac{w_1(w_2-1)(w_2-w_3)}{w_2(w_1-1)(w_1-w_3)},\\
-I_3^{(-1)} &: 
w_1\, \mapsto\, \frac{w_2(w_3-1)(w_1-w_3)}{w_3(w_2-1)(w_1-w_2)},\ 
w_2\, \mapsto\, \frac{w_1(w_3-1)}{w_3(w_1-w_2)},\ 
w_3\, \mapsto\, \frac{w_1-w_3}{(w_2-1)w_3}. 
\end{align*}
We also may get the transcendental basis $p_1^\pm,p_2^\pm,p_3^\pm$ of 
$K(x,y,z)^{\langle \ta_1,\la_1,-I_3^{(\pm 1)}\rangle}$ via $w_1,w_2,w_3$: 
\begin{align*}
p_1^+&=\frac{2\sqrt{a}\sqrt{-1}(w_3-w_1w_2)(w_2-w_1w_3)(w_1-w_2w_3)}
{(w_1w_2-w_1w_3-w_2w_3+w_1w_2w_3)(w_1w_2+w_1w_3-w_2w_3-w_1w_2w_3)},\\
p_2^+&=\frac{2\sqrt{a}\sqrt{-1}(w_1w_2-w_1w_3+w_2w_3-w_1w_2w_3)(w_1w_2+w_1w_3-w_2w_3-w_1w_2w_3)}
{w_1w_2w_3(1+w_1-w_2-w_3)(1-w_1+w_2-w_3)},\\
p_3^+&=\frac{2\sqrt{a}\sqrt{-1}w_1w_2w_3(1-w_1-w_2+w_3)(1-w_1+w_2-w_3)}
{(w_3-w_1w_2)(w_2-w_1w_3)(w_1-w_2w_3)},\\
p_1^-&=\frac{2\sqrt{a}(w_3-w_1w_2)(w_2-w_1w_3)(w_1-w_2w_3)}
{w_1w_2w_3(1+w_1-w_2-w_3)(1-w_1-w_2+w_3)},\\
p_2^-&=-\frac{2\sqrt{a}(w_1w_2-w_1w_3+w_2w_3-w_1w_2w_3)(w_1w_2+w_1w_3-w_2w_3-w_1w_2w_3)}
{(w_3-w_1w_2)(w_2-w_1w_3)(w_1-w_2w_3)},\\
p_3^-&=-\frac{2\sqrt{a}w_1w_2w_3(1+w_1-w_2-w_3)(1-w_1-w_2+w_3)}
{(w_1w_2-w_1w_3+w_2w_3-w_1w_2w_3)(w_1w_2-w_1w_3-w_2w_3+w_1w_2w_3)}.
\end{align*}
\end{remark}


%
%
\section{The case of $G_{7,j,2}$}\label{se72}

In this section, we consider the following five groups $G=G_{7,j,2}$, $1\leq j\leq 5$, 
which have a normal subgroup $\langle\ta_2,\la_2\rangle$: 
\begin{align*}
\hspace*{1.5cm}
G_{7,1,2}&=\langle \ta_2,\la_2,\cb\rangle&\hspace{-1cm} 
&\cong \mathcal{A}_4&\hspace{-1cm} 
&\cong (\mathcal{C}_2\times \mathcal{C}_2)\rtimes \mathcal{C}_3, \\
G_{7,2,2}&=\langle \ta_2,\la_2,\cb,-I_3\rangle&\hspace{-1cm} 
&\cong \mathcal{A}_4\times \mathcal{C}_2&\hspace{-1cm} 
&\cong (\mathcal{C}_2\times \mathcal{C}_2\times \mathcal{C}_2)\rtimes \mathcal{C}_3,\\
G_{7,3,2}&=\langle \ta_2,\la_2,\cb,-\be_2\rangle&\hspace{-1cm} 
&\cong \mathcal{S}_4&\hspace{-1cm} 
&\cong (\mathcal{C}_2\times \mathcal{C}_2)\rtimes \mathcal{S}_3, \\
G_{7,4,2}&=\langle \ta_2,\la_2,\cb,\be_2\rangle&\hspace{-1cm} 
&\cong \mathcal{S}_4&\hspace{-1cm} 
&\cong (\mathcal{C}_2\times \mathcal{C}_2)\rtimes \mathcal{S}_3, \\
G_{7,5,2}&=\langle \ta_2,\la_2,\cb,\be_2,-I_3\rangle&\hspace{-1cm} 
&\cong \mathcal{S}_4\times \mathcal{C}_2&\hspace{-1cm} 
&\cong (\mathcal{C}_2\times \mathcal{C}_2\times \mathcal{C}_2)\rtimes \mathcal{S}_3.
\end{align*}

The actions of $\ta_2$, $\la_2$, $\cb$, $-\be_2$, $\be_2$ and $-I_3$ on $K(x,y,z)$ are given by
\begin{align*}
\ta_2 &: x\ \mapsto\ \frac{ay}{z},\ y\ \mapsto\ \frac{bx}{z},\ z\ \mapsto\  \frac{c}{z},& 
\la_2 &: x\ \mapsto\ \frac{dz}{y},\ y\ \mapsto\ \frac{e}{y},\ z\ \mapsto\ \frac{fx}{y},\\
\cb &: x\ \mapsto\ gy,\ y\ \mapsto\  hz,\ z\ \mapsto\ ix,& 
-\be_2 &: x\ \mapsto\ \frac{jz}{x},\ y\ \mapsto\  \frac{kz}{y},\ z\ \mapsto\ lz,\\
\be_2 &: x\ \mapsto\ \frac{mx}{z},\ y\ \mapsto\ \frac{ny}{z},\ z\ \mapsto\  \frac{o}{z},& 
-I_3 &: x\ \mapsto\ \frac{p}{x},\ y\ \mapsto\ \frac{q}{y},\ z\ \mapsto\ \frac{r}{z}.
\end{align*}

We may assume that $g=h=i=1$ by replacing $(gy,ghz)$ by $(y,z)$ and the other coefficients.
By the equalities $\ta_2^2=\la_2^2=\cb^3=(-\be_2)^2=\be_2^2=(-I_3)^2=I_3$, 
we have $ab=c$, $df=e$, $l=1$ and $m^2=n^2=o$. 

By the relations of the generators of $G_{7,1,2}$ as in (\ref{relmat7}), 
we have $a=f$, $b=d=\ep a$, $c=e$ with $\ep=\pm 1$. 

By the relations of $-\be_2$, $\be_2$, $-I_3$ with the generators of $G_{7,1,2}$, we have 
$j=m=\ep k=a$, $nm=c$, $p=q=r=a^2$ so that the problem reduces to the following cases: 
\begin{align}
\ta_2 &: x\ \mapsto\ \frac{ay}{z},\ y\ \mapsto\ \frac{\ep ax}{z},\ 
z\ \mapsto\ \frac{\ep a^2}{z},& 
\la_2 &: x\ \mapsto\ \frac{\ep az}{y},\ y\ \mapsto\ \frac{\ep a^2}{y},\ 
z\ \mapsto\ \frac{ax}{y},\nonumber\\
\cb &: x\ \mapsto\ y,\ y\ \mapsto\  z,\ z\ \mapsto\ x,& 
\hspace*{-5mm}
-\be_2 &: x\ \mapsto\ \frac{az}{x},\ y\ \mapsto\ \frac{\ep az}{y},\ 
z\ \mapsto\ z,\label{act72}\\
\be_2 &: x\ \mapsto\ \frac{ax}{z},\ y\ \mapsto\ \frac{\ep ay}{z},\ 
z\ \mapsto\  \frac{a^2}{z},& 
-I_3 &: x\ \mapsto\ \frac{a^2}{x},\ y\ \mapsto\ \frac{a^2}{y},\ z\ \mapsto\ \frac{a^2}{z}\nonumber
\end{align}
where $a\in K^\times$ and $\ep=\pm 1$. 

\subsection{The case of $\ep=1$}

We treat the case where $\ep=1$ in this subsection. 
If we replace $(x/a,y/a,z/a)$ by $(x,y,z)$ then 
the actions of $\ta_2$, $\la_2$, $\cb$, $-\be_2$, $\be_2$ and $-I_3$ on $K(x,y,z)$ 
are given by 
\begin{align*}
\ta_2 &: x\ \mapsto\ \frac{y}{z},\ y\ \mapsto\ \frac{x}{z},\ z\ \mapsto\ \frac{1}{z},& 
\la_2 &: x\ \mapsto\ \frac{z}{y},\ y\ \mapsto\ \frac{1}{y},\ z\ \mapsto\ \frac{x}{y},\\
\cb &: x\ \mapsto\ y,\ y\ \mapsto\  z,\ z\ \mapsto\ x,& 
-\be_2 &: x\ \mapsto\ \frac{z}{x},\ y\ \mapsto\ \frac{z}{y},\ z\ \mapsto\ z,\\
\be_2 &: x\ \mapsto\ \frac{x}{z},\ y\ \mapsto\ \frac{y}{z},\ z\ \mapsto\  \frac{1}{z},& 
-I_3 &: x\ \mapsto\ \frac{1}{x},\ y\ \mapsto\ \frac{1}{y},\ z\ \mapsto\ \frac{1}{z}.
\end{align*}

Because the actions of $G_{7,j,2}$, $1\leq j\leq 5$, on $K(x,y,z)$ are $3$-dimensional 
purely monomial, it follows from Theorem \ref{thHKHR} that 
all of the fixed fields $K(x,y,z)^{G_{7,j,2}}$, $1\leq j\leq 5$, are rational over $K$. 

Indeed, using Theorem \ref{thmex}, we get explicit transcendental bases of the fixed 
fields over $K$ as follows: 
Put 
\[
X:=\frac{x+1}{y+z},\quad Y:=\frac{y+1}{x+z},\quad Z:=\frac{z+1}{x+y}. 
\]
Then $K(x,y,z)=K(X,Y,Z)$ because for any $a\in K^\times$, we have
\begin{align}
K(x,y,z)&=K(x+y,y+z,z+x)=K(x+y+z,y+z,z+x)\nonumber\\
&=K\Bigl(\frac{x+y+z}{x+y+z+a},\frac{y+z}{x+y+z+a},\frac{z+x}{x+y+z+a}\Bigr)\label{KxX}\\
&=K\Bigl(\frac{x+y}{x+y+z+a},\frac{y+z}{x+y+z+a},\frac{z+x}{x+y+z+a}\Bigr)\nonumber\\
&=K\Bigl(\frac{x+a}{y+z},\frac{y+a}{x+z},\frac{z+a}{x+y}\Bigr).\nonumber
\end{align}
The actions of $\ta_2$, $\la_2$ on $K(X,Y,Z)$ are given by 
\begin{align*}
\ta_2 &: X\ \mapsto\ \frac{1}{X},\ Y\ \mapsto\ \frac{1}{Y},\ Z\ \mapsto\ Z,\quad 
\la_2 : X\ \mapsto\ \frac{1}{X},\ Y\ \mapsto\ Y,\ Z\ \mapsto\ \frac{1}{Z}.
\end{align*}
We note that the action of $\langle\ta_2,\la_2\rangle$ on $K(X,Y,Z)$ is 
an affirmative case of $G_{3,1,1}\in\mathcal{N}$ (cf. Theorem \ref{thmex}). 
By Theorem \ref{thmex}, we have 
$K(x,y,z)^{\langle\ta_2,\la_2\rangle}=K(v_1,v_2,v_3)$ where 
\begin{align*}
v_1&:=\frac{2(-X+Y+Z-XYZ)}{1-XY-XZ+YZ}=\frac{x^2-y^2-z^2+1}{x-yz},\\
v_2&:=\cb(v_1)=\frac{-x^2+y^2-z^2+1}{y-xz},\quad 
v_3:=\cb^2(v_1)=\frac{-x^2-y^2+z^2+1}{z-xy}.
\end{align*}
Now we take 
\begin{align*}
u_1:=\frac{v_1+v_2}{v_3-2},\quad 
u_2:=\frac{v_2+v_3}{v_1-2},\quad 
u_3:=\frac{v_1+v_3}{v_2-2}.
\end{align*}
Then we see $K(u_1,u_2,u_3)=K(v_1,v_2,v_3)$ (cf. (\ref{KxX})). 
The generators $u_1,u_2,u_3$ are given in terms of $x$, $y$ and $z$ explicitly as follows: 
\begin{lemma}\label{lem72pp}
We have $K(x,y,z)^{\langle\ta_2,\la_2\rangle}=K(u_1,u_2,u_3)$ where 
\begin{align*}
u_1=\frac{(x+y)(z+1)(z-xy)}{(x-yz)(y-xz)},\quad 
u_2=\frac{(y+z)(x+1)(x-yz)}{(y-xz)(z-xy)},\quad 
u_3=\frac{(x+z)(y+1)(y-xz)}{(x-yz)(z-xy)}. 
\end{align*}
\end{lemma}
The actions of $\cb$, $-\be_2$, $\be_2$ and $-I_3$ on $K(u_1,u_2,u_3)$ are given by 
\begin{align}
\cb &: u_1\ \mapsto\ u_2,\ u_2\ \mapsto\ u_3,\ u_3\ \mapsto\ u_1,& 
-\be_2 &: u_1\ \mapsto\ -u_1,\ u_2\ \mapsto\ -u_3,\ u_3\ \mapsto\ -u_2,\label{act72u}\\
\be_2&: u_1\ \mapsto\ u_1,\ u_2\ \mapsto\ u_3,\ u_3\ \mapsto\ u_2,& 
-I_3 &: u_1\ \mapsto\ -u_1,\ u_2\ \mapsto\ -u_2,\ u_3\ \mapsto\ -u_3. \nonumber
\end{align}
From this, we can easily get explicit transcendental basis of 
$K(x,y,z)^{G_{7,j,2}}$, $1\leq j\leq 5$, over $K$ for $\ep=1$. 
\begin{remark}
Hoshi-Rikuna \cite{HR08} solved the rationality problem of the group 
$G_{7,3,2}$ $=$ $\langle \ta_2,\la_2,\cb,-\be_2\rangle$ with $\ep=1$ 
which is the remaining case of purely monomial actions in Hajja-Kang's paper \cite{HK94}. 
Indeed the fact that $G_{7,3,2}\cong\mathcal{S}_4$ acts on 
$K(x,y,z)^{\langle\ta_2,\la_2\rangle}$ via the twisted $\mathcal{S}_3$-action 
as in Theorem \ref{thHK97} was given in \cite[page 1827]{HR08}. 
\end{remark}

\subsection{The case of $\ep=-1$}

We treat the case where $\ep=-1$ in this subsection. 
We consider a method to solve the rationality problem of $G_{7,j,2}$ over $K$ 
by applying results of $G_{7,j,1}$ as in Section \ref{se71}. 
We also give an another proof to the case $\ep=1$, 
although we treated it in the previous subsection. 

We assume that the action of $G_{7,j,1}$ on $K(y_1,y_2,y_3):=K(x,y,z)$ 
is given as in (\ref{act711}) and (\ref{act712}). 
For $1\leq j\leq 5$, the subgroups $G_{7,j,1}$ and $G_{7,j,2}$ of $\mathrm{GL}(3,\bQ)$ 
are conjugate in $\mathrm{GL}(3,\bQ)$ but not in $\mathrm{GL}(3,\bZ)$. 
Indeed we have 
\begin{align}
P^{-1}G_{7,j,1}P=G_{7,j,2},\quad (1\leq j\leq 5)\label{actPG}
\end{align}
where 
\[
P=\left[\begin{array}{ccc} 0 & 1 & 1\\ 1 & 0 & 1\\ 1 & 1 & 0\end{array}\right],\quad 
P^{-1}=\frac{1}{2}\left[\begin{array}{ccc} -1 & 1 & 1\\ 1 & -1 & 1\\ 1 & 1 & -1\end{array}\right].
\]
We put 
\[
x_j:=\prod_{i=1}^3 y_i^{p_{i,j}},\quad  (j=1,2,3)\quad \mathrm{where}\quad 
P=[p_{i,j}]_{1\leq i,j\leq 3}.
\] 

We assume that $G_{7,j,2}$ acts on $K(x_1,x_2,x_3):=K(x,y,z)$ as in (\ref{act72}). 
Then we have the following theorem:
\begin{theorem}\label{thrho}
Let $G_{7,j,1}$ $($resp. $G_{7,j,2})$ act on $K(y_1,y_2,y_3)$ $($resp. $K(x_1,x_2,x_3))$ 
by monomial actions as in $(\ref{act711})$ and $(\ref{act712})$ $($resp. $(\ref{act72}))$. 
Let $\rho$ be a $K$-automorphism on $K(y_1,y_2,y_3)$ defined by 
\[
\rho\ :\ y_1\ \mapsto\ -y_1,\ y_2\ \mapsto\ -y_2,\ y_3\ \mapsto\ -y_3. 
\]
Then we have\\
{\rm (i)} $K(y_1,y_2,y_3)^{\langle\rho\rangle}=K(x_1,x_2,x_3)$ and 
$[K(y_1,y_2,y_3):K(x_1,x_2,x_3)]=\mathrm{det}(P)=2$;\\
{\rm (ii)} under the assumption $\ep_1=1$ for $G_{7,j,1}$ and $\ep=1$ for $G_{7,j,2}$, 
the action of $G_{7,j,1}$ on $K(x_1,x_2,x_3)$ coincides with 
the action of $G_{7,j,2}$ and 
$K(x_1,x_2,x_3)^{G_{7,j,2}}=K(y_1,y_2,y_3)^{\langle G_{7,j,1},\rho\rangle}$ 
for $1\leq j\leq 5$;\\
{\rm (iii)} the action of $G_{7,1,1}$ with $\ep_1=-1$ on $K(x_1,x_2,x_3)$ coincides with 
the action of $G_{7,1,2}$ with $\ep=-1$ and $K(x_1,x_2,x_3)^{\langle\ta_2,\la_2\rangle}
=K(y_1,y_2,y_3)^{\langle \ta_1,\la_1,\rho\rangle}$.
\end{theorem}
\begin{proof}
By the definition, we have $(x_1,x_2,x_3)=(y_2y_3,y_1y_3,y_1y_2)$. 
Then (i) is evident.\\
From (\ref{actPG}), the first assertion (ii) follows if the actions of 
$G_{7,j,1}$ and $G_{7,j,2}$ are purely monomial. 
For the other cases, we can check the assertion by direct calculation. \\
Because we have $\rho\,\sigma=\sigma\rho$ for any $\sigma\in G_{7,j,1}$, 
we see $K(y_1,y_2,y_3)^{\langle G_{7,j,1},\rho\rangle}
=(K(y_1,y_2,y_3)^{\langle\rho\rangle})^{G_{7,j,1}}=K(x_1,x_2,x_3)^{G_{7,j,2}}$. 
\end{proof}
By Theorem \ref{thrho}, we should study the fixed field 
$K(y_1,y_2,y_3)^{\langle \ta_1,\la_1,\rho\rangle}$. 
Using results in the previous section, we obtain the following lemma: 
\begin{lemma}\label{lemv0}
We have 
\[
K(y_1,y_2,y_3)^{\langle\ta_1,\la_1,\rho\rangle}=K(u_1,u_2,u_3)
\]
where 
\begin{align*}
&\begin{cases}
u_1=\displaystyle{\frac{(a(-y_1+y_2+y_3)-y_1y_2y_3)(a(y_1-y_2+y_3)-y_1y_2y_3)}
{(a-y_1y_2-y_1y_3+y_2y_3)(a-y_1y_2+y_1y_3-y_2y_3)}},\ \mathrm{if}\ \ep_1=1,\vspace*{1mm}\\
u_1=\displaystyle{\frac{y_2(y_1^2+a)}{y_1(y_2^2-a)}},\ \mathrm{if}\ \ep_1=-1,
\end{cases}\\
&u_2=\cb(u_1),\quad u_3=\cb^2(u_1)
\end{align*}
and $\ep_1$ is given as in $(\ref{act711})$. 
\end{lemma}
\begin{proof}
When $\ep_1=1$, we see $K(y_1,y_2,y_3)^{\langle\ta_1,\la_1\rangle}=K(v_1,v_2,v_3)$ 
where $v_1,v_2,v_3$ are given by (\ref{defvv}) in Subsection \ref{subse71p}. 
The action of $\rho$ on $K(v_1,v_2,v_3)$ is given by $\rho:v_i\mapsto -v_i$, $(i=1,2,3)$ 
and hence $K(y_1,y_2,y_3)^{\langle\ta_1,\la_1,\rho\rangle}=K(v_1v_2,v_2v_3,v_3v_1)
=K(u_1,u_2,u_3)$. 

When $\ep_1=-1$, the assertion follows from Lemma \ref{lemp1} (ii). 
\end{proof}
\begin{lemma}\label{lemv1}
We have 
\begin{align*}
K(x_1,x_2,x_3)^{\langle \ta_2,\la_2\rangle}=K(u_1,u_2,u_3)
\end{align*}
where 
\begin{align*}
&\begin{cases}
u_1=\displaystyle{\frac{(a(x_1x_2+x_1x_3-x_2x_3)-x_1x_2x_3)(a(x_1x_2-x_1x_3+x_2x_3)-x_1x_2x_3)}
{x_1x_2x_3(a+x_1-x_2-x_3)(a-x_1+x_2-x_3)}},\ \mathrm{if}\ \ep=1,\vspace*{1mm}\\
u_1=\displaystyle{\frac{x_2x_3+ax_1}{x_1x_3-ax_2}},\ \mathrm{if}\ \ep=-1,
\end{cases}\\
&u_2=\cb(u_1),\quad u_3=\cb^2(u_1)
\end{align*}
and $\ep$ is given as in $(\ref{act72})$. 

\end{lemma}
\begin{proof}
By Theorem \ref{thrho} and Lemma \ref{lemv0}, we get 
$K(x_1,x_2,x_3)^{\langle \ta_2,\la_2\rangle}=
K(y_1,y_2,y_3)^{\langle\ta_1,\la_1,\rho\rangle}=K(u_1,u_2,u_3)$. 
By the definition $(x_1,x_2,x_3)=(y_2y_3,y_1y_3,y_1y_2)$, 
the $u_1$ in Lemma \ref{lemv0} is expressed in terms of $x_1$, $x_2$, $x_3$ as given 
in Lemma \ref{lemv1}. 
\end{proof}
%

When $\ep=1$, the actions of $\cb$, $-\be_2$, $\be_2$ and $-I_3$ on 
$K(x,y,z)^{\langle\ta_2,\la_2\rangle}=K(u_1,u_2,u_3)$ are given by 
\begin{align*}
\cb &: u_1\ \mapsto\ u_2,\ u_2\ \mapsto u_3,\ u_3\ \mapsto\ u_1,& 
-\be_2 &: u_1\ \mapsto\ \frac{a^2}{u_1},\ u_2\ \mapsto\ \frac{a^2}{u_3},\ 
u_3\ \mapsto\ \frac{a^2}{u_2},\\
\be_2 &: u_1\ \mapsto\ u_1,\ u_2\ \mapsto\ u_3,\ u_3\ \mapsto\ u_2,& 
-I_3 & : u_1\ \mapsto\ \frac{a^2}{u_1},\ u_2\ \mapsto\ \frac{a^2}{u_2},\ 
u_3\ \mapsto\ \frac{a^2}{u_3}.
\end{align*}
Therefore if we put 
\[
v_1:=\frac{u_1+a}{u_1-a},\quad v_2:=\frac{u_2+a}{u_2-a},\quad 
v_3:=\frac{u_3+a}{u_3-a}
\]
then $K(u_1,u_2,u_3)=K(v_1,v_2,v_3)$ and the actions of 
$G_{7,j,2}$, $1\leq j\leq 5$, on $K(x,y,z)^{\langle\ta_2,\la_2\rangle}=K(v_1,v_2,v_3)$ 
are given as the same as in (\ref{act72u}). 
Then we can confirm that the fixed fields $K(x,y,z)^{G_{7,j,2}}$, $1\leq j\leq 5$, 
are rational over $K$. 
Note that 
\[
v_1=\frac{(xy+az)(xz-ay)(yz-ax)}{(xy-az)(a^2xy+ax^2z-4axyz+ay^2z+xyz^2)}
\]
(cf. Lemma \ref{lem72pp}). 
Though the transcendental basis of $K(x,y,z)^{\langle\ta_2,\la_2\rangle}$ over $K$ 
obtained here is different from the one obtained in Lemma \ref{lem72pp}, the purely monomial 
actions of $G_{7,j,2}$ act in a same way for both of them. 

When $\ep=-1$, the actions of $\cb$, $-\be_2$, $\be_2$ and $-I_3$ on 
$K(x,y,z)^{\langle\ta_2,\la_2\rangle}=K(u_1,u_2,u_3)$ are given by 
\begin{align*}
\cb &: u_1\ \mapsto\ u_2,\ u_2\ \mapsto u_3,\ u_3\ \mapsto\ u_1,& 
-\be_2 &: u_1\ \mapsto\ \frac{-1}{u_2},\ u_2\ \mapsto\ \frac{-1}{u_1},\ 
u_3\ \mapsto\ \frac{-1}{u_3},\\
\be_2 &: u_1\ \mapsto\ \frac{1}{u_2},\ u_2\ \mapsto\ \frac{1}{u_1},\ 
u_3\ \mapsto\ \frac{1}{u_3},& 
-I_3 & : u_1\ \mapsto\ -u_1,\ u_2\ \mapsto -u_2,\ u_3\ \mapsto\ -u_3.
\end{align*}
By Lemma \ref{lemMas}, $K(x,y,z)^{G_{7,1,2}}=K(u_1,u_2,u_3)^{\langle \cb\rangle}$ 
is rational over $K$. 
The groups $G_{7,3,2}$ and $G_{7,4,2}$ act on $K(u_1,u_2,u_3)$ by 
the twisted $\mathcal{S}_3$-action as in Theorem \ref{thHK1}. 
Hence $K(x,y,z)^{G_{7,3,2}}=K(u_1,u_2,u_3)^{\langle \cb,-\be_2\rangle}$ 
and $K(x,y,z)^{G_{7,4,2}}=K(u_1,u_2,u_3)^{\langle \cb,\be_2\rangle}$ are rational 
over $K$ (see Theorem \ref{thHK1} or Section \ref{se5B}). 
By a result of Section \ref{se5B}, we can also get explicit transcendental bases 
of the fixed fields over $K$. 
For $G_{7,2,2}=\langle \ta_2,\la_2,\cb,-I_3\rangle$ and 
$G_{7,5,2}=\langle \ta_2,\la_2,\cb,\be_2,-I_3\rangle$, 
we see 
\begin{align*}
K(x,y,z)^{\langle \ta_2,\la_2,-I_3\rangle}&=K(u_1,u_2,u_3)^{\langle -I_3\rangle}
=K(t_1,t_2,t_3)
\end{align*}
where
\begin{align*}
t_1:=u_2u_3,\quad t_2:=u_3u_1,\quad t_3:=u_1u_2. 
\end{align*}

The actions of $\cb$ and $\be_2$ on $K(t_1,t_2,t_3)$ are given by 
\begin{align*}
\cb &: t_1\ \mapsto\ t_2,\ t_2\ \mapsto\ t_3,\ t_3\ \mapsto\ t_1,& 
\be_2 &: t_1\ \mapsto\ \frac{1}{t_2},\ t_2\ \mapsto\ \frac{1}{t_1},\ 
t_3\ \mapsto\ \frac{1}{t_3}.
\end{align*}

We now see $K(x,y,z)^{G_{7,2,2}}=K(t_1,t_2,t_3)^{\langle \cb\rangle}$ is rational over $K$ 
by Lemma \ref{lemMas}. 
The group $G_{7,5,2}$ acts on $K(t_1,t_2,t_3)$ by the twisted $\mathcal{S}_3$-action. 
Thus it follows from Theorem \ref{thHK97} that 
$K(x,y,z)^{G_{7,5,2}}=K(t_1,t_2,t_3)^{\langle \cb,\be_2\rangle}$ is rational over $K$. 

\begin{remark}
Conversely, the monomial action of $G_{7,1,2}$ on $K(y_1,y_2,y_3)$ leads to the monomial 
action of $G_{7,1,1}$ on $K(x_1,x_2,x_3)$ by putting 
$(x_1,x_2,x_3):=(y_2y_3/y_1,y_1y_3/y_2,y_1y_2/y_3)$. 
However both of $\ep=\pm 1$ for $G_{7,1,2}$ lead to the purely monomial action of 
$\langle\ta_1,\la_1\rangle$, and non-purely monomial action of 
$\langle\ta_1,\la_1\rangle$ does not appear by this method. 
In many cases, the method used in this subsection is useful to investigate purely 
monomial actions, but not so useful for non-purely monomial actions. 
\end{remark}
\begin{example}
We give an application of the result in this section to linear Noether's problem 
over $\bQ$ (see Corollary \ref{cor1}). 
The groups $H:=\mathrm{SL}(2,\mathbb{F}_3)$ and 
$G:=\mathrm{GL}(2,\mathbb{F}_3)$ have a monomial representation $\rho$ 
in $\mathrm{GL}(4,\mathbb{Q})$. 
Indeed we have 
\[
\rho(H)=\langle A,B,C\rangle,\quad 
\rho(G)=\langle A,B,C,D\rangle
\]
where
\begin{center}
$A=${\tiny $\left[\begin{array}{cccc} 
0 & 1 & 0 & 0\\ -1 & 0 & 0 & 0\\
0 & 0 & 0 & 1\\ 0 & 0 & -1 & 0\end{array}\right]$},\quad
$B=${\tiny $\left[\begin{array}{cccc} 
0 & 0 & 1 & 0\\ 0 & 0 & 0 & -1\\
-1 & 0 & 0 & 0\\ 0 & 1 & 0 & 0\end{array}\right]$},\quad
$C=${\tiny $\left[\begin{array}{cccc} 
0 & 0 & 1 & 0\\ -1 & 0 & 0 & 0\\
0 & -1 & 0 & 0\\ 0 & 0 & 0 & 1\end{array}\right]$},\quad
$D=${\tiny $\left[\begin{array}{cccc} 
1 & 0 & 0 & 0\\ 0 & -1 & 0 & 0\\
0 & 0 & 0 & 1\\ 0 & 0 & 1 & 0\end{array}\right]$}.
\end{center}
The linear actions of $H$ and of $G$ on 
$\mathbb{Q}(V)=\mathbb{Q}(w_1,w_2,w_3,w_4)$ are given by 
\begin{align*}
A &: w_1\mapsto -w_2\mapsto -w_1\mapsto w_2\mapsto w_1,\, 
w_3\mapsto -w_4\mapsto -w_3\mapsto w_4\mapsto w_3,\\
B &: w_1\mapsto -w_3\mapsto -w_1\mapsto w_3\mapsto w_1,\, 
w_2\mapsto w_4\mapsto -w_2\mapsto -w_4\mapsto w_2,\\
C &: w_1 \mapsto -w_2\mapsto w_3\mapsto w_1,\, w_4 \mapsto w_4,\quad 
D : w_1 \mapsto w_1,\, w_2\mapsto -w_2,\, w_3\leftrightarrow w_4.
\end{align*}
Then the induced linear actions of $H$ and of $G$ 
on $\mathbb{Q}(V)_0=\mathbb{Q}(x,y,z)$ with $x=w_1/w_4$, $y=-w_2/w_4$, $z=w_3/w_4$ 
coincide with the monomial actions of $G_{7,1,2}$ and of $G_{7,4,2}$ with $\ep=-1$ and $a=1$ 
as in (\ref{act72}) respectively. 
Hence we get another proof of the results of Rikuna \cite{Rik} and Plans \cite{Pla07} 
which claim that $\mathbb{Q}(V)_0^H$, $\mathbb{Q}(V)^H$, 
$\mathbb{Q}(V)_0^G$ and $\mathbb{Q}(V)^G$ 
are rational over $\mathbb{Q}$. 
Indeed Plans showed that the action of $G$ on 
$\mathbb{Q}(V)_0^{\langle A,B\rangle}$ is the twisted action of $\mathcal{S}_3$ 
as in Theorem \ref{thHK97}. 
\end{example}
%


\section{The case of $G_{7,j,3}$}\label{se73}

In this section, we treat the following five groups $G=G_{7,j,3}$, $1\leq j\leq 5$, 
which have a normal subgroup $\langle\ta_3,\la_3\rangle$: 
\begin{align*}
\hspace*{1.5cm}
G_{7,1,3}&=\langle \ta_3,\la_3,\cb\rangle&\hspace{-1cm} 
&\cong \mathcal{A}_4&\hspace{-1cm} 
&\cong (\mathcal{C}_2\times \mathcal{C}_2)\rtimes \mathcal{C}_3, \\
G_{7,2,3}&=\langle \ta_3,\la_3,\cb,-I_3\rangle&\hspace{-1cm} 
&\cong \mathcal{A}_4\times \mathcal{C}_2&\hspace{-1cm} 
&\cong (\mathcal{C}_2\times \mathcal{C}_2\times \mathcal{C}_2)\rtimes \mathcal{C}_3,\\
G_{7,3,3}&=\langle \ta_3,\la_3,\cb,-\be_3\rangle&\hspace{-1cm} 
&\cong \mathcal{S}_4&\hspace{-1cm} 
&\cong (\mathcal{C}_2\times \mathcal{C}_2)\rtimes \mathcal{S}_3, \\
G_{7,4,3}&=\langle \ta_3,\la_3,\cb,\be_3\rangle&\hspace{-1cm} 
&\cong \mathcal{S}_4&\hspace{-1cm} 
&\cong (\mathcal{C}_2\times \mathcal{C}_2)\rtimes \mathcal{S}_3, \\
G_{7,5,3}&=\langle \ta_3,\la_3,\cb,\be_3,-I_3\rangle&\hspace{-1cm} 
&\cong \mathcal{S}_4\times \mathcal{C}_2&\hspace{-1cm} 
&\cong (\mathcal{C}_2\times \mathcal{C}_2\times \mathcal{C}_2)\rtimes \mathcal{S}_3.
\end{align*}

The actions of $\ta_3$, $\la_3$, $\cb$, $-\be_3$, $\be_3$ and $-I_3$ on $K(x,y,z)$ 
are given by
\begin{align*}
\ta_3 &: x\ \mapsto\ ay,\ y\ \mapsto\ bx,\ z\ \mapsto\ \frac{c}{xyz},& 
\la_3 &: x\ \mapsto\ dz,\ y\ \mapsto\ \frac{e}{xyz},\ z\ \mapsto\ fx,\\
\cb &: x\ \mapsto\ gy,\ y\ \mapsto\  hz,\ z\ \mapsto\ ix,& 
-\be_3 &: x\ \mapsto\ \frac{j}{x},\ y\ \mapsto\  \frac{k}{y},\ z\ \mapsto\ lxyz,\\
\be_3 &: x\ \mapsto\ mx,\ y\ \mapsto\ ny,\ z\ \mapsto\ \frac{o}{xyz},& 
-I_3 &: x\ \mapsto\ \frac{p}{x},\ y\ \mapsto\ \frac{q}{y},\ z\ \mapsto\ \frac{r}{z}.
\end{align*}

We may assume that $g=h=i=1$ by replacing $(gy,ghz)$ by $(y,z)$ and the other coefficients.
By the equalities $\ta_3^2=\la_3^2=\cb^3=(-\be_3)^2=\be_3^2=(-I_3)^2=I_3$, 
we see $ab=df=jkl^2=m^2=n^2=1$. 

By the relations of the generators of $G_{7,1,3}$ as in (\ref{relmat7}), 
we have $a=f=1$, $c=e$. 

By the relations of $-\be_3$, $\be_3$ and $-I_3$ with the generators of $G_{7,1,3}$, 
we have $j=k=lc$, $mn=1$, $mo=c$, $p=q=r$, $c^2=p^4$. 
So that the problem reduces to the following cases: 
\begin{align*}
\ta_3 &: x\ \mapsto\ y,\ y\ \mapsto\ x,\ z\ \mapsto\ \frac{c}{xyz},&
\la_3 &: x\ \mapsto\ z,\ y\ \mapsto\ \frac{c}{xyz},\ z\ \mapsto\ x,\\
\cb &: x\ \mapsto\  y,\ y\ \mapsto\  z,\ z\ \mapsto\ x,& 
-\be_3 &: x\ \mapsto\ \frac{j}{x},\ y\ \mapsto\  \frac{j}{y},\ 
z\ \mapsto\ \frac{\ep_1 xyz}{j},\\
\be_3 &: x\ \mapsto\ \ep_2 x,\ y\ \mapsto\ \ep_2 y,\ z\ \mapsto\ \frac{\ep_2 c}{xyz},& 
-I_3 &: x\ \mapsto\ \frac{r}{x},\ y\ \mapsto\ \frac{r}{y},\ z\ \mapsto\ \frac{r}{z}
\end{align*}
where $c,j,r\in K^\times$ and $\ep_1,\ep_2=\pm 1$. 
We have $c=\ep_1 j^2$ for $G=G_{7,3,3}$ and $G_{7,5,3}$, 
$c=\ep_3r^2$ for $G=G_{7,2,3}$ and $G_{7,5,3}$, 
and $j=\ep_2r$, $\ep_1=\ep_3$ for $G=G_{7,5,3}$. 
Put 
\begin{align*}
w:=\frac{c}{xyz}. 
\end{align*}
Then $K(x,y,z)=K(x,y,z,w)$ with $xyzw=c$, and the actions of $\ta_3$, $\la_3$, $\cb$, $-\be_3$, 
$\be_3$ and $-I_3$ on $K(x,y,z,w)$ are: 
\begin{align*}
\ta_3 &: x\ \mapsto\ y,\ y\ \mapsto\ x,\ z\ \mapsto\ w,\ w\ \mapsto\ z,\\
\la_3 &: x\ \mapsto\ z,\ y\ \mapsto\ w,\ z\ \mapsto\ x,\ w\ \mapsto\ y,\\
\cb &: x\ \mapsto\  y,\ y\ \mapsto\  z,\ z\ \mapsto\ x,\ w\ \mapsto\ w,\\
-\be_3 &: x\ \mapsto\ \frac{j}{x},\ y\ \mapsto\  \frac{j}{y},\ 
z\ \mapsto\ \frac{j}{w},\ w\ \mapsto\ \frac{j}{z},\\
\be_3 &: x\ \mapsto\ \ep_2 x,\ y\ \mapsto\ \ep_2 y,\ z\ \mapsto\ \ep_2 w,\ w\ \mapsto\ \ep_2 z,\\
-I_3 &: x\ \mapsto\ \frac{r}{x},\ y\ \mapsto\ \frac{r}{y},\ z\ 
\mapsto\ \frac{r}{z},\ w\ \mapsto\ \frac{r}{w}.
\end{align*}

By Lemma \ref{lemV42}, we have $K(x,y,z)^{\langle \ta_3,\la_3\rangle}
=K(x,y,z,w)^{\langle \ta_3,\la_3\rangle}=K(u_1,u_2,u_3)$ where
\begin{align*}
u_1\ :=\ \frac{x+y-z-w}{xy-zw},\quad
u_2\ :=\ \frac{x-y-z+w}{xw-yz},\quad
u_3\ :=\ \frac{x-y+z-w}{xz-yw}
\end{align*}
with $w=c/(xyz)$. 
The actions of $\cb$, $-\be_3$, $\be_3$ and $-I_3$ on 
$K(x,y,z)^{\langle \ta_3,\la_3\rangle}=K(u_1,u_2,u_3)$ are given by 
\begin{align*}
\cb &: u_1\ \mapsto\  u_2,\ u_2\ \mapsto\  u_3,\ u_3\ \mapsto\ u_1,\\
-\be_3 &: u_1\ \mapsto\ \frac{-u_1+u_2+u_3}{j\, u_2u_3},\ 
u_2\ \mapsto\ \frac{u_1+u_2-u_3}{j\, u_1u_2},\ 
u_3\ \mapsto\ \frac{u_1-u_2+u_3}{j\, u_1u_3},\\
\be_3 &: u_1\ \mapsto\ \ep_2 u_1,\ u_2\ \mapsto\ \ep_2 u_3,\ u_3\ \mapsto\ \ep_2 u_2,\\
-I_3 &: u_1\ \mapsto\ \frac{-u_1+u_2+u_3}{r\, u_2u_3},\ u_2\ \mapsto\ 
\frac{u_1-u_2+u_3}{r\, u_1u_3},\ u_3\ \mapsto\ \frac{u_1+u_2-u_3}{r\, u_1u_2}.
\end{align*}

\subsection{The cases of $G_{7,1,3}$ and $G_{7,4,3}$}\label{subse713}

We consider the cases of 
\[
G_{7,1,3}=\langle \ta_3,\la_3,\cb\rangle,\quad 
G_{7,4,3}=\langle \ta_3,\la_3,\cb,\be_3\rangle.
\]

On the field $K(u_1,u_2,u_3)$, $\cb$ acts as a cyclic permutation, and 
$\langle\cb,\be_3\rangle$ acts as a natural or twisted $S_3$-action according to $\ep_2=1$ 
or $-1$. 
So $K(x,y,z)^{G_{7,1,3}}=K(u_1,u_2,u_3)^{\langle \cb\rangle}$ and 
$K(x,y,z)^{G_{7,4,3}}=K(u_1,u_2,u_3)^{\langle \cb,\be_3\rangle}$ are rational over $K$ 
(For twisted $S_3$-action, see Theorem \ref{thHK97}.).

\subsection{The cases of $G_{7,2,3}$ and $G_{7,5,3}$}
We treat the cases of 
\[
G_{7,2,3}=\langle \ta_3,\la_3,\cb,-I_3\rangle,\quad 
G_{7,5,3}=\langle \ta_3,\la_3,\cb,\be_3,-I_3\rangle.
\]
In these cases, we first consider the fixed field $K(x,y,z)^{\langle\ta_3,\la_3,-I_3\rangle}$ 
since $\langle\ta_3,\la_3,-I_3\rangle$ is a normal subgroup of $G_{7,2,3}$ and of $G_{7,5,3}$ 
respectively. 
By Lemma \ref{lemtlm}, we have 
\[
K(x,y,z)^{\langle\ta_3,\la_3,-I_3\rangle}=K(t_1,t_2,t_3)
\]
where 
\begin{align*}
t_1&:=u_1+(-I_3)(u_1)=\frac{-u_1+u_2+u_3+r\,u_1u_2u_3}{r\,u_2u_3}\\
&\ =\frac{r(x+y-z-w)+xy(z+w)-zw(x+y)}{r(xy-zw)},\\
t_2&:=u_2+(-I_3)(u_2),\quad t_3:=u_3+(-I_3)(u_3)\ \ \mathrm{with}\ \ w=c/(xyz).
\end{align*}
The actions of $\cb$ and $\be_3$ on 
$K(x,y,z)^{\langle\ta_3,\la_3,-I_3\rangle}=K(t_1,t_2,t_3)$ are given by
\begin{align*}
\cb &: t_1\ \mapsto\  t_2,\ t_2\ \mapsto\  t_3,\ t_3\ \mapsto\ t_1,\\
\be_3 &: t_1\ \mapsto\ \ep_2 t_1,\ t_2\ \mapsto\ \ep_2 t_3,\ t_3\ \mapsto\ \ep_2 t_2.
\end{align*}
By the same reason as in Subsection \ref{subse713}, 
$K(x,y,z)^{G_{7,2,3}}=K(t_1,t_2,t_3)^{\langle\cb\rangle}$ and 
$K(x,y,z)^{G_{7,5,3}}=K(t_1,t_2,t_3)^{\langle \cb,\be_3\rangle}$ are rational over $K$. 

\subsection{The case of $G_{7,3,3}$}
We treat the case of 
\[
G_{7,3,3}=\langle \ta_3,\la_3,\cb,-\be_3\rangle.
\]
The actions of $\cb$ and $-\be_3$ on 
$K(x,y,z)^{\langle \ta_3,\la_3\rangle}$ $=$ $K(u_1,u_2,u_3)$ are given by 
\begin{align*}
\cb &: u_1\ \mapsto\  u_2,\ u_2\ \mapsto\  u_3,\ u_3\ \mapsto\ u_1,\\
-\be_3 &: u_1\ \mapsto\ \frac{-u_1+u_2+u_3}{j\, u_2u_3},\ 
u_2\ \mapsto\ \frac{u_1+u_2-u_3}{j\, u_1u_2},\ 
u_3\ \mapsto\ \frac{u_1-u_2+u_3}{j\, u_1u_3}.
\end{align*}

The action of $G_{7,3,3}$ on $K(u_1,u_2,u_3)$ through 
$G_{7,3,3}/\langle\ta_3,\la_3\rangle=\langle \cb,-\be_3\rangle\cong\mathcal{S}_3$ coincides 
with the action as in Theorem \ref{thHK2}. 
By Theorem \ref{thHK2}, we conclude that $K(x,y,z)^{G_{7,3,3}}$ $=$ 
$K(u_1,u_2,u_3)^{\langle \cb,-\be_3\rangle}$ is rational over $K$. 

Note that $\langle\ta_3,\la_3,-\be_3\rangle=\langle\cbb,\la_3\rangle=G_{4,4,2}$ and 
we already showed that $K(x,y,z)^{G_{4,4,2}}=K(u_1,u_2,u_3)^{\langle-\be_3\rangle}$ 
is rational over $K$ as in Lemma \ref{lemuc2}. 

\begin{acknowledgments}
The authors would like to thank Professor Ming-chang Kang for giving them 
valuable and useful comments and for informing them about the recent paper \cite{KP10} 
by Kang and Prokhorov which gives affirmative solutions to the same rationality 
problem of $2$-groups over quadratically closed field. 
They also thank the referee for suggesting many improvements and corrections. 
\end{acknowledgments}



\begin{thebibliography}{BBNWZ78}
\bibitem[AHK00]{AHK00} H. Ahmad, M. Hajja, M. Kang, {\it Rationality of some projective 
linear actions}, J. Algebra \textbf{228} (2000) 643--658.
\bibitem[BBNWZ78]{BBNWZ78} H. Brown, R. B\"ulow, J. Neub\"user, H. Wondratschek, 
H. Zassenhaus. {\it Crystallographic Groups of Four-Dimensional Space}, 
John Wiley, New York, 1978. 
\bibitem[CS07]{CS07} J.-L. Colliot-Th\'el\`ene, J.-J. Sansuc, 
{\it The rationality problem for fields of invariants under linear algebraic groups 
(with special regards to the Brauer group)}, 
Algebraic groups and homogeneous spaces, 113--186, Tata Inst. Fund. Res. Stud. Math., 
Tata Inst. Fund. Res., Mumbai, 2007. 
\bibitem[EM73]{EM73} S. Endo, T. Miyata, \textit{Invariants of finite abelian 
groups}, J. Math. Soc. Japan \textbf{25} (1973) 7--26. 
\bibitem[GAP07]{GAP07} The GAP Group, GAP -- Groups, Algorithms, and Programming, 
Version 4.4.10; 2007. (http://www.gap-system.org) 
\bibitem[Haj83]{Haj83} M. Hajja, {\it A note on monomial automorphisms}, 
J. Algebra \textbf{85} (1983) 243--250. 
\bibitem[Haj87]{Haj87} M. Hajja, {\it Rationality of finite groups of monomial 
automorphisms of $k(x,y)$}, J. Algebra \textbf{109} (1987) 46--51. 
\bibitem[HK92]{HK92} M. Hajja, M. Kang, {\it Finite group actions on rational function 
fields}, J. Algebra \textbf{149} (1992) 139--154. 
\bibitem[HK94]{HK94} M. Hajja, M. Kang, {\it Three-dimensional purely monomial group 
actions}, J. Algebra \textbf{170} (1994) 805--860. 
\bibitem[HK97]{HK97} M. Hajja, M. Kang, {\it Twisted actions of symmetric groups}, 
J. Algebra \textbf{188} (1997) 626--647. 
\bibitem[HHR08]{HHR08} K. Hashimoto, A. Hoshi, Y. Rikuna, {\it Noether's problem 
and $\Bbb Q$-generic polynomials for the normalizer of the 8-cycle in $S\sb 8$ and 
its subgroups}, Math. Comp. \textbf{77} (2008) 1153--1183. 
\bibitem[HK10]{HK10} A. Hoshi, M. Kang, {\it Twisted symmetric group actions}, 
Pacific J. Math. \textbf{248} (2010) 285--304. 
\bibitem[HR08]{HR08} A. Hoshi, Y. Rikuna, {\it Rationality problem of three-dimensional 
purely monomial group actions: the last case}, Math. Comp. \textbf{77} (2008) 1823--1829. 
\bibitem[JLY02]{JLY02} C. Jensen, A. Ledet and N. Yui, 
\textit{Generic polynomials, constructive aspects of the inverse 
Galois problem}, Mathematical Sciences Research Institute Publications, 
Cambridge, 2002. 
\bibitem[Kan04]{Kan04} M. Kang, {\it Rationality problem of $\rm GL\sb 4$ group actions}, 
Adv. Math. \textbf{181} (2004) 321--352.
\bibitem[Kan05]{Kan05} M. Kang, {\it Some group actions on $K(x\sb 1,x\sb 2,x\sb 3)$}, 
Israel J. Math. \textbf{146} (2005) 77--92.
\bibitem[KP10]{KP10} M. Kang, Y. G. Prokhorov, {\it Rationality of three-dimensional 
quotients by monomial actions}, J. Algebra \textbf{324} (2010) 2166--2197. 
\bibitem[KZ]{KZ} M. Kang, J. Zhou, {\it The rationality problem for finite subgroups 
of $GL_4(\mathbb{Q})$}, preprint. arXiv:1006.1156v1
\bibitem[Kem96]{Kem96} G. Kemper, \textit{A constructive approach to 
Noether's problem}, Manuscripta Math. \textbf{90} (1996) 343--363.
\bibitem[Kit10]{Kit10} H. Kitayama, {\it Noether's problem for four- and five- dimensional 
linear actions}, J. Algebra \textbf{324} (2010) 591--597. 
\bibitem[KY09]{KY09} H. Kitayama, A. Yamasaki, {\it The rationality problem for 
four-dimensional linear actions}, J. Math. Kyoto Univ. \textbf{49} (2009) 359--380.
\bibitem[Kun55]{Kun55} H. Kuniyoshi, {\it On a problem of Chevalley}, 
Nagoya Math. J. \textbf{8} (1955) 65--67. 
\bibitem[Len74]{Len74} H. W. Lenstra, Jr. {\it Rational functions invariant under a finite 
abelian group}, Invent. Math. \textbf{25} (1974) 299--325.
\bibitem[Mas55]{Mas55} K. Masuda, {\it On a problem of Chevalley}, 
Nagoya Math. J. \textbf{8} (1955) 59--63.
\bibitem[Miy71]{Miy71} T. Miyata, {\it Invariants of certain groups I}, 
Nagoya Math. J. \textbf{41} (1971) 69--73.
\bibitem[Pla07]{Pla07} B. Plans, {\it Noether's problem for $\rm GL(2,3)$}, 
Manuscripta Math. \textbf{124} (2007) 481--487.
\bibitem[Pro10]{Pro10} Y. G. Prokhorov, {\it Fields of invariants of finite linear groups}, 
Cohomological and geometric approaches to rationality problems,  245--273, 
Progr. Math., 282, Birkh\"auser Boston, Boston, MA, 2010.
\bibitem[Rik]{Rik} Y. Rikuna, {\it The existence of a generic polynomial for $\mathrm{SL}(2,3)$ 
over $\mathbb{Q}$}, preprint. 
Available from \verb+http://www.mmm.muroran-it.ac.jp/~yuji/MuNT/2004/papers/04030602rikuna.pdf+
\bibitem[Sal82]{Sal82} D. J. Saltman, {\it Generic Galois extensions and problems 
in field theory}, Adv. in Math. \textbf{43} (1982) 250--283.
\bibitem[Sal00]{Sal00} D. J. Saltman, {\it A nonrational field, answering a question of Hajja}, 
Algebra and number theory, Lecture Notes in Pure and Appl. Math., 208, 263--271, 
Dekker, New York, 2000. 
\bibitem[Swa83]{Swa83} R. G. Swan, {\it Noether's problem in Galois theory}, 
Emmy Noether in Bryn Mawr (Bryn Mawr, Pa., 1982), 21--40, Springer, New York-Berlin, 1983. 
\bibitem[Vos73]{Vos73} V. E. Voskresenski\u\i, \textit{Fields of invariants 
of abelian groups}, (Russian) Uspehi Mat. Nauk \textbf{28} (1973) 77--102. 
English translation: Russian Math. Survey \textbf{28} (1973) 79--105. 
\bibitem[Yam10]{Yam10} A. Yamasaki, {\it Some cases of four dimensional linear 
Noether's problem}, J. Math. Soc. Japan \textbf{62} (2010) 1273--1288.
\bibitem[Yam]{Yam} A. Yamasaki, {\it Negative solutions to three-dimensional monomial 
Noether problem}, preprint. arXiv:0909.0586v2 [math.NT]
\end{thebibliography}
\end{document}